\documentclass[11pt]{amsart}

\usepackage{amsmath}
\usepackage{amsfonts}
\usepackage{amssymb}
\usepackage{amsthm}
\usepackage{comment}
\usepackage{graphicx}
\usepackage{verbatim}
\usepackage{pgf,tikz,pgfplots}
\usetikzlibrary{calc,angles,quotes}
\usetikzlibrary{hobby}
\usetikzlibrary{arrows}
\pgfplotsset{compat=1.15}
\usepackage{mathrsfs}
\usepackage{epsfig}
\usepackage{xspace}
\usepackage{hyperref,url}
\hypersetup{
    colorlinks=true,
    linkcolor=blue,
    filecolor=magenta,      
    urlcolor=cyan,
    citecolor=red
}
\usepackage{cite}
\usepackage{xstring}
\usepackage{intcalc}
\pgfplotsset{compat=1.15}
\usepackage{mathrsfs}
\usepackage{caption}
\usepackage{enumerate}
\usepackage{overpic}
\usepackage[all,cmtip]{xy}
\usepackage{subcaption}

\usepackage{todonotes}

\usepackage[utf8]{inputenc}

\usepackage{float}

\newcommand{\cal}{\mathcal}
\newcommand{\R}{\mathbb{R}}

\newcommand{\Z}{\mathbb{Z}}
\newcommand{\N}{\mathbb{N}}
\renewcommand{\H}{\mathbb{H}}

\newcommand{\ep}{
\epsilon
}
\newcommand{\mc}[1]{
\mathcal{#1}
}
\newcommand{\mb}[1]{
\mathbb{#1}
}

\newcommand{\GL}{\mathrm{GL}}
\newcommand{\SL}{\mathrm{SL}}

\newcommand{\SO}{\mathrm{SO}}

\usepackage{accents}
\newcommand{\interior}[1]{\accentset{\smash{\raisebox{-0.12ex}{$\scriptstyle\circ$}}}{#1}\rule{0pt}{2.3ex}}
\renewcommand{\tilde}{\accentset{\sim}}

\newcommand{\dimd}{d}  
\newcommand{\sph}{\mathbb{S}}   
\newcommand{\corner}{\mathfrak{C}} 

\newcommand{\blow}{\mathbf{B}}

\usetikzlibrary{cd}

\DeclareMathOperator{\conv}{Conv}

\DeclareMathOperator{\Ker}{Ker}
\DeclareMathOperator{\Span}{Span}

\DeclareMathOperator{\Stab}{Stab}

\DeclareMathOperator{\dev}{dev}
\DeclareMathOperator{\hol}{hol}

\DeclareMathOperator{\tr}{tr}
\DeclareMathOperator{\Aut}{Aut}

\newcommand{\Id}{\mb{I}}
\renewcommand{\Im}{\mathrm{Im}}

\newcommand{\ie}{i.e.\ }
\newcommand{\eg}{e.g.\ }
\newcommand{\resp}{resp.\ }

\newtheorem{prop}{Proposition}[section]
\newtheorem{thm}[prop]{Theorem}

\newtheorem{q}[prop]{Question}
\newtheorem{lemma}[prop]{Lemma}
\newtheorem{fact}[prop]{Fact}
\newtheorem{cor}[prop]{Corollary}

\theoremstyle{definition}
\newtheorem{defi}[prop]{Definition}

\theoremstyle{remark}
\newtheorem{rk}[prop]{Remark}

\numberwithin{equation}{section}

\newtheorem*{claim*}{Claim}
\newtheorem{claiminner}{Claim}  
\newenvironment{claim}[1]{%
\renewcommand\theclaiminner{#1}%
\claiminner
}{\endclaiminner}

\begin{document}

\title{Divisible convex sets with properly embedded cones}

\author{Pierre-Louis Blayac}
\address{University of Michigan, Ann Arbor, USA}
\email{blayac@umich.edu}

\author{Gabriele Viaggi}
\address{Department of Mathematics, Heidelberg University, Heidelberg, Germany}
\email{gviaggi@mathi.uni-heidelberg.de}

\begin{abstract}
In this article we construct many examples of properly convex irreducible domains divided by Zariski dense relatively hyperbolic groups in every dimension at least 3. This answers a question of Benoist. Relative hyperbolicity and non-strict convexity are captured by a family of properly embedded cones (convex hulls of points and ellipsoids) in the domain. Our construction is most flexible in dimension 3 where we give a purely topological criterion for the existence of a large deformation space of geometrically controlled convex projective structures with totally geodesic boundary on a compact 3-manifold.         
\end{abstract}

\maketitle

\setcounter{tocdepth}{1}
\tableofcontents

\section{Introduction}

In this article we construct properly convex open subsets $\Omega\subset\mb{RP}^d$ divided by subgroups $\Gamma<{\rm PSL}_{d+1}\R$ with special geometric properties in arbitrary dimension $d\ge 3$. 

A \emph{properly convex} subset $\Omega\subset\mb{RP}^d$ is an open subset contained in some affine chart of $\mb{R}^d$ where it is convex and bounded. It comes with a group of projective symmetries
\[
{\rm Aut}(\Omega):=\{A\in{\rm PSL}_{d+1}(\mb{R})\left|\; A(\Omega)=\Omega\right.\}
\]
and an ${\rm Aut}(\Omega)$-invariant Finsler metric called the \emph{Hilbert metric}.  

Discrete and torsion-free subgroups $\Gamma<{\rm Aut}(\Omega)$ give rise to quotient projective $d$-manifolds
\[
M:=\Omega/\Gamma.
\]
Among these groups $\Gamma$, of special interest are those such that $M=\Omega/\Gamma$ is compact. If such discrete subgroup exists, then the convex set $\Omega$ is said {\em divisible} and the group $\Gamma$ {\em divides} $\Omega$. Divisible convex sets are a rich source of geometry, dynamics, and group theory (see \cite{Besurvey}).

Basic examples are
\begin{itemize}
\item{The {\em hyperbolic space}, the symmetric space of ${\rm PO}_0(d,1)$
\[
\mb{H}^d=\{[x]\in\mb{RP}^d\left|\;x_1^2+\cdots+x_d^2-x_{d+1}^2<0\right.\}.
\]
}
\item{The projective model of the {symmetric space} of $\SL_n\R$ 
\[
\SL_n\R/\SO(n)={\rm PSym}^+_n=\{[S]\in\mb{P}({\rm Sym}_n(\mb{R}))\left|S \text{ is positive definite}\right.\},
\]
where ${\rm Sym}_n(\mb{R})$ is the space of symmetric square matrices of size $n$.
} 
\item{A projective model of the symmetric space $\mb{R}^d$
\[
\Delta^d=\left\{[x]\in\mb{RP}^d\left| x_1,\dots,x_{d+1}>0\right.\right\}.
\]
}
\end{itemize}
(The Hilbert metrics do not necessarily agree with the Riemannian ones.)

The study of divisible convex sets, initiated by Benzécri \cite{Ben60} and vastly expanded by Benoist \cite{Be1,Be2,B3,Be4}, branches into two major subjects:
\begin{itemize}
\item{{\em Existence}: Construct examples beyond symmetric ones (see Table~\ref{tab:table1}).}
\item{{\em Classification}: Describe the structure of divisible convex sets, of the groups dividing them, and of the quotient projective manifolds.} 

\end{itemize}

We are going to come back to this in Section~\ref{sec:classification and examples}.

The purpose of this article is to contribute to the first goal by exhibiting new examples of divisible convex sets with novel geometric and group theoretic features. 
First we answer a question of Benoist \cite[Prob.\,10]{Bexercises} (see also Marquis \cite[Open\,Q.\,3]{Masurvey}).

\begin{thm}
\label{thm:mainA}
For every $d\ge 3$ there exists a divisible convex set $\Omega\subset\mb{RP}^d$ divided by a {\rm Zariski dense} group $\Gamma<{\rm SL}_{d+1}(\mb{R})$ such that $\Omega$ is {\rm not strictly convex}, i.e. $\partial\Omega$ contains non-trivial line segments.
\end{thm}

We will give a more precise version of Theorem \ref{thm:mainA} later on (see Theorem~\ref{thm:main1}),
which gives a description similar to a result of Benoist \cite[Th.\,1.1]{Be4} on the structure of divisible convex sets in dimension 3 (see also \cite{IZ19relhypb,IZ22relhypb,We21}).
Let us remark the following. 
\begin{itemize}
\item{Recall that a subgroup $\Gamma<{\rm PSL}_{d+1}(\mb{R})$ is {\em Zariski dense} if every polynomial in the matrix entries that vanishes on $\Gamma$ also vanishes on ${\rm PSL}_{d+1}(\mb{R})$.}
\item{Non-strict convexity comes from \emph{properly embedded cones} in $\Omega$, \ie convex hulls in $\Omega$ of a $(d-2)$-dimensional ellipsoid $H\subset\partial\Omega$ (the {\em base} of the cone) and a point $p\in\partial\Omega$ (the {\em vertex}), whose relative boundary is contained in $\partial\Omega$.}
In dimension at least $4$, the presence of these cones is also a novelty.
\item{Our construction allows a good amount of flexibility.
For instance, if $d\geq 5$, then we can produce infinitely many different {\em commensurability classes} of projective manifolds with the properties stated in the theorem. 
Recall that two groups $\Gamma,\Gamma'$ are {\em commensurable} if they admit isomorphic finite index subgroups $G<\Gamma,G'<\Gamma'$.}

In particular, for each $d\geq5$ there exist infinitely many different non-strictly convex divisible convex sets divided by a Zariski dense group.
\item{$\Omega$ will be constructed by gluing compact convex projective manifolds \emph{with totally geodesic boundary}, like the ones in the next result.}
\end{itemize}

Our second result is specific to dimension 3,
where geometry and topology are strictly tied by breakthroughs of Thurston \cite{Th82}. 
We provide a purely topological condition for the existence of (deformation spaces of) convex projective structures with totally geodesic boundary on a large class of compact 3-manifolds with toroidal boundary, addressing a question of Ballas, Danciger, and Lee \cite{BDL18} (see Question~\ref{qn:question2}).

\begin{thm}
\label{thm:mainB}
Consider: 
\begin{itemize}
 \item A compact, orientable, \emph{irreducible} and \emph{atoroidal} 3-manifold $M$  with non-empty connected boundary.
 \item A \emph{doubly incompressible}  separating simple closed curve $\alpha\subset\partial M$.
\end{itemize}
Then for all
$a>1$ close enough to $1$, $b>3$ and $1<c<-2+b^2/2$,
the complement $N=DM\smallsetminus U$ of an open tubular neighborhood $U$ of $\alpha$ in the double $DM$ admits a convex projective structure with totally geodesic boundary isomorphic to a projective torus with holonomy generated by $A={\rm diag}(a,a^{-1},1,1)$ and $B={\rm diag}(b^{-1},b^{-1},bc,bc^{-1})$. 
\end{thm}

We will later state Theorem~\ref{thm:main3}, a more precise version  of Theorem \ref{thm:mainB}, which allows more boundary components in $\partial M$.
We will then recall classical results on the topology of 3-manifolds.

\begin{itemize}
 \item Recall that $M$ irreducible (\resp atoroidal) means that it does not contain essential spheres (\resp tori),
 and $\alpha$ doubly incompressible means that it intersects every properly embedded essential disk, M\"obius band, and annulus.
 \item{Curves $\alpha\subset\partial M$ that satisfy the assumptions of Theorem \ref{thm:mainB} exist and are abundant in a precise sense (by work of Lecuire \cite{Le05}).}
 \item{Given $(M,\alpha), (M',\alpha')$ as in Theorem \ref{thm:mainB}, the control on the boundary holonomy gives us a homeomorphism $f:\partial N\to \partial N'$ such that the closed manifold $N\cup_f N'$ can be endowed with a convex projective structure.} 
 {In particular, the double $DN$ has a convex projective structure.}
\end{itemize}

The remainder of the introduction is structured as follows:
\begin{enumerate}[(a)]
\item{In Section \ref{sec:classification and examples} we put our work into context by briefly surveying the classification of divisible convex sets and the known methods to produce examples.}
\item{In Section \ref{sec:introcvxcocpct} we give and comment a more precise statement (Theorem \ref{thm:main1}) of Theorem \ref{thm:mainA}. 
In particular we give more details on properly embedded cones and their stabilizers. 
Theorem \ref{thm:main1} will be deduced from Theorem \ref{thm:main2}, stated in Section \ref{sec:introcvxcocpct}, which gives compact convex projective manifolds with totally geodesic boundary (which are convex-cocompact in the sense of \cite{DGK17}).}
\item{In Section \ref{sec:dimension 3} we give a more precise version (Theorem \ref{thm:main3}) of Theorem~\ref{thm:mainB}, and we put this result into context by briefly surveying classical result in 3-dimensional topology.}
\end{enumerate}

\subsection{Classification and known examples}\label{sec:classification and examples}
Before describing our new examples and their novel features, let us briefly survey the classes of examples known so far in order to provide an adequate context. 

To this purpose, it is convenient to briefly recall the general classification scheme (see Table \ref{tab:table1} and Figure \ref{fig:figure1}) of divisible convex sets.

\subsubsection*{Classification}

\begin{table}[h]
\centering
\begin{tabular}{c | c | c | c}
\hline
\multicolumn{2}{c |}{} &\multicolumn{2}{| c}{}\\
\multicolumn{2}{c |}{{\slshape symmetric}} &\multicolumn{2}{| c}{{\slshape non-symmetric}}\\
\multicolumn{2}{c |}{${\rm Aut}(\Omega)$ semisimple Lie group} &\multicolumn{2}{| c}{${\rm Aut}(\Omega)<{\rm PSL}_{d+1}(\R)$ Zariski dense}\\
\multicolumn{2}{c |}{${\rm Aut}(\Omega)\curvearrowright\Omega$ transitive} &\multicolumn{2}{| c}{${\rm Aut}(\Omega)\curvearrowright\Omega$ properly discontinuous}\\
\multicolumn{2}{c |}{\cite{Ko99,Vin63,Vin65}} &\multicolumn{2}{| c}{\cite{Be2}}\\
\multicolumn{2}{c |}{} &\multicolumn{2}{| c}{}\\
$\mb{R}$-{\slshape rank} $\ge 2$ &$\mb{R}$-{\slshape rank} $=1$ &{\slshape strictly convex} &{\slshape non-strictly convex}\\
${\rm SL}_n(\mb{R})/{\rm SO}(n)$ &$\mb{H}^d$  &$\Omega$ &$\Omega$\\
\cite{Bo63} &\cite{Si51} &\cite{KV67,CG93,JM80,K07} &\cite{Be4,CLM20,BDL18,Ma10}\\
 & & &\\
\hline
\end{tabular}
\caption{Classification of irreducible divisible convex sets $\Omega$. The last row surveys some of the known examples.} 
\label{tab:table1}
\end{table}

\begin{figure}[h]
\centering
\begin{overpic}{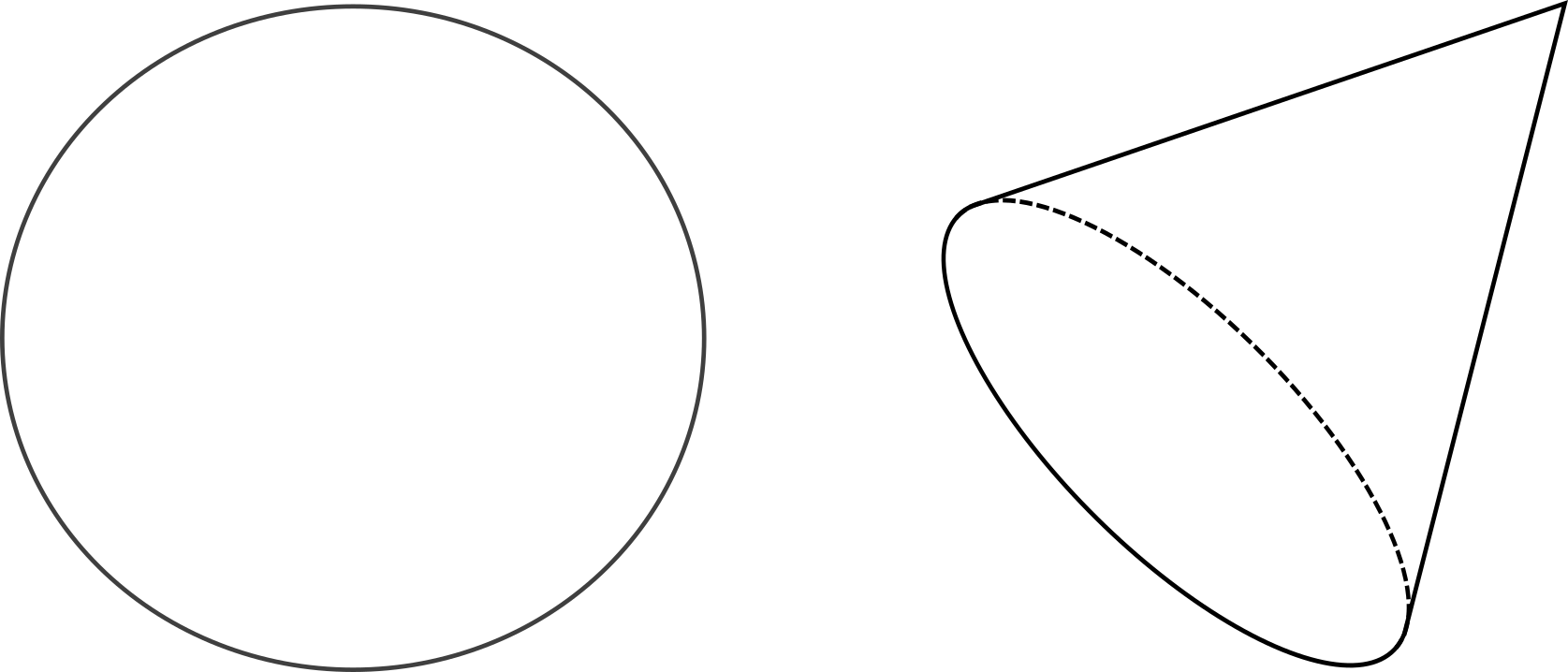}
\put(0,40) {$\mb{H}^d$}
\put(55,40) {$\mc{CH}(\mb{H}^{d-1},\star)$}
\end{overpic}
\caption{Irreducible and reducible examples.}
\label{fig:figure1}
\end{figure}

Let us comment on some features in Table \ref{tab:table1}:

{{\slshape Reducible and irreducible}.
If $\Gamma_j$ divides $\Omega_j\subset\mb{RP}^{d_j}$ for $j=1,2$, then $\Gamma_1\times\Gamma_2\times\mb{Z}$ divides the convex hull of the two 
\[
\mc{CH}(\Omega_1,\Omega_2)\subset\mb{RP}^d
\]
with $d=d_1+d_2+1$ and where $\mb{RP}^d=\mb{P}(\mb{R}^{d_1+1}\oplus\mb{R}^{d_2+1})$. 
Such examples are called {\em reducible}. 
If a divisible convex domain is not of this form, it is {\em irreducible}.
Vey proved \cite{Vey70} that if $\Gamma$ divides $\Omega$ then $\Omega$ is irreducible if and only if $\Gamma$ is \emph{strongly irreducible} (does not preserve any finite union of subspaces of $\R^{\dimd+1}$).
As we saw in Theorem~\ref{thm:mainA}, reducible divisible domains come up naturally in our constructions in the form of {\em cones}, that is, convex hulls of ellipsoids and points.}

{{\slshape Symmetric and non-symmetric}.
{As hinted in Table~\ref{tab:table1}, the dichotomy between the symmetric and non-symmetric cases is the result of the combination of the works 
\cite{Ko99,Vin63,Vin65,Be2}.}

{{\slshape Regularity of $\partial\Omega$ and geometry}. Strict convexity can be seen as an avatar of negative curvature: Benoist proves in \cite[Th.\,1.1]{Be1} that if $\Gamma$ divides $\Omega$ then the following are equivalent: 
\begin{itemize}
\item{$\Omega$ is strictly convex.}
\item{$\Gamma$ is {\em Gromov hyperbolic}.}
\item{$\partial\Omega$ is $\mc{C}^1$.}
\end{itemize}
Recall that Gromov hyperbolicity is a group theoretic abstraction of the geometric properties of fundamental groups of closed hyperbolic manifolds and free groups (see Definition~\ref{def:yaman} and Remark~\ref{rk:yaman}).

In the same direction, Islam \cite[Th.\,1.6]{I19} showed that a larger class of divisible convex sets possesses a weak form of negative curvature: If $\Gamma$ divides a non-symmetric irreducible $\Omega$, then $\Gamma$ is {\em acylindrically hyperbolic} (in the sense of Osin \cite{Osi16}).}

{{\slshape Dimensions $2$ and $3$.} As a consequence of Benz\'ecri's work \cite{Ben60}, every divisible convex set of dimension $2$ is either a triangle (hence reducible) or strictly convex.
Benoist gave in \cite{Be4} a beautiful geometric description of \emph{all} irreducible non-symmetric non-strictly divisible convex domains of dimension $3$, similar to that in Theorem~\ref{thm:main1}.
He also gave examples, as mentioned later on.}

{{\slshape Geometric rank-one and Higher rank rigidity}. As in the non-positively curved Riemannian setting there is a notion of {\em geometric rank-one} of a convex projective manifold $M=\Omega/\Gamma$ introduced by Islam \cite{I19}. Zimmer proved in \cite{Z22} a higher rank rigidity result for projective manifolds (analogous to Ballmann \cite{Ba85} and Burns and Spatzier \cite{BS87}): If $\Omega$ is irreducible and not symmetric then it has geometric rank-one.}

\subsubsection*{Examples}
We are now ready to describe the classes of examples known so far according to the methods used to produce them.

{\slshape Polyhedral tilings in dimension $d\le 6$ and local-to-global convexity}: One way of constructing groups dividing properly convex domains is to consider {\em Coxeter groups} $\Gamma$ generated by the reflections along the codimension 1 faces of a polyhedron $P\subset\mb{RP}^d$. 
Under suitable combinatorial and geometric assumptions on $P$, a local-to-global convexity argument asserts that the domain $\Omega=\Gamma\cdot P$ is properly convex and tiled by copies of $P$. 
This method, that goes back to Poincaré, has been extensively used by Kac and Vinberg \cite{KV67}, Benoist \cite{Be4}, and Marquis \cite{Ma10} to produce irreducible divisible domains  in dimension $d\le 6$ which are non-symmetric.

{\slshape Arithmetic methods}. A classical theorem of Borel \cite{Bo63} shows that every semisimple Lie group $G$ admits a so-called {\em uniform arithmetic lattice} $\Gamma$, which is a discrete subgroup $\Gamma<G$ acting cocompactly on the symmetric space $G/K$. This implies that every irreducible symmetric properly convex domain is divisible. In the case of the hyperbolic space $\mb{H}^d$, the first examples are due to Siegel \cite{Si51}.

{\slshape Bending and Bulging}. A subgroup $\Gamma=\Gamma_1*_\Lambda\Gamma_2<{\rm PSL}_{d+1}(\mb{R})$ can be algebraically deformed by performing a very general procedure called \emph{bending}: 
Let $(B_t)_t\subset{\rm PSL}_{d+1}(\mb{R})$ be a path of elements starting at the identity, that commute with $\Lambda$.
One can deform the inclusion of $\iota:\Gamma\to{\rm PSL}_{d+1}(\mb{R})$ to representations $\rho_t$ which are the identity on $\Gamma_1$ and send $\gamma_2\in\Gamma_2$ to $B_t\gamma_2B_t^{-1}$.  
By work of Koszul \cite{Kos68} and Benoist \cite{B3}, if $\Gamma$ divides some $\Omega$, then the representations $\rho_t$ are injective and $\rho_t(\Gamma)$ divides a properly convex domain $\Omega_t$. 
Furthermore, by the Ehresmann--Thurston principle \cite[Ch.\,3]{ThNotes}, the quotients $\Omega_t/\rho_t(\Gamma)$ are diffeomorphic to $\Omega/\Gamma$. 

The construction applies in particular to hyperbolic $d$-manifolds $M=\mb{H}^d/\Gamma$ containing a codimension 1 submanifold $\Sigma=\mb{H}^{d-1}/\Lambda$ separating $M$ into two connected components $M\smallsetminus\Sigma=M_1\sqcup M_2$ (such manifolds can be constructed using the arithmetic techniques mentioned above). 
By Seifert--van Kampen Theorem, we can write $\Gamma=\Gamma_1*_\Lambda\Gamma_2$ where $\Gamma_j=\pi_1(M_j)$.
The bending obtained using elements $B$ that fix $\H^{\dimd-1}$ and its orthogonal for the Lorentzian form is called \emph{bulging}. 
These deformations have been extensively studied by Johnson and Millson \cite{JM80}. They allow to produce irreducible non-symmetric strictly divisible convex domains in arbitrary dimension (see \cite[\S8.2]{Besurvey}). 

Note that, by Margulis' Superrigidity (see \cite[Ch.\,16]{WitteMorris}) divisible symmetric properly convex domains with $\mb{R}$-rank at least $2$ cannot be deformed. 

Using a bulging construction, Kapovich \cite{K07} showed that certain manifolds constructed by Gromov and Thurston \cite{GT87} admit convex projective structures. These manifolds have the property that they admit metrics of pinched negative curvature, but no purely hyperbolic one (\ie with constant curvature).
In particular, the examples are strictly convex and not deformations of symmetric ones.

Other works that used a bulging construction in a convex projective context include \cite{Ma12,Bob19,BM20}.

{\slshape Surfaces}. Generalizing classical Fenchel--Nielsen coordinates for hyperbolic surfaces, Choi and Goldman \cite{CG93} developed a gluing construction of convex projective surfaces starting from projective pairs of pants with totally geodesic boundary. They showed that their construction parametrizes all {\em Hitchin representations} $\rho:\pi_1(S_g)\to{\rm PSL}_3(\mb{R})$, which are continuous deformations of the holonomy $\rho_0:\pi_1(S_g)\to{\rm PSO}_0(2,1)$ of a hyperbolic structure on a closed orientable surface $S_g$ of genus $g\ge 2$.

{\slshape 3-Manifolds and Dehn fillings}. 
Thanks to Thurston's breakthroughs, in dimension 3 we have several techniques to produce closed hyperbolic 3-manifolds. 
Among those there is the celebrated Hyperbolic Dehn Filling Theorem (see \cite[Ch.\,4]{ThNotes}), which allows to deform a non-compact complete finite volume hyperbolic 3-manifold into many closed hyperbolic ones. 
Inspired by ideas from Thurston's theory, Ballas, Danciger, and Lee \cite{BDL18} gave many examples of convex projective 3-manifolds with a non-trivial JSJ decomposition. 
In a different direction, in combination with some Coxeter theory, Choi, Lee, and Marquis \cite{CLM20} produced examples of irreducible non-symmetric domains in dimension $d\le 6$ divided by groups which are relatively hyperbolic with respect to a collection of $\mb{Z}^2$ subgroups. 
(See also the work of Lee, Marquis, and Riolo \cite{LMR22}.)

Note that there is actually a lot of interactions between these methods.
Likewise, our construction owes to these previous works.
Indeed, to prove Theorem~\ref{thm:mainA}, we will:
\begin{itemize}
 \item Construct via arithmetic methods a particular hyperbolic manifold with codimension 2 cone-singularities, which is the double of a hyperbolic manifold with totally geodesic boundary and corners.
 \item Deform it into a \emph{projective} cone-manifold $M$ via bulging along totally geodesic hypersurfaces adjacent to the singularities.
 \item Use ideas from polyhedral tilings to show that the complement of the singularities in $M$ is properly convex.
 \item Blow up the singularities of $M$ to totally geodesic boundaries.
\end{itemize}

The closed convex projective examples of Theorem~\ref{thm:mainA} are actually doubles $DM$ of the above construction.

Let us now state refined versions of our main theorems.

\subsection{Divisible convex sets with properly embedded cones}
\label{sec:divisibles with cones}

\begin{thm}
\label{thm:main1}
For every $d\ge 3$ there exists a divisible convex set $\Omega\subset\mb{RP}^d$ divided by a discrete and {Zariski dense} subgroup $\Gamma<{\rm PSL}_{d+1}(\mb{R})$ with the following properties: 
\begin{enumerate}
\item{$\Omega$ contains a $\Gamma$-invariant family $\mc{C}$ of properly embedded cones with pairwise disjoint closures such that $\mc{C}/\Gamma$ is finite.}
\item{The stabilizer $\Gamma_C$ of each cone $C\in\mc{C}$ acts cocompactly on it, and has the form $\mb{Z}\times\Theta_C$, where $\Theta_C$ acts properly cocompactly on the base $H_C$ of $C$ while the $\mb{Z}$ factor acts trivially on it, so that $C/\Gamma_C$ is diffeomorphic to $H_C/\Theta_C\times\mb{S}^1$.}
\item \label{item:mainA relhypb} {The group $\Gamma$ is {\rm relatively hyperbolic} with respect to $\{\Gamma_C\}_{C\in\mc{C}}$, and its {\rm Bowditch boundary} naturally identifies with $\partial\Omega/_\sim$, where $p\sim q$ if $p,q\in\bar C$ for some $C\in\cal C$.}
\item \label{item:mainA regularity} {Every point of $\partial\Omega\smallsetminus\bigsqcup_{C\in\mc{C}}{{\bar C}}$ is extremal and $\mc{C}^1$.}
\end{enumerate}
\end{thm}

Let us remark the following. 
\begin{itemize}
\item{{\em Relative hyperbolicity} is a group theoretic abstraction of geometric properties of fundamental groups of finite volume hyperbolic manifolds. It comes with a notion of \emph{Bowditch boundary} (see Definition~\ref{def:yaman}).}
\item{Note that $\Gamma_C=\mb{Z}\times\Theta_C$ is abelian only for $d=3$, since $\Theta_C$ is conjugate to a uniform lattice of ${\rm SO}(d-2,1)$.
The existence of $\Omega$ (non-strictly convex) divided by a group which is hyperbolic relative to non-abelian subgroups is a novelty.

The answer to the following is still unknown:
\begin{q}
\label{qn:question1}
Is there for every $d>6$ an irreducible divisible convex set $\Omega\subset\mb{RP}^d$ which is non-symmetric and non-strictly convex and such that $\Gamma$ is relatively hyperbolic with respect to a collection of {\em abelian} subgroups?  
\end{q}
}
\item{Given the description of the structure of the properly convex set $\Omega$ and its properly embedded cones, parts \eqref{item:mainA relhypb} and \eqref{item:mainA regularity} of the above result can be deduced from \cite[Th.\,1.16]{We21} or \cite[Th.\,1.3-6]{IZ22relhypb}.
We will give a similar but different proof that applies to a different set of examples (see Theorem~\ref{thm:main2}).}
\item{Each component of $\left( \Omega\smallsetminus \bigcup_{C\in\cal C}C \right)/\Gamma$ admits an incomplete hyperbolic metric whose completion is a hyperbolic cone-manifold.}
\end{itemize}

Let us now state formally the construction of the manifold whose double is the non-strictly convex closed convex projective manifold $\Omega/\Gamma$ in Theorem~\ref{thm:main1}.

\subsection{Convex-cocompact manifolds with totally geodesic boundary}\label{sec:introcvxcocpct}

The notion of convex-cocompactness in projective spaces has been introduced by Danciger, Gu{\'e}ritaud, and Kassel \cite{DGK17}. It is inspired by (and generalizes) the corresponding definition of convex-cocompactness for Kleinian groups (see for example \cite[Ch\,3]{C21}) and it is linked to the concept of Anosov subgroups of higher rank Lie groups (see for example \cite[Ch.\,7]{C21}). 

Let $\Omega\subset\mb{RP}^d$ be open and properly convex.
Let $\Gamma<{\rm Aut}(\Omega)$ be a discrete subgroup. Consider the {\em full orbital limit set} $\Lambda\subset\partial\Omega$ consisting of the accumulation points on $\partial\Gamma$ of every orbit $\Gamma\cdot o$ with $o\in\Omega$ (\ie $\Lambda = \cup_{o\in\Omega}\overline{\Gamma o}\cap\partial\Omega$) and let $\mc{CH}(\Lambda)$ be the convex hull in $\Omega$ of $\Lambda$. We say that $\Gamma$ acts {\em convex-cocompactly} on $\Omega$ if $\mc{CH}(\Lambda)/\Gamma$ (the \emph{convex core}) is compact.

Compared to divisibility, convex-cocompactness is far more relaxed and flexible, and examples are abundant (see for example \cite{DGK17,DGKLM}). Their study is an area of very active research (see for example \cite{IZ19relhypb,We21,mesureBM,EeERfH+,IZ22relhypb}).

As mentioned, we will construct the closed convex projective manifolds of Theorem~\ref{thm:mainA} by gluing compact convex projective manifolds with totally geodesic boundary given by the following.

\begin{thm}
\label{thm:main2}
For every $d\ge 3$ there exist properly convex open sets $\Omega_1,\Omega_2\subset\mb{RP}^d$ and discrete groups $\Gamma_1,\Gamma_2<{\rm PSL}_{d+1}(\mb{R})$, preserving respectively $\Omega_1,\Omega_2$, with the following properties.
Let $\Lambda_j\subset\partial\Omega_j$ be the full orbital limit set with convex hull $\mc{CH}_j$ in $\Omega_j$.
Then
\begin{enumerate}
\item{$\Gamma_j$ acts properly and cocompactly on $\overline{\mc{CH}_j}\smallsetminus\Lambda_j$.}
\item{The connected components of $\partial\mc{CH}_j\smallsetminus\Lambda_j=\bigsqcup_{C\in\mc{C}_j}{C}$
form a collection $\mc{C}_j$ of cones (convex hulls of ellipsoids and points) with pairwise disjoint closures.}
\item \label{item:main2 relhypb} {The group $\Gamma_j$ is relatively hyperbolic with respect to the family $\{\Gamma_C\}_{C\in\mc{C}_j}$ of stabilizers of cones $C\in\mc{C}_j$, 
with Bowditch boundary $\partial\mc{CH}_j/_\sim$, where $p\sim q$ if $p,q\in\bar C$ for some $C\in\cal C_j$.}
\item \label{item:main2 regularity} {Every point of $\Lambda_j\smallsetminus\bigsqcup_{C\in\mc{C}_j}{{\bar C}}$ is an extremal and $\cal C^1$ point of $\partial\Omega_j$.}
\item{$\Gamma_1$ acts {convex-cocompactly} on $\Omega_1$ and $\mc{C}_1$ is a family of properly embedded cones.}
\item{$\Gamma_2$ {\rm does not} act convex-cocompactly on any properly convex open set, and $\mc{C}_2$ is the family of codimension 1 faces of $\partial\Omega_2$ (and $\Omega_2=\mc{CH}_2$).}
\end{enumerate}
\end{thm}

Some remarks:
\begin{itemize}
\item{We have $\Gamma_C=\mb{Z}\times\Theta_C$, where $\Theta_C$ acts properly cocompactly on the base of the cone while the $\mb{Z}$ factor acts trivially on it.}
\item{The quotients $M_j=(\overline{\mc{CH}_j}\smallsetminus\Lambda_j)/\Gamma_j$ are compact convex projective manifolds with totally geodesic boundary
\[
\partial M_j=(\partial\mc{CH}_j\smallsetminus\Lambda_j)/\Gamma_j=\bigsqcup_{[C]\in\mc{C}_j/\Gamma_j}{C/\Gamma_C}.
\]
}
\item{The double $DM_1$ is a closed convex projective manifold. Its universal cover $\Omega_1\subset\mb{RP}^d$ is the properly divisible convex set of Theorem \ref{thm:mainA}.}
\item{As in Theorem~\ref{thm:main1}, one can use work of Weisman \cite{We21} and Islam and Zimmer \cite{IZ22relhypb} to simplify our proof of \eqref{item:main2 relhypb} and  \eqref{item:main2 regularity} \emph{in the convex-cocompact case ($j=1$)}.}
\item{In the non-convex-cocompact case ($j=2$): If $d=3$ then ${\rm int}(M_2)=\Omega_2/\Gamma_2$ is a 3-manifold with {\em generalized cusps} of {\em type 2} in the sense of \cite{BCL20}. If $d>3$, then ${\rm int}(M_2)=\Omega_2/\Gamma_2$ is not a manifold with generalized cusps (in the sense of \cite{BCL20}), although it is the interior of a compact $d$-manifold with boundary. The fundamental group of the boundary components have similar features to generalized cusp groups of type $d-1$ but are not solvable.}
\item{The inclusions $\Gamma_j<{\rm PSL}_{d+1}(\mb{R})$ are {\em extended geometrically finite} on the (partial) flag variety of points and hyperplanes in the sense of Weisman \cite[Def.\,1.3]{We22}. 
For $j=1$ this is due to Weisman \cite[Th.\,1.12]{We22}, for $j=2$ this is proved in Section~\ref{sec:EGF}.}
\end{itemize}

\subsection{Geometrization in dimension 3}\label{sec:dimension 3}

Theorem~\ref{thm:mainB} is a geometrization theorem that turns topological information into geometric structure.
Before we state a refined version of it, let us describe a few structural results from 3-dimensional topology.

By Alexander's Theorem (see \cite[Th\,9.2.10]{Mar16}) every convex projective 3-manifold $M=\Omega/\Gamma$ is \emph{irreducible} that is, it does not contain essential spheres. From the work of Jaco, Shalen, and Johannson (see \cite[Th\,11.5.1]{Mar16}), $M$ can be split along an essentially canonical collection $S\subset M$ of tori and Klein bottles into pieces 
\[
M\smallsetminus S=M_1\sqcup\cdots\sqcup M_r
\]
that are either {\em atoroidal} (not containing essential tori and Klein bottles), or {\em Seifert fibered} (fibrations in circles over 2-dimensional orbifolds). The decomposition is called the {\em JSJ decomposition} of $M$.

Thanks to the major breakthough given by Thurston's Geometrization Conjecture \cite{Th82}, whose proof has been completed by Perelman, we know that every piece of the JSJ decomposition of $M$ admits a homogeneous Riemannian metric locally modeled on one of 8 geometries (see \cite[Ch.\,12]{Mar16}), among which hyperbolic geometry occupies a prominent role. 
By Thurston's Hyperbolization for Haken manifolds (see \cite[Th.\,1.43]{K01}), each atoroidal piece $N_j$ with $\partial N_j\neq\emptyset$ admits a complete finite volume hyperbolic metric.

Convex projective geometry fits well into this picture: Benoist proved \cite[Prop.\,3.2]{Be4} that the JSJ decomposition of $M$ only contains atoroidal pieces and it is realized geometrically, meaning that $S\subset M$ can be chosen to be {\em totally geodesic} (the lifts to $\Omega$ are the intersection of $\Omega$ with projective 2-planes). Therefore, each piece of $M\smallsetminus S$ is the interior of a convex projective manifold with totally geodesic boundary.

\begin{q}[Ballas, Danciger, and Lee \cite{BDL18}]
\label{qn:question2}
Suppose a closed orientable irreducible 3-manifold has a JSJ decomposition containing only atoroidal pieces. Does it also admit a convex projective structure?
\end{q}

In the positive direction, Ballas, Danciger, and Lee \cite{BDL18} provide an infinite family of examples. Our second contribution consists of a large flexible class of controlled deformation families of convex projective structures on 3-manifolds with totally geodesic boundary.

\begin{thm}
\label{thm:main3}
Let $M$ be a compact orientable irreducible atoroidal 3-manifold with non-empty boundary $\partial M=\Sigma_1\sqcup\dots\sqcup\Sigma_n$.
Let $\alpha:=\alpha_1\cup\dots\cup\alpha_n$ be a doubly incompressible multicurve where $\alpha_j$ is a separating simple closed curve of $\Sigma_j$.
Consider the manifold 
\[
N:=DM\smallsetminus U_1\cup\cdots\cup U_n
\]
obtained by removing from the double $DM$ of $M$ an open tubular neighborhood $U_1\cup\cdots\cup U_n$ of $\alpha_1\cup\dots\cup\alpha_n$.

Then there exists $\ep=\ep(M,\alpha)>0$ such that for every $(a_j,b_j,c_j)$ in
\[
\mc{P}:=\{(a,b,c)\in(1,1+\ep)\times(3,\infty)\times(1,\infty)\left|c\le -2+b^2/2\right.\}
\]
there exists a convex projective structure with totally geodesic boundary on $N$, such that the boundary $\partial U_j$ is isomorphic to $\Delta/\mb{Z}A_j\oplus\mb{Z}B_j$ where $A_j={\rm diag}(a_j,a_j^{-1},1,1),B_j={\rm diag}(b_j^{-1},b_j^{-1},b_jc_j,b_jc_j^{-1})$ and $\Delta=\Delta(e_1,e_2,e_3)\subset\{x_4=0\}$ is the standard simplex.  
\end{thm}

Double incompressibility was introduced by Thurston (see \cite{Th3}) in the study of deformations of geometrically finite hyperbolic structures on ${\rm int}(M)$ with prescribed cusps. 

Let us remark that: 
\begin{itemize}
\item{These convex projective structures on $N$ are convex-cocompact in the sense of \cite{DGK17} (See Section~\ref{sec:introcvxcocpct}).}
\item{The convex projective structures on $N$ vary continuously with the parameters $\{(a_j,b_j,c_j)\in\mc{P}\}_{j\le n}$.}
\item{Multicurves $\alpha\subset\partial M$ that satisfy the assumptions of Theorem \ref{thm:mainB} exist and are abundant in a precise sense (by work of Lecuire \cite{Le05}).}
\item{Given $(M,\alpha), (M',\alpha')$ as in Theorem \ref{thm:mainB} and a pairing $\{S_j\subset\partial N\leftrightarrow S_j'\subset\partial N'\}_{j\le r}$ of boundary components, the control on the boundary holonomy gives us homeomorphisms $f_j:S_j\to S_j'$ such that the gluing 
$N\cup_{f_1\sqcup\cdots\sqcup f_r} N'$,
can be endowed with a convex projective structure with totally geodesic boundary (which is empty if the pairing involves all connected components of $\partial N$ and $\partial N'$).} 
\item{In particular, the double $DN$ of $N$ is a closed orientable 3-manifold with a convex projective structure.}
\item{The manifold $N$ admits on its interior ${\rm int}(N)$ a complete finite volume hyperbolic metric which is itself a double of a complete finite volume hyperbolic metric with totally geodesic boundary on $M\smallsetminus\alpha$. In particular, it contains $\partial M\smallsetminus\alpha$ as a totally geodesic embedded subsurface.}
\item{The class that we construct is transverse to the one of Ballas, Danciger, and Lee \cite{BDL18}: Since the hyperbolic manifold $N$ contains a totally geodesic subsurface $\partial M\smallsetminus\alpha$ it is not {\em infinitesimally projectively rigid relative to the boundary} as they need in their work.} 
\item{$N$ is a projective manifold with generalized cusps of type $3$ in the sense of \cite{BCL20}.
Allowing $c_i=1$, one obtains projective manifolds with generalized cusps of type $2$ (the formulae for the holonomy would change though).
One should also be able to allow $a_i=1$ to obtain cusps of type $1$ and $0$ (the finite volume hyperbolic metric on $N$), but the above parametrization might not be well-suited for that.
In fact one should retrieve constructions in \cite{Ma12,Bob19,BM20}.}
\end{itemize}

\subsection{Hyperbolic manifolds with totally geodesic boundary and corners}

As mentioned at the end of Section~\ref{sec:classification and examples}, the building blocks for our constructions of the examples of all the previous theorems are convex compact hyperbolic manifolds with totally geodesic boundary and corners.
If $d=3$ then the manifolds we need are provided by work of Bonahon and Otal \cite{BO04}.
For the general case we construct arithmetic manifolds similar to the ones considered by Gromov and Thurston \cite{GT87} and Kapovich \cite{K07}. 
More precisely we prove:

\begin{thm}
\label{thm:main7}
For every $k\ge 3$ there exist:
\begin{enumerate}[(a)]
\item{A closed orientable hyperbolic $d$-manifold $M$ with closed connected totally geodesic hypersurfaces $N_1,\cdots,N_k\subset M$ intersecting along a connected totally geodesic $(d-2)$-submanifold $C(=N_i\cap N_j$ for $i\neq j)$ with angles $\angle_C N_jN_{j+1}=\pi/k$.
Moreover, $C$ is fixed by a cyclic isometry $\rho$ of $M$ with $\rho(N_j)=N_{j+1}$.}
\item{A compact convex hyperbolic $d$-manifold $M'$ with totally geodesic boundary and corners such that each corner has angle $\pi/k$ and the graph of the boundary is bipartite.}
\end{enumerate} 
\end{thm}

The vertices of the graph of the boundary are the totally geodesic pieces of the boundary, and its edges are the corners.
For our bulging construction to work, it is crucial that this graph is bipartite, so that we can ``alternate between huge and small bulging parameters'' (see Section~\ref{sec:proofs}).
Corner angles less than $\pi/4$ would allow one to work with a tripartite assumption instead.

The existence of (b) follows from (a) as the completion of every connected component of
$
M\smallsetminus(N_1\cup\cdots\cup N_k)
$
satisfies all the requirements of (b).

\subsection*{Outline}
The article is organized as follows.
\begin{itemize}
\item{Section \ref{sec:proofs}: {\em Ingredients of the proofs}. We discuss the ideas and ingredients of the proofs first in dimension 2 and then in general.}
\item{Section \ref{sec:sec1}: {\em Preliminaries}. We recall classical notions from convex projective geometry.}
\item{Section \ref{sec:sec2}: {\em Tubes, cone-manifolds, and totally geodesic blowup}. We define and classify tubes, define cone-manifolds, and describe the totally geodesic blowup of cone-manifolds whose singularities satisfy \eqref{eq:geomdyncond}.}
\item{Section \ref{sec:sec3}: {\em Tessellations of convex domains}. We expose the local-to-global convexity argument that guarantees that a given collection of convex sets tiles a convex domain $\Omega$, and then describe the regularity of $\partial\Omega$.}
\item{Section \ref{sec:sec4}: {\em Convex-cocompactness and relative hyperbolicity}. We work with compact convex projective manifolds $M$ with totally geodesic boundary. We give sufficient conditions for $M$ to be convex-cocompact, for $\pi_1(M)$ to be hyperbolic relatively to  its boundary components, and for $\pi_1(M)$ to be extended geometrically finite in the sense of Weisman \cite{We22}.}
\item{Section \ref{sec:sec5}: {\em Projective gluings}. We define projective gluings of projective manifolds with totally geodesic boundary and corners, and explain how to bulge those gluings with respect to polars. We check that projective gluings are cone-manifolds. Then we translate into the language of projective gluings the results obtained in Sections~\ref{sec:sec3} and \ref{sec:sec4}.}
\item{Section \ref{sec:sec6}: {\em Hyperbolic building blocks}. We construct the convex compact hyperbolic manifolds with totally geodesic boundary and corners which will be projectively glued and bulged to prove the main theorems. In particular, we prove Theorem~\ref{thm:main7}.}
\item{Section \ref{sec:sec7}: {\em Hyperbolic doubles}. We apply the results of Section~\ref{sec:sec5} to the particular case of the bulged double gluing of a single convex compact hyperbolic manifolds with totally geodesic boundary and corners, and then prove Theorems~\ref{thm:mainA}, \ref{thm:mainB}, \ref{thm:main1}, \ref{thm:main2} and \ref{thm:main3}.}
\end{itemize}

\subsection*{Acknowledgements}
We thank Dick Canary, Mitul Islam, Fanny Kassel, Beatrice Pozzetti, Alan Reid, Teddy Weisman for useful discussions and comments.


We thank Teddy Weisman for helping us producing the bulging figures which we obtained using his package Geometry Tools available on his website.

Gabriele acknowledges the financial support of the DFG 427903332 (Beatrice Pozzetti's Emmy Noether) and of the DFG 390900948 (Heidelberg Structures Cluster of Excellence).

Pierre-Louis acknowledges financial support from the Max-Planck-Institut für Mathematik during the years 2021-2022.

\section{Ingredients of the proofs}\label{sec:proofs}

In this section we discuss the ideas and ingredients of the proofs of this paper.
In Section~\ref{sec:guiding example} we investigate the special case where the dimension $d$ equals 2, before turning to the general case in Section~\ref{sec:generalization example}.

The proofs involve projective cone-manifolds obtained by gluing projective manifolds with totally geodesic and corners.
We explain how the singularities of certain projective cone-manifolds can be blown up to a totally geodesic boundary, and how to describe the geometry of certain gluings of projective manifolds with totally geodesic and corners.}

Finally, in Section~\ref{sec:super long title} we discuss the hyperbolic building blocks needed for our construction: 
In dimension 3, such objects are classified by work of Bonahon and Otal \cite{BO04}. 
In higher dimension, we construct them using arithmetic techniques (Theorem \ref{thm:main7}).

\subsection{A guiding example in dimension $2$}\label{sec:guiding example}

As mentioned before, the compact convex projective manifolds with totally geodesic boundary in Theorems~\ref{thm:main2} and \ref{thm:main3} are obtained by bulging enough a hyperbolic cone-manifold along totally geodesic hypersurfaces adjacent to the cone-singularities.
This hyperbolic cone-manifold is in fact obtained by taking the double of a convex compact hyperbolic manifold $N$ with totally geodesic boundaries $\Sigma,\Sigma'\subset\partial N$ and \emph{codimension 2} corners $\corner\subset\partial N$, see Figure~\ref{fig:figure2a}. 
($\Sigma,\Sigma',\corner$ are not necessarily connected.)

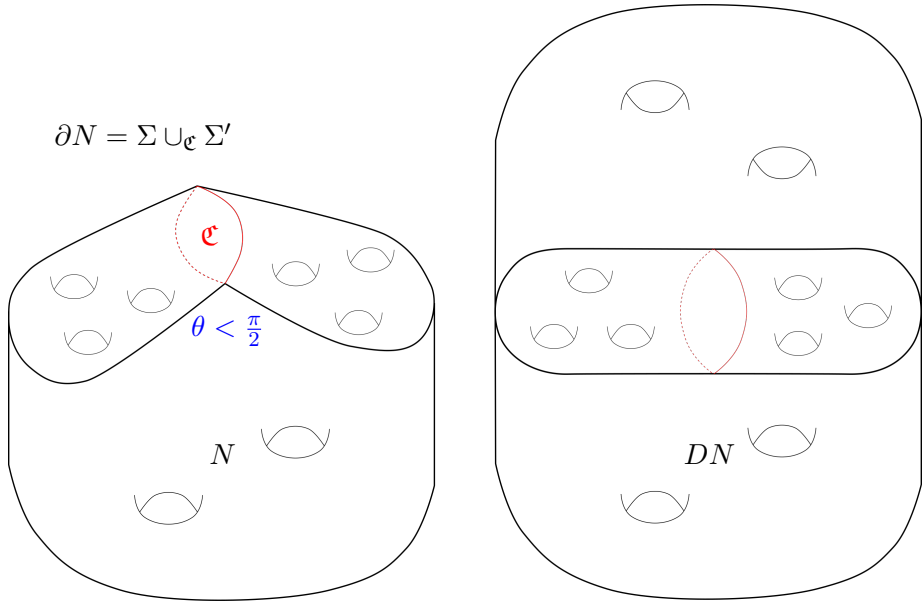
\begin{figure}[h]
\centering
\begin{overpic}[scale=0.7]{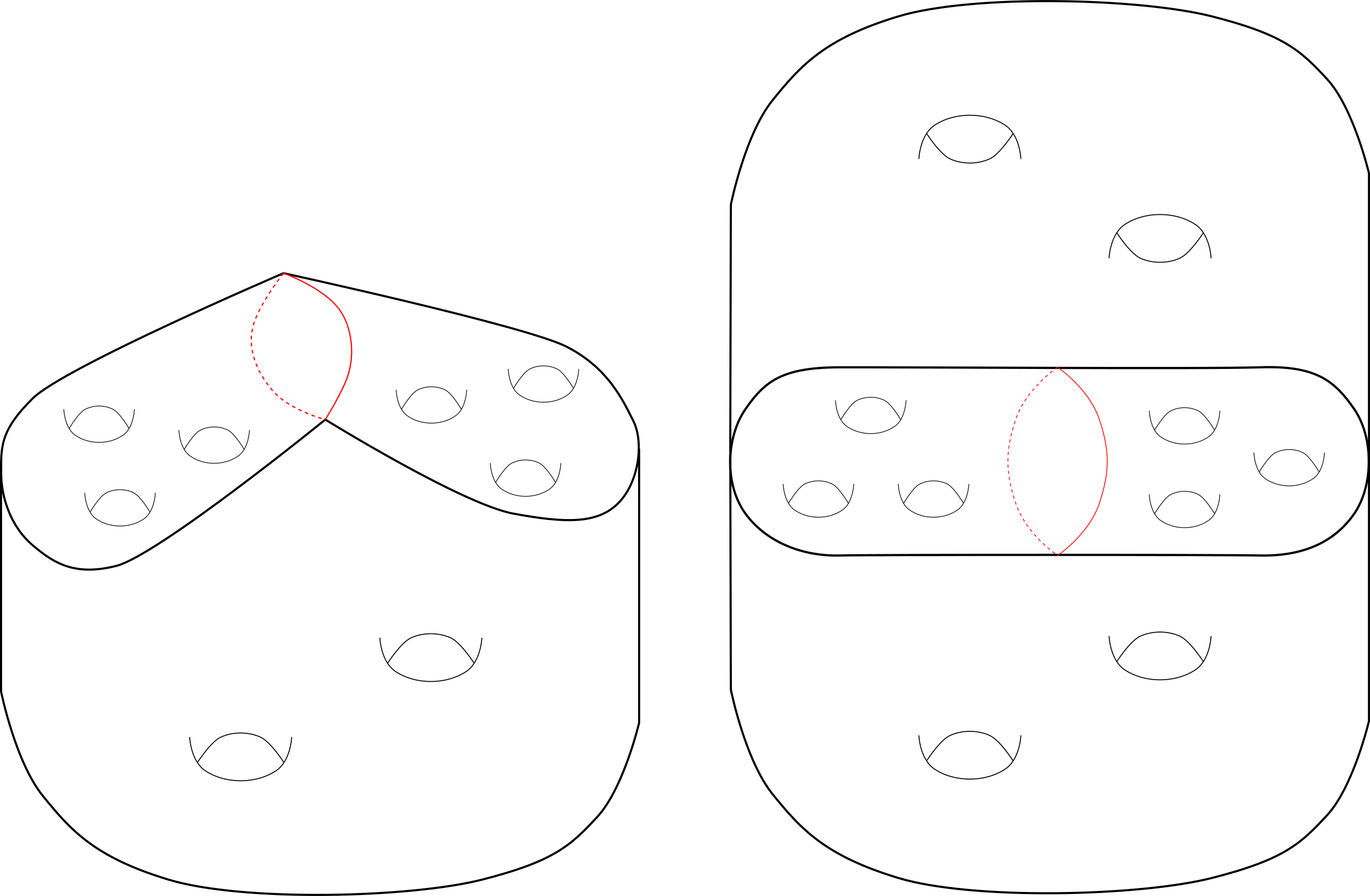}
\put (22,15) {$N$}
\put (21,39) {{\color{red} $\corner$}}
\put (20,29) {{\color{blue} $\theta<\frac{\pi}{2}$}}
\put (5,50) {$\partial N=\Sigma\cup_\corner\Sigma'$}
\put (74,15) {$DN$}
\end{overpic}
\caption{Double of a hyperbolic manifold with corners.}
\label{fig:figure2a}
\end{figure}

In this section we explain this bulging construction in the easier but already interesting case where $d=2$.
In this case one can pick for $N$ a compact quadrilateral of the hyperbolic plane with vertex angles less than $\pi/2$ (in dimension at least $3$ the construction of a suitable $N$ is not as obvious).

We discuss two prototypical examples in order to explain two different problems of the proof of Theorem~\ref{thm:main2}, namely {\em totally geodesic blowup} and {\em global convexity}. 

\subsubsection{Local model}\label{sec:local model}
The first example is a hyperbolic sector $\Sigma\subset\mb{H}^2$ of angle $0<\theta<\pi/2$ bounded by two geodesic rays $\ell_1,\ell_2$ issuing from the vertex $v$.

The double $S=D\Sigma$ is a hyperbolic cone with vertex $v$ and angle $2\theta$.
The complement of the cone point $S'=S-\{v\}$ is a punctured plane with an incomplete hyperbolic metric.

\begin{figure}[h]
\label{fig:figure2b}
\centering
\begin{overpic}[scale=0.5]{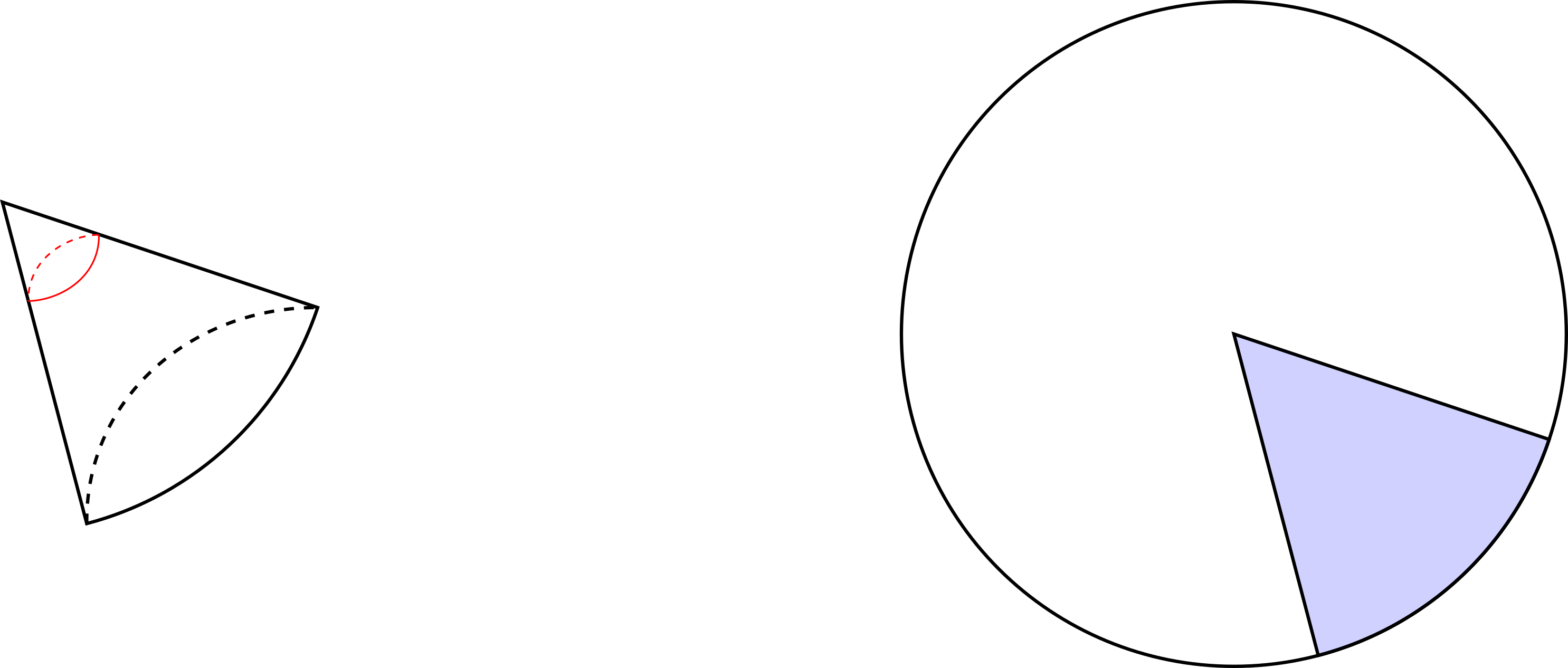}
\put(6,22) {{\color{red} $\gamma$}}
\put(-2,31) {{\color{blue} $2\theta$}}
\put(18,11) {$S=D\Sigma$}

\put(81,16) {{\color{blue} $\theta$}}
\put(77,22) {$v$}
\put(77,10) {$\ell_1$}
\put(88,19) {$\ell_2$}
\put(93,2) {$\Sigma$}

\end{overpic}
\caption{Hyperbolic cone: Double of a sector.}
\end{figure}
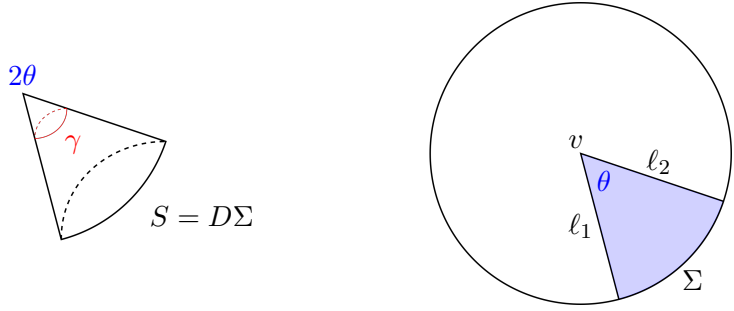

Let us see the universal cover $\tilde{S}'$ as a fan of sectors $\{\Sigma_j\}_{j\in\mb{Z}}$ lifting the sectors of $S'$,
and consider the hyperbolic developing map $\dev_0:\tilde S'\rightarrow\H^2$ sending $\Sigma_0$ onto $\Sigma$.
If $\gamma$ is the simple curve that winds once clockwise around the vertex, then its holonomy is given by the rotation $R_{2\theta}\in{\rm SO}(2,1)$ of angle $2\theta$ around the vertex $v$.
Moreover, $\dev_0(\Sigma_j)=R_{j\theta}\Sigma$ for every $j$.

We deform projectively the hyperbolic structure on $S'$ by performing independent bendings along each of the sides $\ell_1,\ell_2\subset S'$. 
Recall that these bendings are bulgings, \ie of the following form.
Every line $\ell\subset\mb{H}^2$ has a dual point $\ell^*\in\mb{RP}^2-\mb{H}^2$. 
For every $\mu>0$ the ``bulging'' transformation $B_{\ell,\mu}\in{\rm SL}_3(\mb{R})$ is the homothety by $1/\mu$ on the 2-plane $L\subset\R^3$ representing $\ell$ and is the homothety by $\mu^2$ on the line $L^\perp\subset\R^3$ representing $\ell^*$.  

We describe the (bulged) projective structure on $S'$ via the developing map $\dev$ that sends $\Sigma_0$ onto $\Sigma$.
We have $\dev(\Sigma_1)=B_{\ell_2,\mu_2}\dev_0(\Sigma_1)$ and $\dev(\Sigma_{-1})=B_{\ell_1,\mu_1}\dev_0(\Sigma_{-1})$, the bulged holonomy of $\gamma$ (which sends $\dev(\Sigma_{-1})$ on $\dev(\Sigma_1)$) changes as
\[
R_{2\theta}\leadsto\rho:=B_{\ell_2,\mu_2}R_{2\theta}B_{\ell_1,\mu_1}^{-1},
\]
and $\dev(\Sigma_{2j})=\rho^j\Sigma$ and $\dev(\Sigma_{2j+1})=\rho^jB_{\ell_2,\mu_2}R_\theta\Sigma$ for every $j$.

The idea is that, under suitable geometric and dynamical assumptions on $\rho$ the projective structure on $S'$ is convex and one can ``blow up'' the singularity $v$ of the projective cone-manifold $S$ into a totally geodesic circle, a boundary for $S'$.

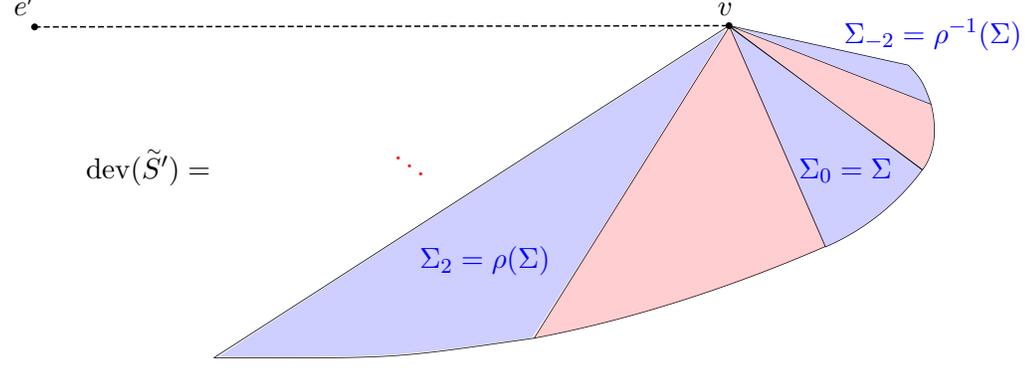
\begin{figure}[h]
\centering
\begin{overpic}[scale=0.5]{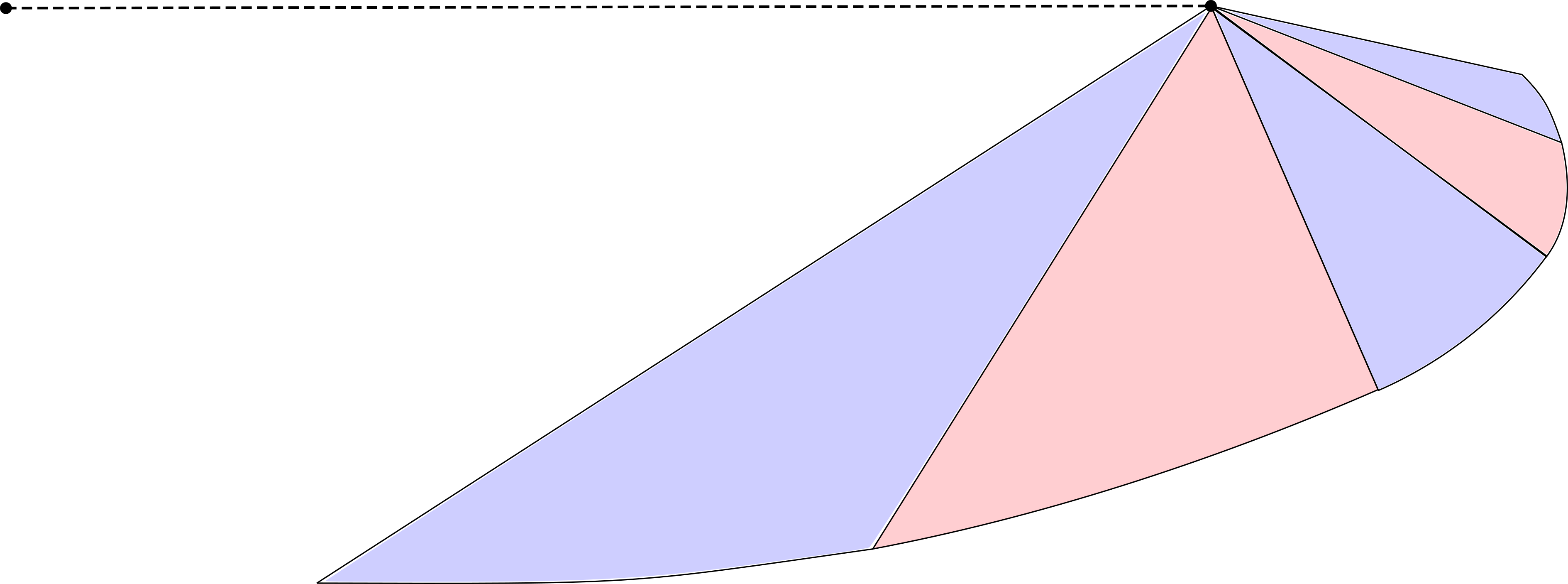}
\put(85,20) {{\color{blue} $\Sigma_0=\Sigma$}}
\put(43,10) {{\color{blue} $\Sigma_2=\rho(\Sigma)$}}
\put(90,35) {{\color{blue} $\Sigma_{-2}=\rho^{-1}(\Sigma)$}}
\put(76,38) {$v$}
\put(-2,38) {$e'$}
\put(6,20) {${\rm dev}(\tilde{S}')=$}
\put(40,20) {{\color{red} $\ddots$}}
\end{overpic}
\caption{Image of the new developing map $\dev$.}
\label{fig:figure2c}
\end{figure}

The heuristic picture is the following (see Figure \ref{fig:figure2c}).
Note that $v$ is an eigenvector of $\rho$ with eigenvalue $\mu_1/\mu_2$.
Assume that
\begin{equation}\label{eq:geomdyncond}
 \rho\ \text{is diagonalizable with eigenvalues}\ \mu_1/\mu_2<\mu<\mu',
\end{equation}
and denote by $e,e'$ the eigenlines of $\mu,\mu'$. 
(This holds \eg if $\theta<\pi/2$ and $\mu_2$ and $\mu_1^{-1}$ are large.)

The image $\Omega=\dev(\tilde{S}')$ does not intersect the two lines through $v$ that contain respectively $e$ and $e'$.
The sectors $\dev(\Sigma_j)$ accumulate onto the segment $[v,e']$ for $j\to\infty$ and to the vertex $v$ for $j\to-\infty$.
The properly convex subset $\Omega\cup(v,e')$ is a $\rho$-invariant convex set where $\rho$ acts properly discontinuously.
This gives the desired totally geodesic blowup 
\[
\bar{S}:=S'\cup((v,e')/\rho)=(\Omega\cup(v,e'))/\rho
\]
of the deformed cone $S'=\Omega/\rho$.

The non-convex-cocompact case of Theorem~\ref{thm:main2} corresponds to  $\mu=\mu'$ and $\rho$ acting as a parabolic transformation on the dual line to $v$.
Then convexity and totally geodesic blowup work the same way, except that $\ell=\ell'$ and $v$ is a $\cal C^1$ point of $\partial\Omega$.

\subsubsection{Global convexity}\label{sec:global convexity}
The second prototypical example is a hyperbolic convex quadrilateral $Q=Q(a,b,c,d)\subset\mb{H}^2$ with angles $0<\alpha,\beta,\gamma,\delta<\pi/2$ at the vertices $a,b,c,d$. 

The double $S=DQ$ is a sphere with a singular hyperbolic metric with four cone points of angle $2\alpha,2\beta,2\gamma,2\delta$ corresponding to the vertices $a,b,c,d$ of $Q$.
The complement of the cone points $S'=S\smallsetminus \{a,b,c,d\}$ is an incomplete hyperbolic sphere with four punctures. 

Let us bulge $S'$ along each of the sides $\ell_{ab},\ell_{bc},\ell_{cd},\ell_{da}$ of the quadrilateral with parameters $\mu_{ab},\mu_{bc},\mu_{cd},\mu_{da}$. 
Suppose $\mu_{ab},\mu_{bc}^{-1},\mu_{cd},\mu_{da}^{-1}$ are large, so that the holonomies $\rho_a,\rho_b,\rho_c,\rho_d$ of the curves around $a,b,c,d$ satisfy \eqref{eq:geomdyncond}.

By Section~\ref{sec:local model}, $S$ admits a projective totally geodesic blowup 
\[
{\bar S}=S'\cup(\mb{S}^1_a\cup \mb{S}^1_b\cup \mb{S}^1_c\cup \mb{S}^1_d)
\]
where each cone point $v\in\{a,b,c,d\}$ is replaced by a totally geodesic boundary circle $\mb{S}^1_v$.
Let us now explain why $\bar S$ is convex.

\begin{figure}[h]
\centering
\begin{overpic}[scale=0.4]{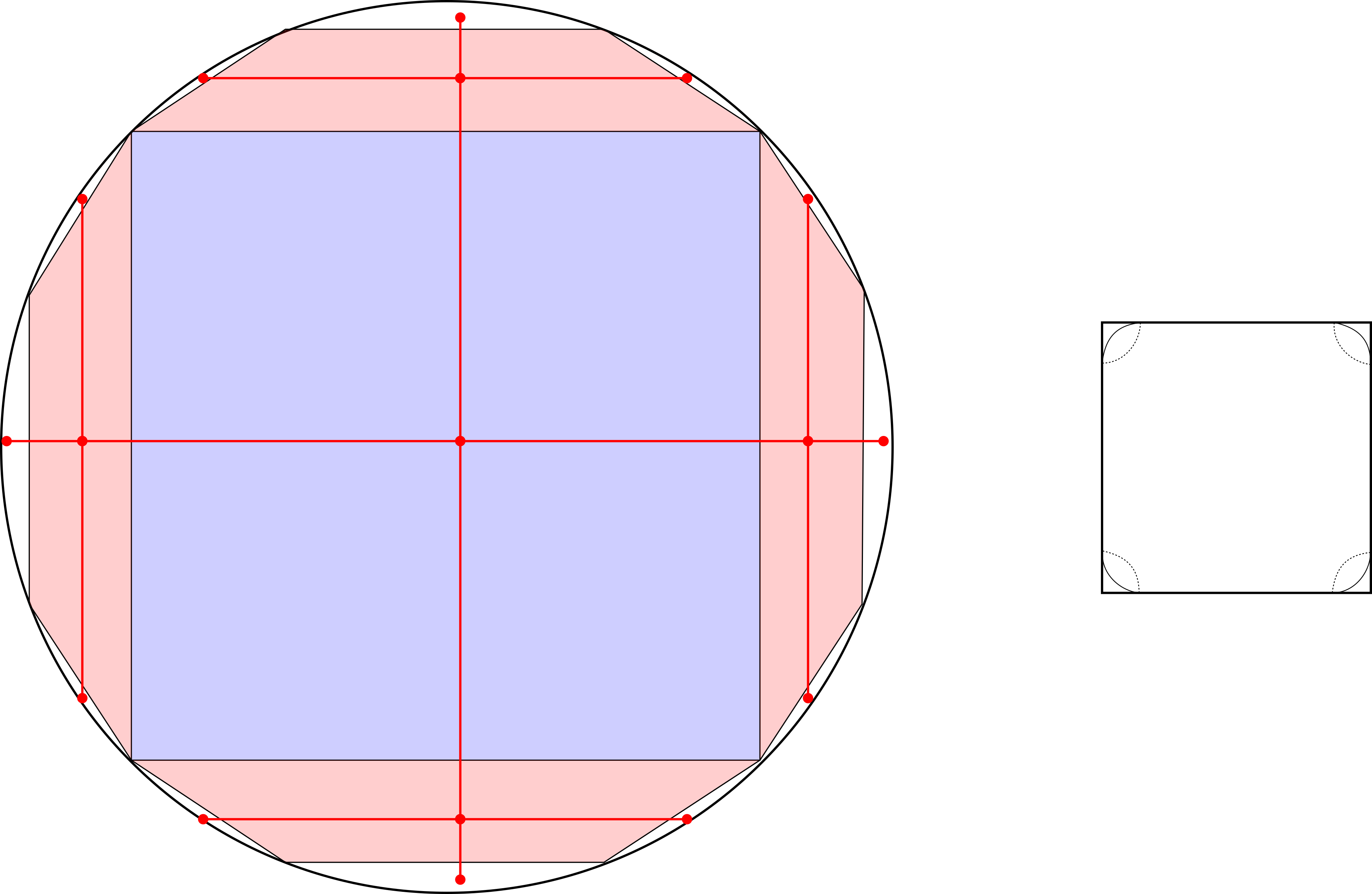}
\put (78,19) {$a$}
\put (100.5,42) {$c$}
\put (78,42) {$d$}
\put (100.5,19) {$b$}

{\small
\put (82.5,24.5) {{\color{blue} $2\alpha$}}
\put (94.5,36.5) {{\color{blue} $2\gamma$}}
\put (82.5,36.5) {{\color{blue} $2\delta$}}
\put (94.5,24.5) {{\color{blue} $2\beta$}}
}
\end{overpic}
\caption{Universal covering of $S'$.}
\label{fig:figure3b}
\end{figure}
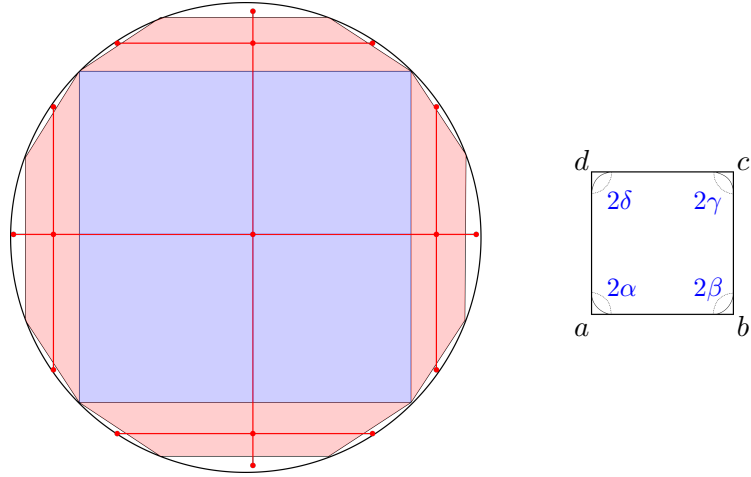

One way to describe combinatorially the universal cover $\tilde{S}'$ of $S'$ is to identify it with $\mb{H}^2$ via \emph{another} hyperbolic metric on $S'$ which is  complete finite volume.
The two copies of $Q\smallsetminus \{a,b,c,d\}$  lift to an ideal tessellation $\mc{F}$ of $\mb H^2$ whose dual graph is a regular 4-valent tree (see Figure~\ref{fig:figure3b}). We now describe the image $\Omega$ of the (bulged) developing map 
\[
\dev:\tilde{S}'\to\mb{RP}^2.
\]
By the local model, for each vertex at infinity of the tessellation of $\tilde{S}'$, the fan of quadrilaterals adjacent to it develops into an open convex subset of a quadrant of $\mb{RP}^2-(\ell\cup\ell')$ where $\ell,\ell'$ are suitable lines. 
A local-to-global argument inspired by work of Vinberg \cite{Vin71} and revisited by Benoist \cite{Be5lectures} (see Proposition~\ref{prop:injetcvx}) guarantees that the image $\Omega=\dev(\tilde{S}')$ is a properly convex domain tiled by the images of the quadrilaterals in the tessellation.

\begin{figure}[h]
\centering  
\begin{subfigure}{.25\textwidth}
\centering
\includegraphics[scale=0.34]{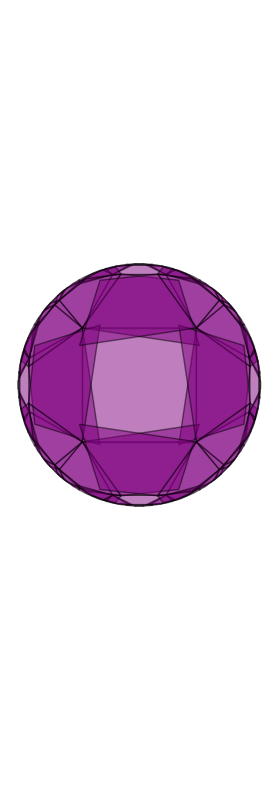}
\caption{No bulging.}
\label{fig:sub1}
\end{subfigure}%
\begin{subfigure}{.5\textwidth}
\centering
\includegraphics[scale=0.34]{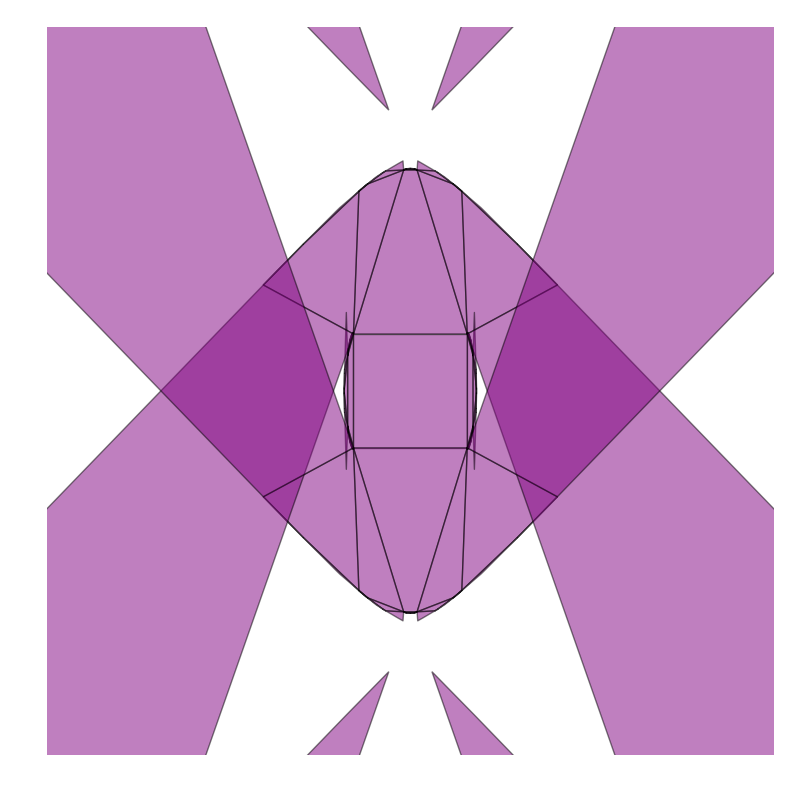}
\caption{Moderate bulging.}
\label{fig:sub2}
\end{subfigure}%
\begin{subfigure}{.25\textwidth}
\centering
\includegraphics[scale=0.34]{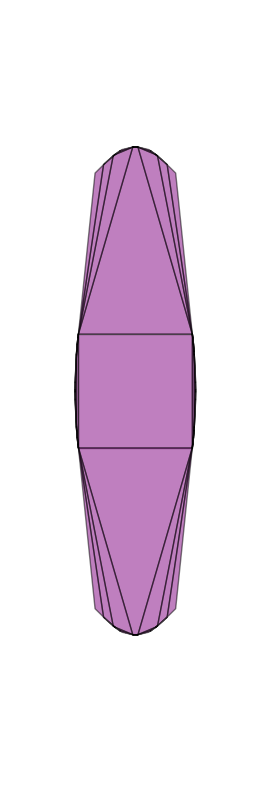}
\caption{Large bulging.}
\label{fig:sub3}
\end{subfigure}
\caption{Developing 52 tiles with different bulging parameters. The last picture is convex.}
\label{fig:figure3c}
\end{figure}

%

\subsection{A generalization of the guiding example}
\label{sec:generalization example}

The ideas presented in Section~\ref{sec:guiding example} apply more generally to gluings of the hyperbolic manifolds with totally geodesic boundary and corners.
We will in fact develop these ideas in a much more general framework, which includes:
\begin{itemize}
\item{Totally geodesic blowups of projective cone-manifolds.}
\item{A local-to-global convexity result for general unions of convex sets.}
\end{itemize}

The reason why we work in such a general setting is threefold: 
First, as we use gluings twice with different building blocks, it is convenient to work in an abstract setup that encompasses both constructions. 
Second, we think that within that general framework some of the arguments are natural. 
Last, we hope that our techniques could be well-suited to produce many new interesting projective manifolds (for example, as the ones in Questions \ref{qn:question1} and \ref{qn:question2}).  

We briefly describe the main steps.

\subsubsection*{Tubes and cone-manifolds}
We define a suitable class of {\em projective manifolds with cone-singularities} $(M,\mc{C})$, slightly more general than the previous notions in \cite{BBS11,Dan13,LMR22}.
The most important feature is that every singularity $\corner\in\mc{C}$ has a neighborhood $U$ locally modeled on a {\em tube}.

We work in $\sph^d$ the sphere of $\mb{R}^d$ (a double cover of $\mb{RP}^d$). 
Consider a {\em sector} $\Sigma\subset\sph^d$ bounded by two half-spheres $H_1,H_2$ containing $\sph^{d-2}$.
Glue $H_1$ and $H_2$ via  $\alpha\in G$  the group of transformations of $\sph^\dimd$ fixing $\sph^{\dimd-2}$.
The quotient
$S=\Sigma/\alpha$
is a projective manifold with a singularity $\sph^{d-2}\subset S$.
We call it \emph{tube} and think of it as a {\em projective cone} with {\em angle} $\alpha\in G$, analogous to a {\em hyperbolic cone} of angle {\em smaller} than $2\pi$. 
As in the hyperbolic setting, we will allow angles larger than $2\pi$ by modifying the above picture, taking $\alpha\in\tilde G$ in the universal cover of $G$.

Generalizing Section~\ref{sec:local model}, the {\em totally geodesic blowup} that we introduce is a local construction for standard tubes whose angle satisfies \eqref{eq:geomdyncond} (and ``have angle less than $\pi$'').
It replaces each singularity $\corner\subset M$ with a totally geodesic boundary component of the (topological) form $\corner\times\mb{S}^1$. 

\subsubsection*{Periodic tessellations of convex domains}
As in Section~\ref{sec:global convexity}, consider a collection of convex sets (tiles) that satisfies a local convexity property (ensured among other things by considerations like in Section~\ref{sec:local model}).

A local-to-global convexity argument ensures that this collection tessellates a properly convex domain $\Omega$.
Then we describe under additional assumptions the boundary of $\Omega$ and its regularity, with roughly three kinds of points.

The codimension 1 faces of the tiles that lie inside $\Omega$ are called {\em walls}.
First, for any tile $D$, the complement of closures of walls in $\partial D$ is contained in $\partial\Omega$, and its points have the same regularity in $\partial D$ and in $\partial \Omega$.

The remaining part of the boundary $\partial\Omega$ consists of {\em cells at infinity}, Hausdorff limit of sequences of tiles.
By \emph{periodicity} of the tiling (a finite number of tiles up to isometry), there will be essentially two cases to consider:
\begin{enumerate}
\item{A {\em telescope} of consecutive tiles $\{D_n\}_{n\in\mb{N}}$ with empty global intersection.
The corresponding cell at infinity is reduced to an extremal point of $\partial\Omega$.}
\item{A {\em fan} of tiles $\{D_n\}_{n\in\mb{Z}}$ sharing a codimension 2 face $H\subset\partial\Omega$, with $\gamma D_j=D_{j+k}$ for some $\gamma\in{\rm Aut}(\Omega)$ satisfying \eqref{eq:geomdyncond}.
The \emph{two} corresponding cells at infinity are $H$ and a codimension 1 face ${\mc CH}(H,\gamma_+)\subset\partial\Omega$ where $\gamma^+\in\partial\Omega\smallsetminus H$ is the attracting fixed point of $\gamma$ (as in Figure \ref{fig:figure2c}).}
\end{enumerate}

In addition to these ideas of totally geodesic blowup and periodic tessellations (which generalize Section~\ref{sec:guiding example}) and to the constructions presented in Section~\ref{sec:super long title}, the proof of the main theorems will use classical arguments from the theory of convergence actions and geometrically finite actions to prove the statements about relative hyperbolicity, as in \cite{IZ19relhypb,We21,IZ22relhypb}.

We now discuss how to produce hyperbolic manifolds with totally geodesic boundary and pleated along codimension 2 corners in arbitrary dimension.

\subsection{Hyperbolic building blocks}
\label{sec:super long title}
Convex compact hyperbolic manifolds with totally geodesic boundary and corners that look like the one in Figure~\ref{fig:figure2a} exist in every dimension.

\subsubsection{Dimension $d=3$}
This is a special case, where we have the following complete topological classification thanks to work of Bonahon and Otal \cite{BO04}.
Simple closed curves $\alpha_1\cup\cdots\cup\alpha_k\subset\partial M$ and angles $0<\theta_1,\cdots,\theta_k<\pi$ appear as pleating (or bending) locus and pleating angles of the boundary of a hyperbolic metric on $M$ with totally geodesic boundary and corners if and only if
\begin{itemize}
\item{$\sum_{j\le k}{\theta_j i(\alpha_j,\partial A)}>0$ for every properly embedded essential annulus or Möbius band $(A,\partial A)\subset(M,\partial M)$.}
\item{$\sum_{j\le k}{\theta_j i(\alpha_j,\partial D)}>2\pi$ for every properly embedded essential disk $(D,\partial D)\subset(M,\partial M)$.}
\end{itemize}

Such collections $\alpha_j,\theta_j$ are abundant (for example by \cite{Le05}).
Note that the $\theta_i$'s are \emph{pleating angles} and not \emph{corner angles}, which are the $\pi-\theta_i$'s.

Furthermore, we also have control on the lengths $\ell(\alpha_j)$ of the geodesic corners $\alpha_j\subset\partial M$ by work of Choi and Series \cite{CS06}.

\begin{fact}[{Bonahon--Otal \cite[Th.\,2-3]{BO04}, Choi--Series \cite[Th.\,A]{CS06}}]\label{fact:ChoiSeries}
 Let $M$ be a compact orientable irreducible atoroidal 3-manifold with non-empty boundary $\partial M=\Sigma_1\sqcup\dots\sqcup\Sigma_n$.
 Let $\alpha:=\alpha_1\cup\dots\cup\alpha_n$ be a doubly incompressible multicurve where $\alpha_j$ is a separating simple closed curve of $\Sigma_j$. Then for $\epsilon>0$ small enough, to any lengths $\ell_1,\dots,\ell_n<\epsilon$ can be associated a hyperbolic metric on $M$ with totally geodesic boundary and corners, whose corners are the $\alpha_i$'s, with length $\ell_i$ and angle at most $\pi/4$.
\end{fact}

See Section~\ref{sec:dim3} for more details.

\subsubsection{Dimension $d\ge 4$}
In the general case our hyperbolic building blocks are provided by Theorem~\ref{thm:main7}, which is proved via arithmetic methods.
Let us discuss the construction of the manifold $M$ of the point (a) of this theorem in the case $k=4$ (where $M$ contains $4$ totally geodesic closed hypersurfaces intersecting along a single codimension $2$ closed submanifold).

The manifold $M$ is $\mb{H}^d/G$ where $G$ is a suitable finite index subgroup of
\[
G_0:={\rm SO}\left(x_1^2+\cdots+x_d^2-\sqrt{2} x_{d+1}^2\right)\cap{\rm SL}_{d+1}\left(\mb{Z}[\sqrt{2}]\right)
\] 

By classical facts, $G_0$ is discrete and $\mb{H}^d/G_0$ is compact. Consider the hyperplane $H=\{x_1=0\}$, the codimension 2 plane $V=\{x_1=x_2=0\}$, and the $\pi/4$-rotation around $V$ 
\[\rho=
\begin{pmatrix}
\sqrt{2}/2 &-\sqrt{2}/2 &\\
\sqrt{2}/2 &\sqrt{2}/2 &\\
  & &\mb{I}_{d-1}\\
 \end{pmatrix}\in {\rm SL}_{d+1}\left(\mb{Q}(\sqrt{2})\right).
\]
We exploit a ping-pong system adapted to the configuration of four hyperplanes $H_j=\rho^jH$ and their stabilizers in $G_0$ for $j\le 4$ to find a subgroup $Q<G_0$ with the following properties:
\begin{enumerate}[(i)]
\item{$Q$ is convex-cocompact.}
\item{The hyperplanes $H_j$ project to embedded compact hypersurfaces $N_j\subset\mb{H}^d/Q$ that pairwise intersect only along the projection of $V$ which is a codimension 2 totally geodesic embedded submanifold.}
\end{enumerate}
Given (i) and (ii), strong subgroups separability properties of $G$ (relatively to $Q$) due to Bergeron, Haglund, and Wise \cite{BHW11} allow us to embed $N_1\cup N_2\cup N_3\cup N_4$ in a finite cover $M=\mb{H}^d/G\to\mb{H}^d/G_0$.


\section{Convex projective geometry}
\label{sec:sec1}

In this section we review two notions that will be needed throughout the paper: We first recall the basic structure of convex sets in projective spaces. Then we discuss geometric structure, or $(G,X)$-structure, introducing (convex) projective manifolds with totally geodesic boundary and corners.

\subsection{Convex domains}
\label{sec:reminder cvx}

We start with some basic notions and terminology about convex subsets of $\sph^\dimd$, the sphere of rays in $\mb{R}^{d+1}$.

\begin{defi}[Properly Convex]
\label{def:prop convex}
A {\em convex} subset $K\subset\sph^\dimd$ is the image of a convex cone of $\R^{\dimd+1}$. The convex set $K$ is {\em properly convex} if its closure does not contain two antipodal points, that is, the closure of $K$ is contained in an affine chart $\mb{R}^\dimd\subset\sph^\dimd$.
\end{defi}

Note that intersections of (properly) convex subsets is again (properly) convex. Thus one can always define the smallest closed convex subset and the smallest linear subspace containing a given properly convex set. 

\begin{defi}[Convex Hull]
If $A\subset\sph^\dimd$ is contained in some properly convex subset $K\subset\sph^\dimd$, then we can define the {\em convex hull} $\mc{CH}(A)$ of $A$ in $K$ as the intersection of all the closed convex subsets of $K$ containing $A$. 
\end{defi}

\begin{defi}[Span and Dimension]
The {\em span} $\sph(K)\subset\sph^\dimd$ of the convex set $K$ is the smallest linear subspace containing $K$. We call ${\rm dim}(\sph(K))$ the {\em dimension} of $K$. One-dimensional convex subsets of $\sph^\dimd$ are called {\em segments}.
\end{defi}

Observe that any two non-antipodal points $x,y\in\sph^\dimd$ are contained in a unique minimal (properly convex) segment denoted by $[x,y]$.

\subsubsection{Topology}
The span $\sph(K)$ of the convex set $K\subset\sph^\dimd$ has the property that $K$ has non-empty interior in it.

\begin{defi}[Interior and Boundary]
We call ${\rm int}_{\sph(K)}(K)$ and $\partial_{\sph(K)}K$ the {\em relative interior} and {\em relative boundary} of $K$ respectively.
\end{defi}

Topologically, the pair $({\rm int}_{\sph(K)}(K)\cup\partial_{\sph(K)}K,\partial_{\sph(K)}K)$ is always homeomorphic to $(D,\partial D)$ where $D$ is an Euclidean disk of dimension ${\rm dim}(K)$. 

\subsubsection{Structure of the boundary}
The relative boundary of a properly convex set $K$ has a stratified structure.

\begin{defi}[Supporting Hyperplanes]
A {\em supporting hyperplane} of a convex subset $K\subset\sph^\dimd$ with non-empty interior is a hyperplane which intersects the closure of $K$ only in $\partial K$.
\end{defi}

For every point $\xi\in\partial K$ there always exists a supporting hyperplane passing through $\xi$, however it is not necessarily unique. 

\begin{defi}[$\mc{C}^1$-point]
A point $\xi\in\partial K$ admitting a unique supporting hyperplane, denoted by $T_\xi\partial K$, is called a {\em $\mc{C}^1$-point}.
\end{defi}

The following criterion will be sometimes useful.

\begin{fact}
\label{item:CNS smooth} 
Suppose $K$ spans $\sph^\dimd$. Consider two subspaces $\sph(V),\sph(W)\subset\sph^\dimd$ whose union spans $\sph^\dimd$ such that $\sph(V)\cap\sph(W)$ intersects the interior of $K$, and $x\in \sph(V)\cap\sph(W)\cap\partial K$. Then $x$ is $\cal C^1$ in $\partial K$ if and only if it is $\cal C^1$ in both $\partial K\cap\sph(V)$ and $\partial K\cap \sph(W)$.
\end{fact}

\begin{defi}[Faces]
Let $K\subset\sph^\dimd$ be a convex set with non-empty interior and let $\xi\in\partial K$ be a boundary point. The convex subset
\[
F_K(\xi):=\overline{K}\cap\bigcap_{\text{\rm $H$ supporting at $\xi$}}{H}
\] 
is the {\em closed face} of $\xi$ in $\overline{K}$ (its relative interior is the {\em open face} of $\xi$). Note that it is contained in $\partial K$. Every face has a {\em dimension} ${\rm dim}(F_K(\xi))$ and a {\em codimension} ${\rm codim}(F_K(\xi))={\rm dim(K)}-{\rm dim}(F_K(\xi))$.
\end{defi}

\begin{defi}[Extremal and Strictly Convex]
If $F_K(\xi)$ is reduced to a point, then $\xi$ is said to be {\em extremal}. If every boundary point is extremal, or, equivalently, there are no non-trivial segments in $\partial K$, then $K$ is said to be {\em strictly convex}. 
\end{defi}

\subsubsection{Symmetries}
The group ${\rm SL}_{d+1}(\mb{R})$ acts on the sphere of rays $\sph^\dimd$ by linear isomorphisms. Given a properly convex subset $K\subset\sph^\dimd$, we define the following natural group of symmetries:

\begin{defi}[Automorphisms]
Let $K\subset\sph^\dimd$ be an open properly convex subset. The group of {\em projective automorphisms} of $K$ is
\[
{\rm Aut}(K):=\{A\in{\rm SL}_{\dimd+1}(\mb{R})\left|\;A(K)=K\right.\}.
\]
\end{defi}

We will measure the size of ${\rm Aut}(\Omega)$ in terms of its Zariski closure:

\begin{defi}[Zariski Dense]
The {\em Zariski topology} on ${\rm SL}_{d+1}(\mb{R})$ is the topology whose closed subsets are the intersections of zero loci of polynomials in the matrix entries.

In this topology a subset $\Gamma<{\rm SL}_{d+1}(\mb{R})$ is {\em Zariski dense} if every polynomial in the matrix entries that vanishes on the elements of $\Gamma$ vanishes on all matrices in ${\rm SL}_{d+1}(\mb{R})$.  
\end{defi}

Notice that the Zariski closure of a subgroup $\Gamma<{\rm SL}_{d+1}(\mb{R})$ is always a Lie subgroup. Similar to the concept of lattices in Lie groups, we have the following: 

\begin{defi}[Divisible]
An open properly convex subset $K\subset\sph^\dimd$ is {\em divisible} if ${\rm Aut}(K)$ contains a subgroup $\Gamma$ acting properly discontinuously, freely, and cocompactly on $K$.
\end{defi}

Our goal is to construct examples of domains $K$ with special geometric properties where $K$ is divisible and $\Gamma$, the group dividing it, is as large as possible in the sense that it is Zariski dense. 

\subsubsection{Hilbert metric}
Every properly convex set $K\subset\sph^\dimd$ comes equipped with a natural ${\rm Aut}(K)$-invariant metric, called the Hilbert metric, whose geometric features are strictly tied to the regularity of the boundary $\partial K$.  

\begin{defi}[Hilbert Metric]
\label{rem:hilbert metric}
The {\em Hilbert metric} on $K$ is defined as follows: Consider $x,y\in K$. Let $a,b$ be the intersections of the projective line spanned by $x,y$ with $\partial K$ where $a$ lies on the side of $x$ and $b$ lies on the side of $y$. One sets
\[
d_K(x,y):=\frac12 \log([a,x,y,b]),
\]
where $[a,x,y,b]$ denotes the {\em cross-ratio} on the projective line spanned by $x,y$ normalized so that $[0,1,t,\infty]=t$.
\end{defi}

It is a classical fact that $d_K(\bullet,\bullet)$ defines a metric inducing the standard topology and such that metric balls are compact. As a consequence, as the group of projective automorphisms preserves the Hilbert metric, the action ${\rm Aut}(K)\curvearrowright K$ is proper.

There are a couple of useful general properties of the Hilbert metric that we will exploit several times. First, notice that $d_K(\bullet,\bullet)$ is monotone under inclusion, that is, if $K\subset K'$, then $d_{K}(x,y)\ge d_{K'}(x,y)$ for all $x,y\in K\subset K'$. Secondly, $d_K$ can always be extended to a lower semi-continuous function $d_{\bar K}(\bullet,\bullet)$ on $\bar K\times\bar K$ by setting
\begin{itemize}
 \item $d_{\bar K}(x,y)=\infty$ if $x$ and $y$ are in different faces.
 \item $d_{\bar K}(x,y)=d_F(x,y)$ if $x$ and $y$ share the same open face $F$.
\end{itemize}

\subsection{Geometric structures}
We now introduce the second crucial notion of the paper, namely the one of geometric structure and, in particular, projective structures with totally geodesic boundary and corners.

Let $G$ be a Lie group. Let $G\curvearrowright X$ be a {\em transitive} and {\em effective} action on a manifold $X$. Effective means that if $g,g'\in G$ agree on some open set of $X$, then they are equal.

\begin{defi}[$(G,X)$-structure]
A {\em $(G,X)$-structure} on a topological space $M$ is an atlas of charts
\[
\mc{A}=\{\phi_j:U_j\subset M\to V_j\subset X\}_{j\in J}
\]
such that every change of charts $\phi_i\phi_j^{-1}$ is a restrictions of a transformation $g_{ij}\in G$.

Via a standard process of analytic continuation of the local charts, every $(G,X)$-structure on $M$ has an associated {\em developing map} 
\[
{\rm dev}:{\tilde M}\to X,
\]
where ${\tilde M}$ is the universal cover of $M$, and a {\em holonomy representation} 
\[
\rho:\pi_1(M)\to G.
\]
The developing map is a local homeomorphism which is $\rho$-equivariant with respect to the deck group action $\pi_1(M)\curvearrowright{\hat M}$.
\end{defi}

\begin{defi}[Uniformisability and Completeness]
A $(G,X)$-manifold $M$ is {\em uniformisable} in $X$ if the developing map ${\rm dev}:\tilde{M}\to X$ is an embedding. In this case, we identify $M$ with the quotient ${\dev}(\tilde{M})/\rho(\pi_1(M))$, where $\rho:\pi_1(M)\to G$ is the holonomy representation. 

If $M$ is uniformisable and ${\dev}(\tilde{M})=X$ we say that $M$ is {\em complete}.
\end{defi}

We refer to \cite[Ch.\,3]{ThNotes} for more on the subject.

\subsubsection{Projective structures}
Projective structures correspond to the pair $({\rm SL}_{\dimd+1}(\mb{R}),\sph^\dimd)$. We will consider also projective structures locally modeled on pieces of $\sph^\dimd$ like hemispheres and half-hemispheres. They are defined as follows.

\begin{defi}[Projective structures] 
\label{def:projstructures}
A {\em projective structure} with {\em totally geodesic boundary} and {\em corners} on $M$ is a maximal atlas of charts into $\mb{S}^\dimd\cap\{x_1,x_2\ge 0\}$ with change of charts induced by restrictions of elements of ${\rm SL}_{\dimd+1}(\R)$.

The atlas of charts induces a compatible structure of $\dimd$-manifold with boundary and corners on $M$ where
\begin{itemize}
\item{The interior ${\rm int}(M)$ consists of those point mapped to $\mb{S}^\dimd\cap\{x_1,x_2>0\}$ by some (hence every) local chart.}
\item{The boundary $\partial M$ is the set of points mapped to $\mb{S}^\dimd\cap\{x_1=0\}\cup\mb{S}^\dimd\cap\{x_2=0\}$ by some (hence every) local chart.}
\item{The union of corners $C\subset\partial M$ is given by the set of points mapped to $\mb{S}^\dimd\cap\{x_1=x_2=0\}$ by some (hence every) local chart.}
\item{A \emph{wall} is a connected component of $\partial M\smallsetminus C$ and a \emph{corner} is a connected component of $C$.}
\end{itemize}  
\end{defi}

In order to ease the terminology, throughout the paper we will write \emph{b-manifold} (\resp \emph{bc-manifold}) instead of projective manifold with totally geodesic boundary (\resp projective manifold with totally geodesic boundary and corners).

Let ${\rm dev}:\tilde{M}\to\sph^\dimd$ be the developing map of a bc-manifold $M$. 

\begin{defi}[Convex Projective Manifolds]
A bc-manifold is said to be \emph{convex} (\resp {\em properly convex}) if the developing map is an embedding (uniformisable in $\sph^\dimd$) and its image is convex (\resp properly convex) in $\mb{S}^\dimd$.
\end{defi}

In Section \ref{sec:sec5}, we will explain a procedure to assemble several convex bc-manifolds in a single singular object, a projective manifold with cone-singularities. We formally define and analyze such objects in the next section (Section \ref{sec:sec2}) where we also explain why they are very useful for our purposes: Under suitable (local) geometric conditions, one can remove from a projective cone-manifold the singularities and replace them with totally geodesic boundaries.

\subsubsection{Hyperbolic structures}
A particularly flexible and rich class of projective structures is provided by hyperbolic manifolds. We will use them in the proofs of our main theorems as building blocks for our constructions.

\begin{defi}[Hyperbolic Space]
Let $\langle\bullet,\bullet\rangle$ denote the quadratic form
\[
\langle\bullet,\bullet\rangle:=x_1^2+\ldots+x_d^2-x_{d+1}^2
\]
on $\mb{R}^{\dimd+1}$. The hyperbolic $d$-space $\mb{H}^\dimd$ can be described as 
\[
\mb{H}^d=\{x\in\sph^d\left|\langle x,x\rangle<0,x_{d+1}>0\right.\}.
\]
\end{defi}

Note that $\mb{H}^\dimd$ is a properly convex subset of $\sph^\dimd$.

Its symmetry group is ${\rm SO}_0(d,1)<{\rm SL}_{d+1}(\mb{R})$, the (identity component of the) group of linear isometries of the quadratic form. Its Hilbert metric comes from a complete ${\rm SO}_0(d,1)$-invariant Riemannian metric defined at a point $x\in\mb{H}^\dimd$ by the (positive definite) restriction of the quadratic form $\langle\bullet,\bullet\rangle$ to the tangent space $T_x\mb{H}^\dimd=x^\perp$. 

Totally geodesic subspaces of $\mb{H}^\dimd$ are exactly those of the form $\sph(V)\cap\mb{H}^\dimd$ where $V\subset\mb{R}^\dimd$ is a linear subspace. Thus, in $\mb{H}^\dimd$ the notion of being geodesically convex with respect to the Riemannian metric coincides with the notion of being convex in $\sph^\dimd$.

\begin{defi}[Convex Hyperbolic Manifolds] 
A {\em convex hyperbolic manifold} is a quotient $M=\mc{C}/\Gamma$ of a convex subset $\mc{C}\subset\mb{H}^\dimd$ by a discrete torsion free subgroup $\Gamma<{\rm SO}_0(\dimd,1)$ preserving $\mc{C}$.
\end{defi}

We will also exploit the following useful criterion:

\begin{fact}
\label{fact:quasi-convex}
Let $\Gamma<{\rm SO}(d,1)$ be a finitely generated discrete subgroup. The following are equivalent:
\begin{itemize}
\item{The orbit $\Gamma o$ of a point $o\in\mb{H}^\dimd$ is a {\rm quasi-convex} subset of $\mb{H}^d$.}
\item{There exists a $\Gamma$-invariant convex subset $\mc{C}\subset\mb{H}^\dimd$ such that $\mc{C}/\Gamma$ is compact.}
\end{itemize}
\end{fact}

We recall that a subset $S\subset\mb{H}^\dimd$ is quasi-convex if there exists a constant $R>0$ such that for every $x,y\in S$ the geodesic segment $[x,y]\subset\mb{H}^\dimd$ is contained in the $R$-neighborhood of $S$.

\begin{defi}[Totally Geodesic Boundary and Corners] 
A convex hyperbolic manifold $M=\mc{C}/\Gamma$ has {\em totally geodesic boundary and corners} if the boundary $\partial M$ decomposes as a union of totally geodesic codimension 1 submanifolds intersecting along totally geodesic codimension 2 submanifolds.
\end{defi}

Note that a convex hyperbolic manifold with totally geodesic boundary and corners is also a convex projective manifold with boundary and corners according to Definition \ref{def:projstructures}.

\section{Tubes, cone-manifolds, and totally geodesic blowup}
\label{sec:sec2}

The goal of this section is to generalize the notion of cone-manifolds (with singularities of codimension $2$ only) from the hyperbolic setting to the projective setting.
Projective cone-manifolds have been considered in \cite{BBS11,Dan13,RS22,LMR22}.
Note that Riolo-Seppi \cite[Def.\,5.3]{RS22} allow for singularities of any codimension, although they have other kinds of restrictions.

Our presentation is most similar to that of \cite[Def.\,2.1]{LMR22}, but we will be slightly more general.
Indeed, while in \cite{LMR22} projective singularities are (essentially) determined by a projective structures on a circle (equivalently, a conjugacy class in the universal cover of $\SL_2\R$), in our case they can be (roughly) parametrized by conjugacy classes in the universal cover of the fixator of $\sph^{\dimd-2}$ in $\sph^\dimd$.

Depending on the type of this conjugacy class, we show that the corresponding singularity admits a totally geodesic blowup. This is the main result of this section.

\subsection{Tubes and cone-manifolds}
We start by studying the local model of a singularity of a projective cone-manifold, which is what we call a tube.

We identify $\R^{\dimd-1}$ with a subspace of $\R^{\dimd+1}$, $\sph^{\dimd-2}$ with the corresponding subspace of $\sph^\dimd$, and $\sph^1$ with $\sph(\R^{\dimd+1}/\R^{\dimd-1})$ the set of half-hyperspheres that contain $\sph^{\dimd-2}$. Note that we have a natural map
\[
\sph^\dimd\smallsetminus\sph^{\dimd-2}\rightarrow\sph^1=\sph(\R^{\dimd+1}/\R^{\dimd-1})
\]
sending a point $x\in\sph^\dimd\smallsetminus\sph^{\dimd-2}$ to the unique half-hypersphere passing through $x$ and $\sph^{\dimd-2}$.

\begin{defi}[Universal Branched Cover]
Consider 
\[
 \blow = \{(H,x)\in \tilde\sph^1\times (\sph^\dimd\smallsetminus\sph^{\dimd-2}):\ x\in \pi(H)\},
\]
where $\pi:\tilde\sph^1\rightarrow \sph^1$ is the universal cover of $\sph^1$. The {\em universal branched cover of $\sph^\dimd$ along $\sph^{\dimd-2}$} is the space (endowed with the quotient topology)
\[
 \blow\sqcup\sph^{\dimd-2} = \{(H,x)\in \tilde\sph^1\times \sph^\dimd:\ x\in \pi(H)\}/_\sim,
\]
where $(H,x)\sim (H',x)$ for all $H,H'\in\tilde\sph^1$ and $x\in\sph^{\dimd-2}$.
\end{defi}

Note that the natural projection $\blow\to\sph^\dimd\smallsetminus\sph^{\dimd-2}$ is a (universal) covering.

\begin{defi}[Structure Group]
There is natural group $\Aut(\blow)$ acting on $\blow$ and $\blow\sqcup\sph^{\dimd-2}$, consisting of transformations of the form 
\[g=\begin{pmatrix}
    \mu^{-1} A & C \\
    0 & \mu^{(\dimd-1)/2} \tilde B
   \end{pmatrix}\]
where: 
\begin{itemize}
\item{$\mu>0$.}
\item{$\tilde B\in \tilde\SL^\pm_2(\R)$.}
\item{$A\in\SL^\pm_{\dimd-1}(\R)$.}
\item{$C\in\mathrm{Mat}_{2,\dimd-1}(\R)$.}
\end{itemize}
We use the convention that $C'\cdot \tilde B':=C'\cdot B'$ for all $\tilde B'\in \tilde\SL^\pm_2(\R)$ and $C'\in\mathrm{Mat}_{2,\dimd-1}(\R)$ where $B'\in \SL^\pm_2(\R)$ is the projection of $\tilde B'$.
\end{defi}

Observe that the projection $\blow\to\tilde\sph^1$ is equivariant with respect to the actions of ${\rm Aut}(\blow)$ and $\tilde{\rm SL}_2(\mb{R})$ and the projection ${\rm Aut}(\blow)\to\tilde{\rm SL}_2(\mb{R})$. 

Using this projection we define sectors and walls:

\begin{defi}[Sector and Wall]
A \emph{sector} of $\blow$ is the preimage by $\blow\rightarrow\tilde\sph^1$ of an interval.
A sector of $\blow\sqcup\sph^{\dimd-2}$ is the union of $\sph^{\dimd-2}$ with a sector of $\blow$.
The \emph{walls} of a sector are the preimages of the endpoints of the corresponding interval of~$\tilde\sph^1$.
\end{defi}

We are now ready to define tubes:

\begin{defi}[Tubes]
\label{def:tubes}
 A \emph{tube} is a space obtained by identifying the walls of a compact sector of $\blow\sqcup\sph^{\dimd-2}$ via an element of $\Aut(\blow)$ of the form
 \[g=\begin{pmatrix}
    \mu^{-1} \Id_{\dimd-1} & C \\
    0 & \mu^{\frac{\dimd-1}{2}} \tilde B
   \end{pmatrix}.\]
 The \emph{singular locus} of the tube is the image of $\sph^{\dimd-2}$ in the quotient, and the \emph{smooth locus} is the complement of the singular locus.
 
 A \emph{meridian} is a generator of the fundamental group of the smooth locus of the tube.
 Note that its holonomy is $g$ or $g^{-1}$.
 
 The \emph{$\tilde\SL_2$-angle} of the tube is the $\tilde\GL_2$-conjugacy class of $\tilde B$.
 \end{defi}

Notice that the smooth locus of a tube $T$ naturally fibers over a projective circle 
\[
T\to C,
\]
whose holonomy is given by the $\tilde\SL_2(\R)$-angle of the tube. The fibers of this bundle are the hyperplanes passing through the singular locus. 

We have the following natural notion of equivalence between tubes:

\begin{defi}[Isomorphisms of Tubes]
Let $T\cup\sph^{\dimd-2}$ and $T'\cup\sph^{\dimd-2}$ be tubes with corresponding fibrations of the smooth loci $T\to N$ and $T'\to N'$. An {\em isomorphism} between the tubes is a homeomorphism $T\cup\sph^{\dimd-2}\to T'\cup\sph^{\dimd-2}$ that induces a diagram
\[
\xymatrix{
T\ar[r]\ar[d] & T'\ar[d]\\
N\ar[r] &N'
}
\] 
where the horizontal arrows are projective isomorphisms.

An {\em automorphism} of a tube $T\cup\sph^{\dimd-2}$ is an isomorphism $T\cup\sph^{\dimd-2}\to T\cup\sph^{\dimd-2}$. We denote the group of automorphisms by ${\rm Aut}(T)$.
\end{defi}

We are now ready to give the definition of projective cone-manifolds.
We will see examples in Section~\ref{sec:proj gluing}, where we show that any gluing of projective manifold with totally geodesic boundary and corners is a projective cone-manifold.

\begin{defi}[Projective Cone Manifold]
\label{def:cone-manifold}
 A \emph{cone-manifold} consists of the following data:
 \begin{itemize}
  \item A topological manifold $M$.
  \item A codimension 2 submanifold $\corner\subset M$ (the \emph{singular locus}).
  \item A choice of a tube $T_C$ for each connected component  $C\subset\corner$.
  \item An atlas of charts $\{\phi_x\}_{x\in M}$ such that if $x$ is in the smooth locus $M\smallsetminus\corner$ then $\phi_x:U_x\subset M\smallsetminus\corner\hookrightarrow\sph^\dimd$, and if $x\in C\subset\corner$ then $\phi_x:U_x\hookrightarrow T_C$ such that $U_x\subset M\smallsetminus(\corner\smallsetminus C)$ and $\phi_x^{-1}(\sph^{\dimd-2})=U_x\cap C$.
  \item If $U_x\cap U_y\neq\emptyset$, then either $x,y\in C\subset\corner$ for some $C$ and the transition map $\phi_y\circ\phi_x^{-1}$ is the restriction of an automorphism of $T_C$, or $U_x\cap U_y\subset M\smallsetminus\corner$ and $\phi_y\circ\phi_x^{-1}$ is projective.
 \end{itemize}
\end{defi}

Note that the singular locus of a cone-manifold has a natural projective structure.

We now classify tubes in terms of their holonomy.

\subsection{Uniformisability and completeness}

Note that in general the projective structure of the smooth locus of a tube is not uniformisable in $\sph^\dimd$.
However it is uniformisable in the universal cover of $\sph^\dimd\smallsetminus\sph^{\dimd-2}$.
This is a consequence of the fact that closed projective manifolds of dimension 1 are uniformisable in the universal cover of $\sph^1$, as explained in \cite[\S3.3.1]{BBS11}, where all projective structures on $\sph^1$ are described.

\begin{prop}\label{prop:uniformization}
Let $T\sqcup\sph^{\dimd-2}$ be a tube, such that the smooth locus $T$ fibers over the projective circle $N$.
Then:
\begin{enumerate}
\item{$T$ and $N$ are uniformisable in $\blow$ and $\tilde\sph^1$ respectively. Choose developing maps with images $\tilde T\subset \blow$ and $\tilde N\subset\tilde\sph^1$ and holonomy generated by $g\in\Aut(\blow)$ fixing $\sph^{\dimd-2}$. Then $T\sqcup\sph^{\dimd-2}$ is the quotient by $g\Z$ of the sector $\tilde T\sqcup\sph^{\dimd-2}$ of $\blow\sqcup\sph^{\dimd-2}$.}
\item{We have a one-to-one correspondence
\[
\begin{array}{r c l}
\{\text{\rm Tubes}\}/\text{\rm isom.} &\simeq &\left\{\begin{array}{c} \mb{Z}<{\rm Fix}_{\Aut(\blow)}(\sph^{\dimd-2})\\ \text{ \rm with non-trivial angle}\\ \end{array}\right\}/\Aut(\blow)\text{\rm -conj}.\\
 & &\\
\left[T\right] &\mapsto &\left[g\mb{Z}\right].\\
\end{array}
\]
}
\end{enumerate}

 Moreover, there are two disjoint cases.
 \begin{itemize}
  \item \emph{$T$ is complete in $\blow$}: $\tilde N=\tilde\sph^1$ and $\tilde T=\blow$.
  \item \emph{$T$ is uniformisable in $\sph^\dimd\smallsetminus\sph^{\dimd-2}$}:
  $\tilde N\subset\tilde\sph^1$
  projects injectively onto an interval $I\subset\sph^1$ between two fixed points of the non-elliptic $\SL_2$-angle $h\in\SL_2(\R)$ (that has positive trace).
  If $h$ is hyperbolic then $I$ is properly convex. If $h$ is parabolic then it is a half-circle.
 \end{itemize}
\end{prop}

\begin{rk}[{Barbot--Bonsante--Schlenker \cite[\S3.3.1]{BBS11}}]\label{rk:uniformization}
 In the setting of Proposition~\ref{prop:uniformization}, let $\tilde h\in\tilde\SL_2(\R)$ be the $\tilde\SL_2$-angle of $g$ with projection $h\in\SL_2(\R)$. Then the following is classical.
 \begin{itemize}
  \item If $h$ is elliptic, \ie $-2<\tr(h)<2$ or $h=\pm\Id_2$, then $\tilde h$ acts freely, properly, and cocompactly on $\tilde S^1$, hence $T$ is complete in $\blow$.
  \item If $\tr(h)\leq -2$ with $h\neq -\Id_2$, then $\tilde h$ acts freely, properly, and cocompactly on $\tilde S^1$, hence $T$ is complete in $\blow$.
  \item If $\tr(h)\geq 2$ with $h\neq \Id_2$, then it has exactly one lift $\tilde h_0\in\tilde\SL_2(\R)$ with fixed points in $\tilde\sph^1$.
  \begin{itemize}
   \item If $\tilde h=\tilde h_0$ then $\tilde h$ acts freely, properly, and cocompactly on any interval of $\tilde \sph^1$ between two fixed points, which projects injectively onto an interval of $\sph^1$, so $T$ is uniformisable in $\sph^\dimd\smallsetminus\sph^{\dimd-2}$.
   \item Otherwise $\tilde h$ acts freely, properly, and cocompactly on $\tilde\sph^1$, hence $T$ is complete in $\blow$.
  \end{itemize}
 \end{itemize}
\end{rk}

\begin{proof}[Proof of Proposition~\ref{prop:uniformization}]
 {\bf Property (1)}. The developing maps of $\tilde T$ and $\tilde N$ fit into the following diagram, which is a morphism of bundles.
 \[
  \xymatrix{
  {\tilde T} \ar[d] \ar[r] & \blow \ar[d] \\
  {\tilde N} \ar[r] & {\tilde\sph^1} \\
  }
 \]
In other words, the map $\tilde T\rightarrow \blow$ is an embedding on each fiber (which is a half-hypersphere of $\sph^\dimd$).
 Thus, uniformisability or completeness of $\tilde T$ in $\blow$  follows from that of $\tilde N$ in $\tilde\sph^1$.
 For the same reasons, uniformisability or completeness of $\tilde T$ in $\sph^{\dimd}\smallsetminus\sph^{\dimd-2}$ follows from that of $\tilde N$ in $\sph^1$. The latter is always uniformisable in $\tilde\sph^1$ since any local homeomorphism from $\R$ to $\R$ is injective.
  
The dichotomy at the end of Proposition~\ref{prop:uniformization} is an immediate consequence of Remark~\ref{rk:uniformization} and the fact that uniformisability or completeness of $\tilde T$ follows from that of $\tilde N$.

{\bf Property (2)}. Finally, let us check the one-to-one correspondence.

\begin{claim}{1}
The map $\Psi$ that associates to the isomorphism class of $T$ the conjugacy class of $g\Z$ is well-defined.
\end{claim}

\begin{proof}[Proof of the claim]
By definition, if two tubes $T\cup\sph^{\dimd-2},T'\cup\sph^{\dimd-2}$ are equivalent, then the isomorphism between them determines projective isomorphisms $\tilde{T}\subset\blow\to\tilde{T}'\subset\blow$ and $\tilde{N}\subset\tilde{\sph}^1\to\tilde{N}'\subset\tilde{\sph}^1$ that fit into the commutative diagram
\[
\xymatrix{
{\tilde{T}}\ar[r]\ar[d] &{\tilde{T}'}\ar[d]\\
{\tilde{N}}\ar[r] &{\tilde{N}'}.
}
\] 

Such isomorphism extends to an automorphism of the bundle $\blow\to\tilde{\sph}^1$: Let us work locally. On a sufficiently small sector, the projective isomorphism $\phi:\tilde{T}\to\tilde{T}'$ descends to a projective isomorphism between sectors of $\sph^\dimd$ centered at $\sph^{\dimd-2}$. On $\sph^\dimd$ any projective isomorphism between open subsets is the restriction of a global projective transformation. In our case, the unique extension of the projective isomorphism between sectors has to fix $\sph^{\dimd-2}$. Hence it lifts to an automorphism of $\blow$ that agrees with $\phi$ on the small sector. By analytic continuation, they agree everywhere. This implies that the holonomies of $T,T'$ are conjugate in ${\rm Aut}(\blow)$. 
\end{proof}

\begin{claim}{2}
The map $\Psi$ is bijective.
\end{claim}

\begin{proof}[Proof of the claim]
We construct an inverse. Let $g\Z<{\rm Fix}_{\Aut(\blow)}(\sph^{\dimd-2})$ be an infinite cyclic subgroup generated by an element $g$ with non-trivial angle $\tilde h\in\tilde{{\rm SL}}_2(\mb{R})$. 

Consider the interval $I=[x,\tilde{h}x]\subset\tilde{\sph}^1$ where $x$ is not a fixed point of $\tilde{h}$ (see Remark \ref{rk:uniformization}). Let $S\subset\blow$ be the corresponding sector. Then the tube $T=S/g$ is mapped to $g\mb{Z}$ by $\Psi$. It is not hard to check that the isomorphism class of $T$ does not depend on the choice of the representative of $\Z$ in the conjugacy class, the generator $g\in\mb{Z}$, and the point $x$.  
\end{proof}

This concludes the proof of the proposition.
\end{proof}

We will need later on (proof of Theorem~\ref{thm:hyperbolic doubles}, Section~\ref{sec:pfTh8.1}) the following technical characterization of tubes that are uniformisable in $\sph^{\dimd}\smallsetminus\sph^{\dimd-2}$.

\begin{lemma}\label{lem:uniformization}
 In the setting of Proposition~\ref{prop:uniformization} and Remark~\ref{rk:uniformization}, $T$ is uniformizable in $\sph^\dimd\smallsetminus\sph^{\dimd-2}$ if and only if $h\neq \Id_2$, $\tr(h)\geq 2$, and $[x,\tilde hx]\subset \tilde N$  projects onto a properly convex segment of $\sph^1$ for some $x\in\tilde N$.
\end{lemma}
\begin{proof}
Remark~\ref{rk:uniformization} implies that if $T$ is uniformisable in $\sph^\dimd\smallsetminus\sph^{\dimd-2}$ then $\tr(h)\geq 2$ and $h\neq \Id_2$ and  any segment of $\tilde N$ projects onto a properly convex segment of $\sph^1$.

 Conversely, suppose $\tr(h)\geq 2$ and $h\neq \Id_2$.
 Identify $\tilde \sph^1$ with $\R$ so that the preimage of a rotation of $\SL_2(\R)$ of angle $\theta$ acts on $\R$ as a translation of the form $\tau^{\theta+2n\pi}(x)=x+\theta+2n\pi$ for some integer $n$.
 Then $\tilde h=\tau^{2n\pi}\tilde h_0$ with $n$ integer and $\tilde h_0$ from Remark~\ref{rk:uniformization}.
 
 Up to conjugation, we may assume $\tilde h_0$ fixes every multiple of $\pi$ in $\R$ (if $h$ is hyperbolic then $\tilde h_0$ has more fixed points), and preserves every interval in between two consecutive multiples of $\pi$.
 
 If  $T$ is not uniformisable in $\sph^\dimd\smallsetminus\sph^{\dimd-2}$, then $n\neq 0$. 
 Say $n\geq 1$. 
 Consider $x\in\R$, say in $[m\pi,(m+1)\pi]$.
 Then $\tilde h(x)\geq \tilde h (m\pi)=(m+2n)\pi\geq x+\pi$ so the projection in $\sph^1$ of $[x,\tilde hx]$ is not properly convex.
\end{proof}

\subsection{Totally geodesic blowup}
Our goal is to replace the singularities of a projective cone-manifold with totally geodesic boundary components. This is not always possible. We describe a family of tubes for which this process can be carried out. 

\begin{defi}[Special Tubes]
\label{def:special tubes}
 Consider a tube with holonomy in $\Aut(\blow)$ generated by
 \[
 g=\begin{pmatrix}
    \mu^{-1} \Id_{\dimd-1} & C \\
    0 & \mu^{(\dimd-1)/2}\tilde B
   \end{pmatrix}
   \]
 with $\mu\ge1$. Let $B\in\SL_2(\R)$ be the projection of $\tilde B\in\tilde\SL_2(\R)$.
 
 The tube is called \emph{special} if
 \begin{itemize}
 \item{It is uniformisable in $\sph^{\dimd}$ (which implies $B$ is hyperbolic or parabolic).}
 \item{All eigenvalues of $\mu^{(\dimd-1)/2}B$ are greater than $\mu^{-1}$.}
 \end{itemize}
  
  The latter is equivalent to 
\begin{equation}\label{eq:special tubes}
 \mu^{\dimd+1}-\mu^{(\dimd+1)/2}\tr B+1>0.
\end{equation}
\end{defi}

Note that the above implies that $\mu>1$. Hence, among the two generators of the holonomy of the tube only one satisfies this property. We call such generator the \emph{special generator}.

Before going on, let us make the following simple observation.

\begin{rk}\label{rk:attracting wall}
Notice that for any $g\in\Aut(\blow)$, if $g$ preserve a sector $S$ of $\blow$, then it preserves \emph{each} of its walls.
In other words, the endpoints of the interval $I$ in $\tilde \sph^1$ defining the sector are fixed points of $g$.

If the sector $S$ is small enough, more precisely, if $I$ projects (injectively) to a convex proper interval of $\sph^1$ (we allow half-circles), then $g$ acts properly discontinuously on the interior of $I$, hence one of the endpoints $H^+$ is \emph{attracting in $I$} ($g^nH\to H^+$ for any $H\in I$ and $n\to\infty$) and the other is \emph{repelling in $I$} ($g^nH\to H^-$ for any $H\in I$ and $n\to-\infty$). Note that if $I$ is a half-circle then $H^+$ is not an attracting fixed point in $\sph^1$.

We use the same terminology for the corresponding walls.
\end{rk}

The next proposition describes the blowup of a special tube.

\begin{prop}\label{prop:totgeod blowup}
 Consider a special tube $T\sqcup\sph^{\dimd-2}$ with smooth locus $T$ and special generator $g$, with the notation from Definition~\ref{def:special tubes}.
 By Proposition~\ref{prop:uniformization}, $T\sqcup\sph^{\dimd-2}$ is a quotient by $g\Z$ of a sector ${\tilde T}\sqcup\sph^{\dimd-2}\subset \sph^\dimd$. 
 
 Let $H^+,H^-$ be the $g$-invariant attracting and repelling walls of $\tilde T$ (see Remark~\ref{rk:attracting wall}).
 Note that $H^+$ contains a fixed point $x^+\in\sph^\dimd$ of $g$, corresponding to the highest eigenvalue of $g$. Then:
 \begin{enumerate}
 \item{$g\Z$ acts freely and properly discontinuously on  
 \[
 {\tilde T}\sqcup \left(H^+\smallsetminus(\sph^{\dimd-2}\cup\{x^+\})\right).
 \]
 }
 \item{The quotient $\bar T$ is a compact projective manifold with totally geodesic boundary with interior isomorphic to the smooth locus $T$, such that the isomorphism extends to a continuous onto map $\bar T\rightarrow T\sqcup\sph^{\dimd-2}$.}
 \item{Every isomorphism of tubes $T\sqcup \sph^{\dimd-2}\to T'\sqcup\sph^{\dimd-2}$ extends to an isomorphism of projective manifolds with totally geodesic boundary $\bar T\to\bar T'$.}
 \end{enumerate}
\end{prop}

\begin{proof}
 {\bf Property (1)}. Set
 \[
 \bar S:={\tilde T}\sqcup \left(H^+\smallsetminus(\sph^{\dimd-2}\cup\{x^+\})\right).
 \]
 
Notice that the assumption on the holonomy of a special tube implies that there is a change of basis fixing $\mb{R}^{\dimd-1}$ where $g$ is represented by a matrix of the form
 \[
 g=\begin{pmatrix}
 \mu^{-1} \Id_{\dimd-1} & &\\
  &\mu^{(d-1)/2}\lambda & \\
  & &\mu^{(d-1)/2}/\lambda\\
 \end{pmatrix}
 \text{ \rm or }
 \begin{pmatrix}
 \mu^{-1} \Id_{\dimd-1} & &\\
  &\mu^{(d-1)/2} &1\\
  & &\mu^{(d-1)/2}\\
 \end{pmatrix} 
 \]
 where $\lambda>1/\lambda>1/\mu$ or $\mu>1$. The first case occurs when $B$ is hyperbolic while the second case occurs when $B$ is parabolic. Note that $x^+=[e_d]$ and $H^+$ is a half-hypersphere in ${\rm Span}\{e_1,\cdots,e_d\}$.
 
We now show that the action is properly discontinuous on $\bar{S}$.

In the case $B$ is hyperbolic, it is an elementary computation to show that the action of $g$ on the space $\sph^\dimd\smallsetminus(\sph({\rm Span}\{e_1,\cdots,e_{d-1},e_{d+1}\})\cup\{\pm[e_d]\})$ is properly discontinuous.

We give a more general argument that works also in the parabolic case: Consider a sequence $x_n\in\tilde T$ converging to $x\in \bar{S}$ and a sequence of integers $k_n$ going to $k=\pm\infty$.
We have to prove that $g^{k_n}x_n$ leaves every compact set of $\bar{S}$.
There are four cases, depending whether $k=+\infty$ or $-\infty$, and whether $x$ lies in $\tilde T$ or in $H^+$.
 
Let $H_n\in\sph^1$ be the projection of $x_n$, seen as half-hyperspheres containing $\sph^{\dimd-2}$, that converge to $H$.
Let $C\subset\sph^\dimd$ be the $g$-invariant $1$-sphere complementary to $\sph^{\dimd-2}$.
\begin{enumerate}
\item If $k=\infty$ and $x\in H^+$ then one checks that $g^{k_n}H_n$ converges to $H^+$ and $g^{k_n}x_n$ accumulates on $H^+\cap C=\{x^+\}$.
\item If $k=\infty$ and $x\in \tilde T$ then $g^{k_n}x_n$ converges to $x^+$.
\item \label{item:bordification} If $k=-\infty$ and $x\in H^+$ then $g^{k_n}x_n$ converges to the stereographic projection of $x$ on $\sph^{\dimd-2}$ seen from $x^+$.
\item If $k=-\infty$ and $x\in\tilde T$ then $g^{k_n}H_n$ converges to $H^-$ and $g^{k_n}x_n$ accumulates on $H^-$.
\end{enumerate}
This ends the proof of properness of the action.

 {\bf Property (2)}. The projection $\bar T\rightarrow T\sqcup\sph^{\dimd-2}$ is constructed by descending the map $\bar{S}\rightarrow{\tilde T}\sqcup\sph^{\dimd-2}$ that sends every $x\in H^+$ to its stereographic projection on $\sph^{\dimd-2}$ seen from $x^+$.
 To see that the resulting map is continuous, one may check that for every sequence $x_n\in\tilde T$ converging to $x\in H^+\smallsetminus (\sph^{\dimd-2}\cup\{x^+\})$, with projection $H_n\in\sph^1$, the sequence of integers $k_n$ such that $g^{k_n}$ brings back $x_n$ in a compact fundamental domain of ${\tilde T}\sqcup\sph^{\dimd-2}$ converges to $-\infty$, and then apply point {\color{red} \eqref{item:bordification}} above.
 
{\bf Property (3)}. Finally, consider an isomorphism of tubes $T\sqcup \sph^{\dimd-2}\to T'\sqcup\sph^{\dimd-2}$.
 As we saw in the proof of Proposition~\ref{prop:uniformization}, the induced projective isomorphism $\tilde T\subset\blow\to\tilde T'\subset\blow$ is the restriction of an element $\phi\in\Aut(\blow)$.
 To conclude, it is not difficult to check that $\phi$ maps the attracting wall of $\tilde T$ to the attracting wall of $\tilde T'$.
\end{proof}

\begin{defi}[Totally Geodesic Blowup]
\label{def:totgeod blowup}
 Using the notation from Proposition~\ref{prop:totgeod blowup}, $\pi:\bar T\rightarrow T\cup\sph^{\dimd-2}$ is called the \emph{totally geodesic blowup} of $T\cup\sph^{\dimd-2}$.
 
 The \emph{totally geodesic blowup} of a cone-manifold with special singularities is obtained by locally blowing-up each component of the singular locus in  charts in tubes such that the change of charts are automorphisms of tubes.
 This process is independent of the chart by invariance of the blowup under automorphisms in Proposition~\ref{prop:totgeod blowup}.
\end{defi}

\section{Tessellations of convex domains}
\label{sec:sec3}

The key result of this section is Proposition~\ref{prop:injetcvx}.
It provides conditions which imply (uniformisability and) convexity for a union of convex tiles.

We will see in Section~\ref{sec:combidescription} that this result can be used to establish uniformisability and convexity for gluings of properly convex projective manifolds.
The reason is that the universal cover of a gluing of properly convex projective manifolds can be described as a gluing of universal covers of these convex projective manifolds.

This phenomenon is illustrated in Section~\ref{sec:proofs} and Figures~\ref{fig:figure3b} and \ref{fig:figure3c}.

\subsection{A local-to-global convexity statement}\label{sec:cvxstatement}

Let us now state the key result of this section, which is based on an idea of Vinberg \cite[\S3]{Vin71}, revisited by Benoist \cite[\S1.5]{Be5lectures}.

\begin{prop}\label{prop:injetcvx}
 Let $\cal G$ be a connected graph with vertex set $\cal V$ and edge set $\cal E$. 
 Suppose that to each vertex $v\in\cal V$ is associated  a closed convex subset $D_v\subset\sph^\dimd$ with non-empty interior. 
 Assume that: 
 \begin{enumerate}
  \item \label{item:hypi1} $F_e:=D_v\cap D_w$ is a codimension $1$ face of $D_v$ and $D_w$ for any edge $e=[v,w]\in \cal E$, such that every point of $\partial F_e$ is contained in the boundary of a half-space of $\sph^{\dimd}$ which contains $D_v\cup D_w$.
  \item \label{item:hypi2} For any path $p=(v_1,v_2,\dots,v_n)\subset \cal G$ such that $F_p:=D_{v_1}\cap\dots\cap D_{v_n}$ is a codimension $2$ face of $D_{v_i}$ for every $i$, there exists a half-space of $\sph^{\dimd}$ which contains $D_{v_1}\sqcup\dots\sqcup D_{v_n}$ and whose boundary contains $F_p$.
 \end{enumerate}
 Let $\sim$ be the smallest equivalence relation on $\bigsqcup_{v\in\cal V}D_v\times\{v\}$ such that $(x,v)\sim (y,w)$ whenever $v$ and $w$ are adjacent and $x=y$. Denote by $X$ the quotient space. 
 Then the natural projection map 
 $$\pi : X \rightarrow \sph^{\dimd}$$
 is injective with convex image.
\end{prop}

The proof of Proposition~\ref{prop:injetcvx} is postponed to Section~\ref{sec:cvxalavinberg}.

\begin{cor}
 In the setting of Proposition~\ref{prop:injetcvx}, the graph $\cal G$ is a tree.
\end{cor}

\begin{proof}
Suppose by contradiction that there is a non-trivial loop $(v_1,\dots,v_n)\subset\cal G$ with $v_n=v_1$.
Let $\cal G'$ be the graph with vertices $1,\dots,n$ and with an edge from $i$ to $i+1$ for every $1\leq i<n$.
Associate to every vertex $i$ the convex set $D_i:=D_{v_i}$.
One checks that $\mc{G}'$ satisfies again the assumptions of Proposition~\ref{prop:injetcvx}.
However, $\bigsqcup_{1\leq i\leq n}D_i\times\{i\}/_\sim\rightarrow\sph^\dimd$ is clearly not injective since since any $x$ in ${\rm int}(D_{1})={\rm int}(D_{n})$ is the image of both $(x,1)$ and $(x,n)$, which are not equivalent for $\sim$.
\end{proof}

\begin{defi}[Gluing Kit]
\label{def:gluing kit}
 The data of $(\cal G, (D_v)_{v\in\cal V})$ satisfying Conditions~\ref{item:hypi1} and \ref{item:hypi2} of Proposition~\ref{prop:injetcvx} is called a \emph{gluing kit}.
 To ease the notation, we will identify $X=\bigsqcup_{v\in\cal V}D_v\times\{v\}/_\sim$ with its image under $\pi$.
 Moreover, Proposition~\ref{prop:injetcvx} implies that the gluing kit is determined by the family of convex sets $\{D_v\}_{v\in\cal V}$.
 As a consequence, we will identifies each vertex $v\in \cal V$ with its associated convex set $D_v$.
\end{defi}

\begin{defi}[Cells, Walls, Strata]
The convex sets of the form $D_v$ with $v\in\cal V$ (\resp $D_v\cap D_w$ with $v,w\in\cal V$ adjacent, \resp $\cap_{v\in A}D_v$ with $A\subset\cal V$ with size at least $2$) are called \emph{cells} (\resp \emph{walls}, \resp \emph{strata}).
A stratum with codimension $2$ is called a \emph{corner}.
\end{defi}

\begin{rk}\label{rem:succession}
 In the proof of Proposition~\ref{prop:injetcvx}, it will be proved that for any geodesic $(D_1,\dots,D_n)\subset\cal G$ and any pair $(x,y)\in D_1\times D_n$, there is a segment $[x,y]$ that \emph{goes successively through $D_1,\dots, D_n$}, in the sense that there exists $x_1,\dots,x_{n-1}\in[x,y]$ in this order such that $x_i\in D_i\cap D_{i+1}$ for any $1\leq i\leq n$.
\end{rk}

\begin{rk}\label{rk:union of two convex sets}
 Note that Proposition~\ref{prop:injetcvx} has the following corollary.
 Consider two closed convex sets $D,D'\subset\sph^\dimd$ that intersect on a codimension $1$ face $F=D\cap D'$.
 Then $D\cup D'$ is convex if and only if
 every point of $\partial F$ is contained in the boundary of a half-space of $\sph^{\dimd}$ which contains $D\cup D'$.
\end{rk}

\subsection{The proof of Proposition~\ref{prop:injetcvx}}\label{sec:cvxalavinberg}

In this section we explain the argument of Vinberg that, informally speaking, local convexity along codimension 1 and codimension 2 faces are enough to show the convexity of a union of convex tiles.

\begin{proof}[Proof of Proposition~\ref{prop:injetcvx}]
Let us call \emph{segment} a subset $s\subset X$ such that the restriction of $\pi$ to $s$ is injective and $\pi(s)$ is a segment of $\sph^{\dimd}$. 
In order to prove that $\pi$ is injective with convex image, it is enough to show that every pair of points of $X$ can be joined by a segment.

For all $v\in \cal V$ and $x\in D_v$ (\resp $A\subset D_v$), let us denote by $x|v$ (\resp $A|v$) the projection in $X$ of $(x,v)$ (\resp $A\times\{v\}$).

Let $X'\subset X$ be the projection of $\bigsqcup_{v\in\cal V}{\rm int}(D_v)\times\{v\}\sqcup\bigsqcup_{e\in\cal E}{\rm int}(F_e)\times e$. If $X$ is endowed with the quotient topology, then 

\begin{fact}
\label{obs:injetcvx:localhomeo}
$\pi$ is a continuous map, $X'$ is an open subset, and the restriction of $\pi$ to $X'$ is a local homeomorphism.
\end{fact}

Denote  $F_p:=D_{v_1}\cap\dots\cap D_{v_n}$ for any path $p=(v_1,v_2,\dots,v_n)\subset \cal G$.

Fix a path $p=(v_1,v_2,\dots,v_n)\subset\cal G$ and set $D_i:=D_{v_i}$ for every $i$.
Denote by $S_p$ the set of pairs of points $(x,y)\in D_{1}\times D_{n}$ such that $x|v_1$ and $y|v_n$ can be joined by a segment of $X$ going successively through $D_{1}|v_1,D_{2}|v_2,\dots,D_{n}|v_n$ (see Remark~\ref{rem:succession}).

\begin{fact}
\label{obs:injetcvx:Sclosed}
 $S_p$ is a closed subset of $D_{1}\times D_{n}$.
\end{fact}

Denote by $G_p$ the set of pairs of points $(x,y)\in {\rm int}(D_{1})\times{\rm int}(D_{n})$ such that $x\neq -y$ and $\Span(x,y)\cap F_q=\emptyset$ for any subpath $q\subset p$ such that $F_q$ has codimension at least $3$.

\begin{fact}
\label{obs:injetcvx:Gopendenseconn}
 $G_p$ is an open, dense and connected subset of ${\rm int}(D_{1})\times{\rm int}(D_{n})$.
\end{fact}

Now comes the main ingredient, which makes use of the assumptions~\eqref{item:hypi1} and \eqref{item:hypi2}. Denote $F_i=D_{i}\cap D_{{i+1}}$ for any $1\leq i<n$.

\begin{lemma}\label{lem:injetcvx}
 For any $(x,y)\in G_p$, any segment of $X$ from $x|v_1$ to $y|v_n$ and going successively through $D_{1}|v_1,\dots,D_{n}|v_n$ is contained in 
 \[
 {\rm int}(D_{1})|v_1\cup {\rm int}(F_1)|v_1\cup {\rm int}(D_{2})|v_2\cup\dots\cup{\rm int}(F_{n-1})|v_{n-1}\cup{\rm int}(D_{n})|v_n.
 \]
\end{lemma}

\begin{proof}
 Let $s$ be a segment from $x|v_1$ to $y|v_n$ and going successively through $D_{1}|v_1,\dots,D_{n}|v_n$.
 Pick $z_1|v_1,\dots, z_{n-1}|v_{n-1}\in s$ in this order such that $z_i|v_i=z_i|v_{i+1}\in F_i|v_i$ for any $1\leq i<n$.
 It is enough to show that $z_i\in{\rm int}(F_i)$ for any $1\leq i<n$. 
 Let us assume by contradiction that $z_i\in\partial F_i$ for some $i$.
 We may assume that $z_j\in{\rm int}(F_j)$ for any $1\leq j<i$.
 Pick $a\in ]z_{i-1},z_i[\subset{\rm int}(D_{i})$, where we may need the notation $z_{0}:=x$. Set also $z_n:=y$.
 
If $z_i\neq z_{i+1}$, then by Assumption~\eqref{item:hypi1} we may find a closed half-space $H\subset \sph^{\dimd}$ which contains $D_{i}\cup D_{{i+1}}$ and whose boundary contains $z_i$. We have $z_{i+1}\in H$ and $a\in{\rm int}(H)$ since $a\in{\rm int}(D_{i})$. Hence $z_i\in]a,z_{i+1}[\subset{\rm int}(H)$, which is a contradiction.
 
 Therefore $z_i=z_{i+1}$. Let $i+1<j\leq n$ be such that $z_i=z_{i+1}=\dots=z_{j-1}\neq z_j$. Set $q=(v_i,v_{i+1},\dots,v_{j-1})\subset p$. 
 By definition, $z_i\in F_q$, which has codimension $2$ since $(x,y)\in G_p$. 
 By Assumption~\eqref{item:hypi2}, there exists a closed half-space $H\subset\sph^{\dimd}$ which contains $D_{i}\cup\dots\cup D_{{j-1}}$ and whose boundary contains $z_i$. We have $z_j\in D_{{j-1}}\subset H$ and $a\in{\rm int}(D_{i})\subset{\rm int}(H)$, hence $z_i\in]a,z_j[\subset{\rm int}(H)$, which is a contradiction.
\end{proof}

Facts~\ref{obs:injetcvx:localhomeo}, \ref{obs:injetcvx:Sclosed} and \ref{obs:injetcvx:Gopendenseconn}, and Lemma~\ref{lem:injetcvx}  have the following consequence.
 
\begin{cor}\label{cor:injetcvx}
 $G_p\cap S_p$ is clopen in $G_p$. Thus, if it is non-empty then $S_p=D_{1}\times D_{n}$.
\end{cor}

Let us prove that if $G_p\cap S_p$ is non-empty, then so is $G_q\cap S_q$ for any extension $q=(p,v_{n+1})$ such that $v_{n+1}\neq v_{n-1}$.
 
Set $F_n:=D_{n}\cap D_{{n+1}}$.
For dimensions reasons, we can find $(x,y)\in {\rm int}(D_{1})\times{\rm int}(F_n)$ such that $x\neq -y$ and $\Span(x,y)\cap F_{q'}=\emptyset$ for any subpath $q'\subset q$ such that $F_{q'}$ has codimension at least $3$, and $\Span(x,y)$ is transverse to $F_n$.
Then any segment from $x|v_1$ to $y|v_n$ going successively through $D_{1}|v_1,\dots,D_{n}|v_n$ (it exists by Corollary~\ref{cor:injetcvx}) may be extended to a point $y'|v_{n+1}\in {\rm int}(D_{{n+1}})|v_{n+1}$ such that $x\neq-y'$.
One checks that $(x,y')\in S_q\cap G_q$, which is therefore non-empty. 

\end{proof}

\subsection{The geometry of $X$}
Let us fix for the whole section a gluing kit $\cal G=(\cal V,\cal E)$ with $X=\bigcup_{D\in\cal V}D$.  

We now describe the geometry of the convex set $X$. In particular, we relate the stratification of the boundary to the shape of the cells. As it turns out, we have three distinct types of faces:
\begin{enumerate}[(i)]
\item{Some faces of the tiles are also faces of $X$.}
\item{Limits of {\em fans} of tiles sharing a common codimension 2 face.}
\item{Limits of {\em telescopes} of tiles containing infinitely many pairwise disjoint walls.}
\end{enumerate}
We will define them later on in this section and we will analyze in detail the last two types under additional assumptions of periodicity, satisfied for example if the tiling is preserved by a group $\Gamma$ acting cocompactly on $X$.

\subsubsection{The interior of $X$}
Let us start with the interior of $X$.

\begin{prop}\label{prop:intX}
 The interior of $X$ is
 \[
 \bigcup_{v\in\cal V}{\rm int}(D_v) \cup \bigcup_{e\in\cal E}{\rm int}(F_e).
 \]
\end{prop}

\begin{proof}
It is clear if $\cal G$ is finite. 
In the general case write $\cal G$ as the union of an increasing sequence of finite subgraphs $\cal G_n$, and set $X_n:=\bigcup_{v\in\cal V_n}D_v$ for each $n$. As ${\rm int}(X)=\bigcup_n{\rm int}(X_n)$, this finishes the proof.
\end{proof}

\begin{cor}
 Any segment of $\bar X$ that intersects non-trivially a wall has to be contained in it.
\end{cor}
\begin{proof}
 Consider a segment $s$ intersecting non-trivially a closed wall $W$.
 It is clear that $s\subset \Span(W)$.
 By Proposition~\ref{prop:intX}, $\partial W\subset\partial X$ so $\Span(W)\cap\bar X=W$, hence $s\subset W$.
\end{proof}

We now turn to the boundary $\partial X$.

\subsubsection{The boundary of the cells in $\partial X$}
We first show that some of the faces of $X$ correspond to faces of the tiles.

\begin{lemma}\label{lem:faces of cells}
 Consider a cell $D\in\cal V$. Then 
 \begin{itemize}
 \item{$F\cap D$ is a closed face of $D$ for any closed face $F$ of $\bar X$.}
 \item{If $F$ is a closed face of $D$ that does not intersect any closed wall, then it is also a closed face of $\bar X$.}
 \end{itemize}
\end{lemma}
\begin{proof}
 The first statement is clear.
 
 Consider a closed face $F$ of $D$ that does not intersect any closed wall.
 Let $F'$ be the closed face of $\bar X$ containing $F$. Note that $F'\cap D=F$.
 Suppose by contradiction that there is $x\in F'\smallsetminus F$.
 Pick $p\in{\rm int}(D)$.
 By Proposition~\ref{prop:intX}, $[p,x[$ meets $\partial D$ in the relative interior of some closed wall $W$ of $D$.
 The non-empty subset $A\subset F'$ consisting of points $y$ such that $[p,y]$ meets $W$ is clearly closed and disjoint from $F$.
 By connectedness of $F'$, in order to finish the proof it is enough to check that $A$ is open.
 In particular, it suffices to check that $[p,y]$ meets ${\rm int}(W)$ for any $y\in A$.
 If this was not the case then $[p,y]$ would meet $\partial W$ at $y$ itself by Proposition~\ref{prop:intX}, but then $y$ would lie in $D$ and, hence, in $W\cap F$ which is absurd.
\end{proof}

\begin{cor}\label{cor:extremal points}
 Consider two distinct adjacent cells $D,D'\in\cal V$.
 \begin{itemize}
 \item{If $x$ is an extremal point of $\partial D$ outside closed walls then it is also extremal in $\partial X$.}
 \item{If $x\in D\cap D'$ is an extremal point of $D\cup D'$ outside closed strata then it is also extremal in $\partial X$.}
 \end{itemize}
\end{cor}

\begin{proof}
 The first statement is a direct application of Lemma~\ref{lem:faces of cells}.
 
 The second statement is an application of the first one.
 Indeed the tree $\cal G'$ obtained as a quotient of $\cal G$ by replacing $D$ and $D'$ by the single vertex $D\cup D'$ also yields a gluing kit, for which $D\cap D'$ is not a wall.
\end{proof}

Finally in order to detect $\cal C^1$ points of $\partial X\smallsetminus X$, we can use the following.

\begin{fact}\label{obs:smooth}
 Consider two distinct adjacent cells $D,D'\in\cal V$.
 \begin{itemize}
 \item{If $x$ is a $\cal C^1$ point of $\partial D$ outside closed walls then it is also $\cal C^1$ in $\partial X$.}
 \item{If $x\in D\cap D'$ is a $\cal C^1$ point of $\partial(D\cup D')$ then it is also $\cal C^1$ in $\partial X$.}
 \end{itemize}
\end{fact}

\begin{proof}
 This is an immediate consequence of the well-known observation that for any two nested properly convex open sets $\Omega\subset\Omega'$ and $x\in\partial\Omega\cap\partial\Omega'$ if $x$ is $\cal C^1$ in $\partial\Omega$ then it is also $\cal C^1$ in $\partial\Omega'$.
\end{proof}

\subsubsection{The complement in $\partial X$ of the boundary of the cells}
Next, we consider the portion of the boundary that does not come from faces of the tiles. It naturally comes from accumulation points of sequences of tiles which we conveniently organize as points on the boundary at infinity of the graph $\mc{G}$. 

Recall that $\cal G$ is a tree, so it has a natural boundary at infinity $\partial_\infty\cal V$: The set of rays up to equivalence, two rays being equivalent if they agree after some time. 
We set $\bar{\cal V}:=\cal V\sqcup\partial_\infty\cal V$, a space on which we have a natural topology and notion of convexity.

Each point $\xi\in\partial_\infty\mc{V}$ corresponds to a {\em cell at infinity}.

\begin{lemma}
 Consider a ray $\xi=\{D_n\}_{n\in\mb{N}}\in\partial_\infty\cal V$. Then the sequences of compact convex sets $D_{n}$ and $D_n\cap D_{n+1}$ both converge to the same closed convex subset $D(\xi)\subset\partial{X}$, called a \emph{cell at infinity}.
\end{lemma}
\begin{proof}
 Let $K$ be the set of accumulation points of sequences $x_n\in D_n$.
 We need to check that any point of $K$ is also the limit of a sequence $y_n\in D_n\cap D_{n+1}$. Consider $x\in K$, limit of $x_k$ where $x_k\in{\rm int}(D_{n_k})$ with $n_k$ increasing.
 Consider for each $k$ the intersection points $y_{n_k},\dots,y_{n_{k+1}-1}$ of the segment $[x_k,x_{k+1}]$ with successively $D_{n_k}\cap D_{n_k+1},\dots,D_{n_{k+1}-1}\cap D_{n_{k+1}}$.
 The sequence $y_n$ converges to $x$, as required.
\end{proof}

\begin{rk}
As a warning:
\begin{itemize}
\item{It is possible that $D(\xi)\subset X$.}
\item{It is possible that $D(\xi)=D(\eta)$ for different $\xi,\eta\in\partial_\infty\mc{V}$.}
\item{It is also possible that $D(\xi)$ is not a face of $X$.}
\end{itemize}
\end{rk}

\begin{prop}\label{prop:cells infinity}
 We have:
 \begin{enumerate}
  \item \label{item:cells infinity 1} For every point $x\in\bar X\smallsetminus X$ there exists a unique $\xi\in\partial_\infty\cal V$ such that $x\in D(\xi)$.
  \item \label{item:cells infinity 2} Consider a geodesic ray $\xi=\{D_n\}{n\in\mb{N}}\in\partial_\infty\cal V$, then there exists $N$ such that
  \[
  D(\xi)\cap X=\bigcap_{n\geq N}D_{n}.
  \]
  Moreover $D(\xi)\cap X$ is a closed face of $D(\xi)$.
 \end{enumerate}
\end{prop}

\begin{proof}
 {\bf Property (1)}.
 Fix $D\in\cal V$.
 Pick $x\in\bar X\smallsetminus X$ and a countable basis of (decreasing) properly convex neighborhoods $U_n$ of $x$ in $\sph^\dimd$.
 For each $n$ let $\cal V_n\subset\cal V$ be the set of all vertices that intersect $U_n$.
 Since $U_n$ is properly convex, $\cal V_n$ is a convex subset of $\cal V$.
 Let $D_n$ be the shortest point projection in $\cal G$ of $D$ on $\cal V_n$.
 Since $\cal V_n$ is non-increasing, $D_n$ must lie on a geodesic ray of $\cal G$ starting at $D$.
 Choose $x_n\in D_n\cap U_n$ for each $n$. Then $x_n$ tends to $x$.
 The fact that $x\not\in X$ implies that $D_n$ is not constant after some time and, hence, the geodesic ray containing $D_n$ is infinite and represents  some $\xi\in\partial_\infty\cal V$ whose cell $D(\xi)$ contains $x$.
 Moreover, if $D'_k$ accumulates on $x$ then $D'_k\in \cal V_{n_k}$ for some diverging sequence $n_k$, and $D'_k$ converges to $\xi$.
 
 {\bf Property (2)}.
 The sequence of sets $\bigcap_{n\geq N}D_{n}$ is non-decreasing for inclusion and constant after some time. Hence, it suffices to check that any $x\in D(\xi)\cap X$ lies in $\bigcap_{n\geq N}D_n$ for some $N$.
 Pick $D\in \cal V$ containing $x$ and $x_n\in D_n$ converging to $x$.
 There is $N$ such that $D_m$ lies on the geodesic segment from $D$ to $D_n$ for every $n\geq m\geq N$.
 For all $n\geq m\geq N$ pick $y_{n,m}\in[x,x_n]\cap D_m$.
 Observe that $y_{n,m}\in D_m$ converges to $x$ for every $m$, so $x\in D_m$.
 
 Suppose by contradiction that $D(\xi)\cap X$ is not a face of $D(\xi)$, \ie that there are $x,y\in D(\xi)\smallsetminus X$ and $z\in]x,y[\cap D_n$ for some $n$.
 Pick $p\in {\rm int}(D_n)$. 
 By Remark~\ref{rem:succession}, $[p,q]$ meets $D_n\cap D_{n+1}$ for any $q\in\bigcup_{k>n}D_k$.
 Passing to the limit, $[p,x]$ (\resp $[p,y]$) meets $D_n\cap D_{n+1}$ at $x'\neq x$ (\resp $y'\neq y$).
 By Proposition~\ref{prop:intX}, these two intersection points lie in the relative interior of $D_n\cap D_{n+1}$.
 The span of $D_n\cap D_{n+1}$ intersects the plane spanned by $p,x,y$ into a line which contains $x',y',z$. 
 But these three points are not collinear, which is a contradiction.
\end{proof}

As we have already observed, sequences of tiles can essentially be grouped into two different types: {\em Fans} of tiles and {\em telescopes} of tiles. We now analyze in detail both cases.

\subsection{Fan of tiles}

Informally speaking, a fan of tiles is an infinite sequence of consecutive tiles sharing a common codimension 2 face. We will work under the assumption that the sequence is periodic under some special projective transformation and describe for such a fan of tiles the corresponding cell at infinity.  See Figure~\ref{fig:figure2c}.

\subsubsection{Condition~\ref{item:hypi2} revisited}

First, we discuss a convexity criterion for a periodic fan of tiles. In order to do so, we investigate the Condition~\eqref{item:hypi2} in Proposition~\ref{prop:injetcvx} in this special setting.

\begin{prop}\label{prop:hypi2 revisited}
 Let $(D_n)_{n\in\Z}$ be a sequence of closed convex subsets of $\sph^{\dimd}$ with non-empty interior such that:
 \begin{enumerate}[(a)]
  \item They all share a codimension $2$ face $E$ spanning $\sph^{\dimd-2}\subset\sph^\dimd$.
  \item Any two consecutive domains $D_n$ and $D_{n+1}$ intersect along a codimension 1 face $F_n$, with $F_n\neq F_{n-1}$.
 \end{enumerate}
 Consider for every $n$ the projections $D'_n$ of $D_n$ to the circle $\sph(\R^{\dimd+1}/\R^{\dimd-1})=\sph^1$. Then $\sph^{\dimd-2}$ is contained in the boundary of a half-space of $\sph^{\dimd}$ containing $\bigcup_nD_n$ if and only if $\Omega:=\bigcup_n D_n'$ is contained in a half-circle.
 
Assume now that there exists $g\in\GL_{\dimd+1}(\R)$ and $T>0$ such that $gD_{n+T}=D_n$ for every $n$. Let $g'\in\GL(\R^{\dimd+1}/\R^{\dimd-1})$ be the induced transformation on $\sph^1$. Then $\Omega$ is contained in a half-circle if and only if
 \begin{enumerate}[(i)]
  \item $|\det(g')|^{-1/2}g'\neq\mb{I}_2$ and $\tr(|\det(g')|^{-1/2}g')\geq 2$.
  \item $D'_0\cup\dots\cup D'_{T-1}$ does not intersect any eigenline of $g'$.
 \end{enumerate}
\end{prop}

\begin{proof}
 The fact that $\sph^{\dimd-2}$ is contained in the boundary of a half-space of $\sph^{\dimd}$ containing $\bigcup_nD_n$ if and only if $\Omega':=\bigcup_n D'_n$ is contained in a half-circle is clear.

Assume that $\Omega'$ is contained in a half-circle (it is a proper, open and convex subset of $\sph^1$).

Let us prove properties (i)+(ii). As $\Omega'$ is $g'$-invariant and no transformation $h\in\SL_2(\R)$ with $\tr(h)<2$ preserves a proper convex subset of $\sph^1$, we must have $\tr(|\det(g')|^{-1/2}g')\geq 2$. As $g'D'_n=D'_{n+T}$, we also have $|\det(g')|^{-1/2}g'\neq\mb{I}_2$ and $\Omega'$ does not contain fixed the points of $g'$, that is, its eigenline(s).
 
Assume now that properties (i),(ii) hold.

Let us prove that $\Omega'$ is contained in a half-circle. As $\tr(|\det(g')|^{-1/2}g')\geq 2$, the eigenlines of $g'$ determine either a half-circle or a unique convex segment. As $\Omega'$ does not intersect any eigenline of $g'$ and is connected, it must be contained in one the components of $\sph^1$ minus the eigenline(s) of $g'$. 
\end{proof}

\begin{rk}
 In the setting of Proposition~\ref{prop:hypi2 revisited}, if $\Omega$ is contained in a half-circle then:
  \begin{itemize}
  \item If $\tr(|\det(g')|^{-1/2}g')=2$, then $\Omega$ is one of the two half-circles delimited by the unique eigenline of $g'$.
  \item If $\tr(|\det(g')|^{-1/2}g')>2$, then $\Omega$ is one of the four quadrants delimited by the two eigenlines of $g'$.
 \end{itemize}
\end{rk}

\subsubsection{The cell at infinity of a fan of tiles}

Let $\cal G=(\cal V,\cal E)$ be a gluing kit with $X=\bigcup_{D\in\cal V}D$.

Pick $\mu>1$ and $\lambda\geq 1$ and set, depending whether $\lambda>1$ or $\lambda=1$,
\[g:=
 \begin{pmatrix}
  \mu^{-1}\Id_{\dimd-1} & 0 & 0 \\
  0 & \lambda\mu^{\frac{\dimd-1}2} & 0 \\
  0 & 0 & \lambda^{-1}\mu^{\frac{\dimd-1}2}
 \end{pmatrix}
 \ \text{or}\ 
 \begin{pmatrix}
  \mu^{-1}\Id_{\dimd-1} & 0 & 0 \\
  0 & \mu^{\frac{\dimd-1}2} & 1 \\
  0 & 0 & \mu^{\frac{\dimd-1}2}
 \end{pmatrix}.
\]

The following describes the cell at infinity of a fan of tiles:

\begin{prop}\label{prop:aroundcodim2}
 Let $\cal S$ be a corner.
 Suppose that: 
 \begin{itemize}
  \item $\cal S\subset \sph^{\dimd-2}$.
  \item The set of $D\in\cal V$ such that $\cal S\subset D$ is a bi-infinite path $(D_n)_{n\in\Z}$.
  \item $g D_{n}=D_{{n+m}}$ for any $n$, for some $m>0$.
 \end{itemize}
 Set $\xi:=\{D_n\}_{n\in\mb{N}}\in\partial_\infty\cal V$. 
 Then there exists $\alpha,\beta\in\{\pm1\}$ such that:
 \begin{enumerate}
  \item \label{item:aroundcodim2 cvxhull} $D(\xi)$ is the convex hull of $\cal S$ with $\alpha e_{\dimd}\R_+$.
  \item \label{item:aroundcodim2 face} If $X$ is properly convex or $\lambda>1$, then $D(\xi)$ is a closed face of $X$.
  \item \label{item:aroundcodim2 containment} If $\lambda>1$ then $X$ is contained in the convex hull of $D(\xi)$ and $\beta e_{\dimd+1}\R_+$.
 \end{enumerate}
\end{prop}

\begin{proof}
{\bf Property (1)}. For any generic $x\in\sph^\dimd$, the sequence $g^nx$ tends to $e_{\dimd}\R_+$ or $-e_{\dimd}\R_+$.
 We may assume that the limit is $e_{\dimd}\R_+$ for some $x\in{\rm int}(D_0)$.
 This implies that the convex hull of $\cal S$ with $e_{\dimd}\R_+$ is contained in $D(\xi)$, and hence that $\sph^{\dimd-1}$ is a supporting hyperplane for $X$ containing $D(\xi)$. 
 
By Part (2) of Proposition~\ref{prop:cells infinity}, $\cal S\subset\sph^{\dimd-2}$ is a closed face of $D(\xi)$. Hence $\sph^{\dimd-2}\subset\sph^{\dimd-1}$ is a supporting hyperplane of $D(\xi)$, which lies in the closed half-space $H$ of $\sph^{\dimd-1}$ delimited by $\sph^{\dimd-2}$ that contains $e_{\dimd}\R_+$
 
 For any $x\in D(\xi)\smallsetminus\{e_{\dimd}\R_+\}$, the limit $y$ of $g^{-n}x$ is in $\sph^{\dimd-2}\cap D(\xi)=\cal S$ (as $D(\xi)$ is clearly $g$-invariant), and $x\in[y,e_{\dimd}\R_+]$.
 Therefore $D(\xi)$ is the convex hull of $\cal S$ with $e_{\dimd}\R_+$.
 
 {\bf Property (2)}. Suppose that there exists $x\in \sph^{\dimd-1}\smallsetminus D(\xi)$, lying in some cell $D(v)$ with $v\in\bar{\cal V}\smallsetminus\{\xi\}$.
 Let us show that $X$ is not properly convex and $\lambda=1$.
 Set $\eta=\{D_{-n}\}_{n\in\mb{N}}\in\partial_\infty\cal V$ and $X':=\bigcup_nD_n$, so that $\bar X'=X'\cup D(\xi)\cup D(\eta)$.
 Pick $y\in{\rm int}({\cal S})$.
 
 We may assume that $x\in \bar X'$.
 Indeed, if it not the case then consider the shortest point projection $D_k$ of $v$ on $\{D_n\}_{n\in\Z}$ in $\bar{\cal V}$, and  a geodesic path $(v_m)_{0\leq m\leq M}$ from $D_k$ to $v$ in $\cal V$, with $M\in\N\cup\{\infty\}$.
 Then the last point of $[y,x]$ in $D_k$ lies in $D_k\cap D(v_1)$, which does not intersect $D(\xi)$.
 We may then consider this last point instead of $x$.
 
 By Proposition~\ref{prop:intX}, $x$ does not lie in $\sph^{\dimd-2}$.
 Moreover, $x$ does not lie in ${\rm int}(H)$, otherwise $[x,y]$ would intersect $D(\xi)\smallsetminus X$, which is not possible since the complementary of $D(\xi)\smallsetminus X$ in $\bar X$ is convex.
 
 Hence $g^nx$ tends to $-e_{\dimd}\R_+$, which hence belongs to $\bar X'$ since this last set is $g$-invariant.
 Thus $X$ is not properly convex.
 Moreover $\lambda=1$ by Proposition~\ref{prop:hypi2 revisited}.
 
 {\bf Property (3)}. Finally, we may assume that $\bar X$ is contained in the half-space of $\sph^\dimd$ delimited by $\sph^{\dimd-1}$ and containing $e_{\dimd+1}\R_+$.
 Let us prove that, under the condition that $\lambda>1$, then $\bar X$ is contained in the convex hull $K$ of $D(\xi)$ and $e_{\dimd+1}\R_+$.
 Suppose by contradiction that there exists $x\in \bar X\smallsetminus K$.
 First note that $x\not\in\bar X'$ by Lemma~\ref{lem:NSdynandcontainment} below.
 
 As before, consider $v\in\bar{\cal V}$ such that $x\in D(v)$, the shortest point projection $D_k$ of $v$ on $\{D_n\}_{n\in\Z}$ in $\bar{\cal V}$, and  a geodesic path $(v_m)_{0\leq m\leq M}$ from $D_k$ to $v$ in $\cal V$, with $M\in\N\cup\{\infty\}$.
 One may find $y\in \partial D(\xi)\smallsetminus\partial \cal S$ such that $[x,y[$ does not intersect $K$.
 The last point $z$ of $[y,x]$ in $\bar X'$ belongs to $D_k\cap D(v_1)$ which does not intersect $\partial D(\xi)\smallsetminus\partial \cal S$.
 Thus $z\neq y$, so $z\in]y,x]$, so $z\in \bar X'\smallsetminus K$ which is absurd.
\end{proof}

\begin{lemma}\label{lem:NSdynandcontainment}
 Consider  $0<\lambda_1<\lambda_2<\lambda_3$ and
 $$g=
 \begin{pmatrix}
  \lambda_1 I_{\dimd-1} & 0 & 0 \\
  0 & \lambda_3 & 0 \\
  0 & 0 & \lambda_2
 \end{pmatrix}.$$
 Let $\Omega\subset\sph^{\dimd}$ be a $g$-invariant properly convex open set. 
 Then its projection on $\sph^{\dimd-1}$ parallel to $[e_{\dimd+1}]$ is contained in $\overline\Omega\cap\sph^{\dimd-1}$, which is the convex hull of $\partial\Omega\cap \sph^{\dimd-2}$ and $[e_{\dimd}]$.
\end{lemma}
\begin{proof}
 Consider $x=[x_1,\dots,x_{\dimd+1}]\in\overline\Omega$.
 If $x$ lies in $\Omega$ or $\sph^{\dimd-1}\smallsetminus\{[e_{\dimd}]\}$, then $x_i\neq 0$ for some $1\leq i\leq \dimd-1$.
 Thus $g^{-n}x$ converges in this case to $[x_1,\dots,x_{\dimd-1},0,0]$ when $n$ tends to infinity, and this limit must belong to $\partial\Omega\cap \sph^{\dimd-2}$.
 This implies that $[x_1,\dots,x_\dimd,0]$ lies in the convex hull of $\partial\Omega\cap \sph^{\dimd-2}$ and $[e_{\dimd}]$.  
\end{proof}

\subsection{Telescope of tiles}
\label{sec:contraction}

We now move to the second type of sequences which we call telescopes of tiles. Informally speaking, a telescope of tiles is a sequence of consecutive tiles $D_n$ such that infinitely many of the walls $D_m\cap D_{m+1}$ are pairwise disjoint.

Again, we will assume some periodicity assumptions, namely, that, up to projective transformations, there are only finitely (or compactly) many configurations of four consecutive tiles $D_n,D_{n+1},D_{n+2},D_{n+3}$ where the walls $D_n\cap D_{n+1}$ and $D_{n+1}\cap D_{n+2}$ are disjoint.

In this case $(D_n\cap D_{n+1},\Omega,D_{n+2}\cap D_{n+3})$ form what we call a visible triple. 

\begin{defi}[Visible Triple]
A \emph{visible triple} is a triple $(L,K,M)$ of (non-empty) compact properly convex sets such that $L,M\subset K$, $L\cap M=\emptyset$, and any segment of $K$ from $L$ to $M$ intersects the interior of $K$.
\end{defi}

The analysis of the cell at infinity of a telescope of tiles rests on a contraction property of visible triples which we now discuss.

\subsubsection{Contraction for visible triples}
The set of visible triples is endowed with a Hausdorff topology. A sequence of compact subsets $C_n\subset\sph^\dimd$ converges to a compact subset $C\subset\sph^\dimd$ if the following two conditions are satisfied:
\begin{itemize}
\item{If a sequence $x_{n_j}\in C_{n_j}$ converges to $x\in\sph^\dimd$, then $x\in C$.}
\item{For every $x\in C$ there exists a sequence $x_n\in C_n$ converging to it.}
\end{itemize}

This induces a Hausdorff topology on the space of triples in a natural way.

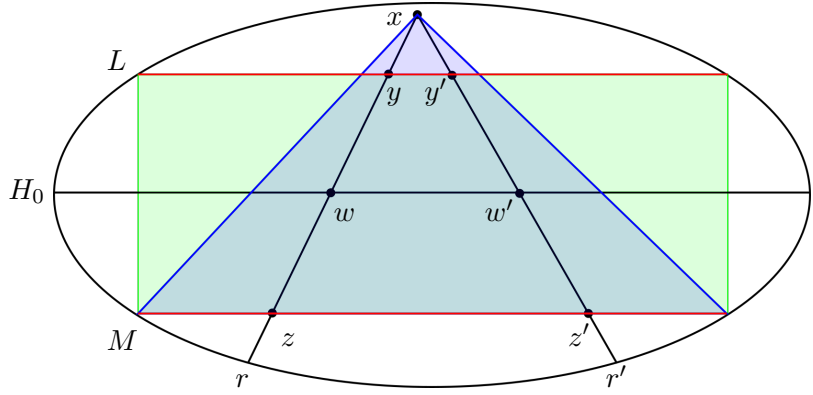
\begin{figure}[h]
\centering
\begin{overpic}{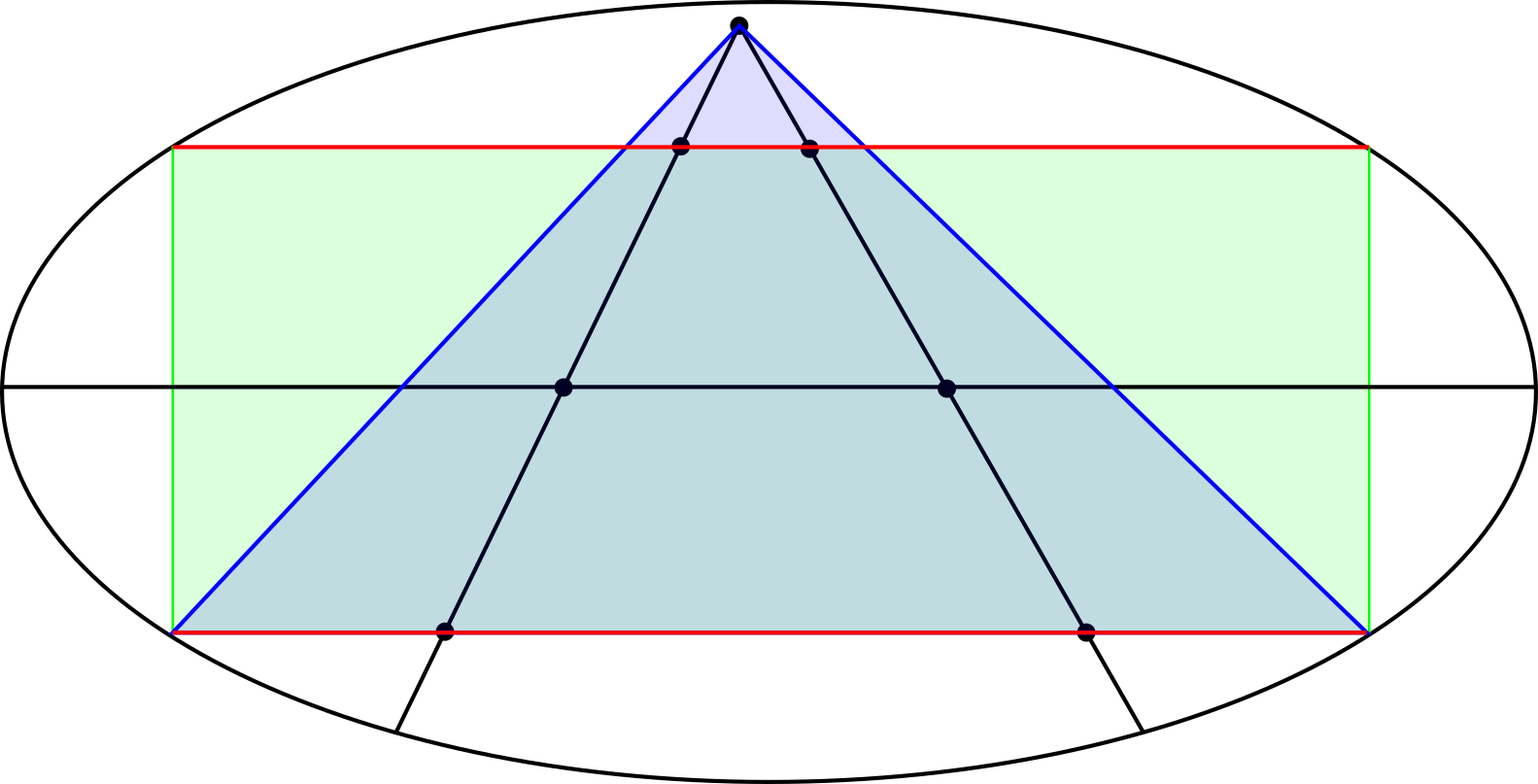}
\put (44,48) {$x$}
\put (44,38) {$y$}
\put (49,38) {$y'$}
\put (37,22) {$w$}
\put (57,22) {$w'$}
\put (30,5.5) {$z$}
\put (68,5.5) {$z'$}
\put (24,0) {$r$}
\put (73,0) {$r'$}
\put (7,42) {$L$}
\put (7,5) {$M$}
\put (-6,25) {$H_0$}

\end{overpic}
\caption{Contraction for visible triples.}
\label{fig:visible3}
\end{figure}

\begin{lemma}\label{lem:contraction}
 Let $\cal K$ be a compact set of visible triples.
 Then there exists $\lambda>1$ such that the following holds.
 Consider a properly convex open set $\Omega$, a point $x\in\Omega$, two rays $r,r'$ of $\Omega$ starting at $x$, two hyperplanes $H,H'$ and a closed convex subset $K\subset\bar\Omega$ such that:
 \begin{itemize}
  \item $(H\cap \bar\Omega,K,H'\cap \bar\Omega)\in\GL_{\dimd+1}(\R)\cdot\cal K$.
  \item The ray $r$ (\resp $r'$) meets first $H\cap \Omega$ at $y$ (\resp $y'$) and then $H'\cap\Omega$ at $z$ (\resp $z'$).
 \end{itemize}
 Then $d_\Omega(z,z')\geq \lambda d_\Omega(y,y')$.
\end{lemma}
\begin{proof}
 We may assume that $\cal K$ is a very small neighborhood of a given triple $(L,K,M)$.
 Fix an affine chart $\cal A$ containing $K$, with a Euclidean distance $d_{\cal A}$. 
 We may assume that it contains $K'$ for every $(L',K',M')\in\cal K$.
 
 Let $H_0\subset\sph^\dimd$ by a hyperplane separating $L$ and $M$.
 We may assume that it separates $L'$ and $M'$ for every $(L',K',M')\in\cal K$.
 
 Let $\epsilon>0$ be the minimal distance for $d_{\cal A}$ between a point of $H_0\cap \conv(L',M')$ and a point of $H_0\cap \partial K'$ for $(L',K',M')\in\cal K$, and $D$ the maximal possible $d_{\cal A}$-diameter for $K'$ .
 By \cite{Birkhoff_perronfrob} (see also \cite[Fact 2.1.3]{ThesePL}), there exists $\lambda>1$ depending on $\epsilon$ and $D$ such that for any $(L',K',M')\in\cal K$, denoting by $\Omega_1$ (\resp $\Omega_2$) the relative interior of $H_0\cap \conv(L',M')$ (\resp $H_0\cap K'$), for all $x,y\in\Omega_1$, we have
 $$d_{\Omega_1}(x,y)\geq \lambda d_{\Omega_2}(x,y).$$
 
 Consider a properly convex open set $\Omega$, a point $x\in\Omega$, two rays $r,r'$ of $\Omega$ starting at $x$, two hyperplanes $H,H'$ and a closed convex subset $K'\subset\bar\Omega$ such that, denoting $L':=H\cap \bar\Omega$ and $M':=H'\cap \bar\Omega$
 \begin{itemize}
  \item $(L',K',M')\in\GL_{\dimd+1}(\R)\cdot\cal K$.
  \item The ray $r$ (\resp $r'$) meets first ${\rm int}(L')$ at $y$ (\resp $y'$) and then ${\rm int}(M')$ at $z$ (\resp $z'$).
 \end{itemize}
 Denote by $\Omega_0$ (\resp $\Omega_1$, \resp $\Omega_2$, \resp $\Omega_3$) the relative interior of $H_0\cap \conv(x,M')$ (\resp $H_0\cap \conv(L',M')$, \resp $H_0\cap K'$, \resp $H_0\cap \Omega$).
 Observe that 
 $$\Omega_0\subset\Omega_1\subset\Omega_2\subset\Omega_3.$$

 Therefore, denoting by $w$ and $w'$ the stereographic projections in $\Omega_0$ of $z$ and $z'$ seen from $x$, we obtain by Remark~\ref{rem:hilbert metric} that
 \begin{align*}
 d_{\Omega}(z,z')=d_{\Omega_0}(w,w')\geq d_{\Omega_1}(w,w')&\geq \lambda d_{\Omega_2}(w,w')\\ &\geq \lambda d_{\Omega_3}(w,w') \geq \lambda d_{\Omega}(x,x').\qedhere
 \end{align*}
\end{proof}

\subsubsection{Visible triples in $X$}

Let $X=\bigcup_{v\in\mc{V}}{D_v}$ be a tessellated properly convex set as above. We now describe where to find visible triples in it.

Assume that each tile $D_v$ has the following properties:
\begin{enumerate}
 \item \label{item:find contract1} Any intersection of (at least $2$) walls is either empty or a corner.
 \item \label{item:find contract2} Any closed proper face touching a wall is contained in a (possibly different) wall.
 \item \label{item:find contract3} Any closed proper face with codimension at least $2$ that touches a corner is contained in a corner.
\end{enumerate}

We have the following:

\begin{figure}[h]
\centering
\begin{overpic}{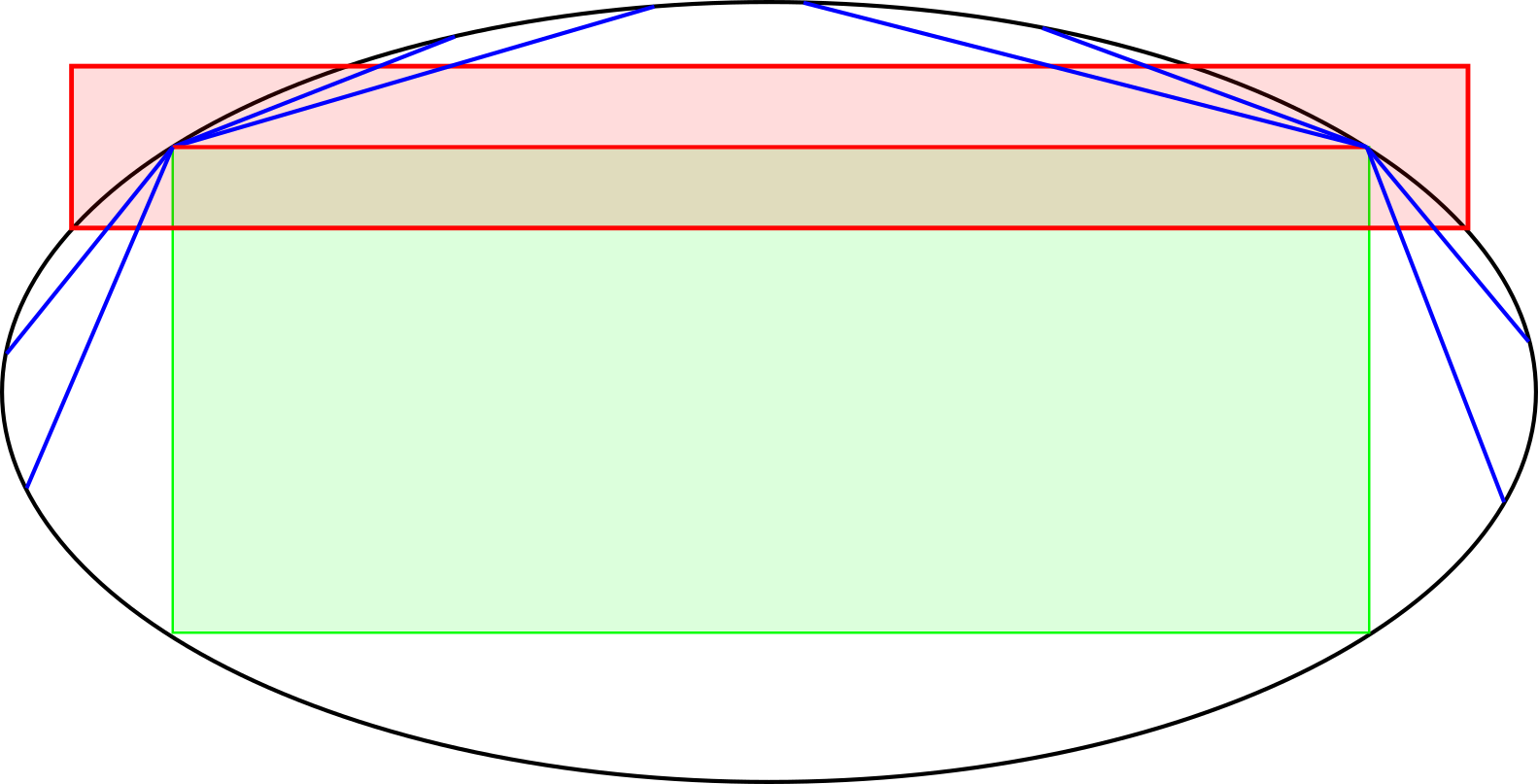}
\put (0,43) {$U$}
\put (42,23) {$D\smallsetminus U$}
\put (98,15) {$W_j$}
\put (45,43) {$W$}

\end{overpic}
\caption{Visible triples in $X$.}
\label{fig:visible1}
\end{figure}
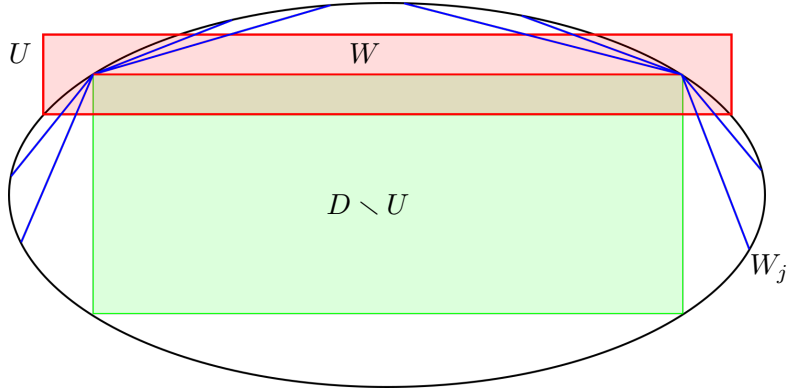

\begin{lemma}\label{lem:finding visible triples}
 Let $D\subset X$ be a tile. Let $W$ be a wall of $D$. Consider an open neighborhood $U$ of $W$ in $\sph^\dimd$ such that $D\smallsetminus U$ is convex. Then: 
 \begin{itemize}
 \item{There are only finitely many walls $W_1,\dots,W_n$ adjacent to $W$ and not contained in $U$.}
 \item{$(W,K,D\smallsetminus U)$ is a visible triple for every convex compact set $K\subset\bar{X}$ with ${\rm int}(W_j)\subset{\rm int}(K)$ for every $j\le n$.}
 \end{itemize}
\end{lemma}

\begin{proof}
 Only a finite number of walls $W_1,\dots, W_n$ are adjacent to $W$ but not contained in $U$.
 Indeed, if there was an infinite sequence of them then they would accumulate on a closed codimension $\geq 2$ face $F$ touching $W$ but not contained in $U$, which by assumption would be contained in a wall $W'$ adjacent to $W$. The fact that $F$ touches the corner $W\cap W'$ would then imply that it is contained in it, which is absurd.
 
 To finish the proof we only need to prove that any segment $s$ from $W$ to $D\smallsetminus U$ intersects ${\rm int}(D)$ or ${\rm int}(W_i)$ for some $i$.
 If $s\subset\partial D$, then by assumption it is contained in a wall adjacent to $W$, which is not contained in $U$. So $s\subset W_i$ for some $i$.
 If, by contradiction, $s\subset\partial W_i$ then it is contained in a face of $D$ of codimension at least $2$ which touches $W\cap W_i$, hence is contained in it. This is absurd.
\end{proof}

We will also need the following technical lemma.

\begin{lemma}\label{lem:walls far away}
 For any neighborhood $U$ of $W$ there exists a smaller neighborhood $U'$ such that $D\smallsetminus U'$ contains all walls not contained in $U$ and non-adjacent to $W$.
\end{lemma}

\begin{proof}
 If by contradiction this was not the case, then there would exist an infinite sequence of walls limiting on a comdimension-$\geq 2$ face $F$ intersecting $W$ but not contained in $U$. This is absurd, as remarked in the proof of Lemma~\ref{lem:finding visible triples}.
\end{proof}

\subsubsection{The extremal point of a telescope of tiles}

 Let $\cal G=(\cal V,\cal E)$ be a gluing kit with $X=\bigcup_{D\in\cal V}D$.

\begin{prop}\label{prop:contracting cells at infinity}
Suppose that all strata have codimension 2, and that there exists a compact set $\cal K$ of visible triples such that the following holds: 
For any path $\{A,B,C,D\}\subset\cal V$ such that ${A\cap B}$ and ${C\cap D}$ are disjoint, there exists  a convex compact subset $K\subset \bar X$ such that
\[
(A\cap B,K,C\cap D)\in\GL_{\dimd+1}(\R)\cdot\cal K.
\]
Then any cell at infinity that does not intersect $X$ is an extremal $\cal C^1$ point.
\end{prop}

\begin{proof}
 Consider the ray $\xi=\{D_n\}_{n\in\mb{N}}\in\partial_\infty\cal V$, such that $D(\xi)\cap X=\emptyset$.
 
 Let us prove that $D(\xi)$ is reduced to a point. 
 
 Suppose by contradiction that it is not the case. Pick $x\in{\rm int}(D_0)$ and $p,p'\in{\rm int}(D(\xi))$, and denote by $r$ and $r'$ the rays of $\Omega:={\rm int}(X)$ starting at $x$ and ending respectively at $p$ and $p'$.
 Set $F_n^k:=D_n\cap\dots \cap D_{n+k}$ for all $n,k$.
 Let $y_n$ (\resp $y'_n$) be the successive intersection points of $r$ (\resp $r'$) with ${\rm int}(F_n^1)$.
 
 \begin{claim}{1}
 The sequence $d_\Omega(y_n,y'_n)$ is non-decreasing and bounded from above by $d_{{\rm int}(D(\xi))}(p,p')<\infty$.
 \end{claim}
 
\begin{proof}[Proof of the claim]
Let us take care of monotonicity first: Consider the 2-plane $P$ spanned by $r,r'$ and the intersection $X\cap P$. Let $[a_n,a_n']$ be the intersection of the projective line spanned by $y_n,y_n'$ with $X$. Notice that we have $F_n^1\cap P=[a_n,a_n']$ and, by construction, $F_m^1\cap P=[a_m,a_m']$ is contained in the cone with vertex $x$ bounded by the lines $[x,a_n]$ and $[x,a_n']$. The intersection of the line spanned by $[a_m,a_m']$ with such cone is a larger segment $[b_m,b_m']$. By standard properties of the the cross ratio, we get $[b_m,y_m,y_m',b_m']=[a_n,y_n,y_n',a_n']$. By monotonicity under inclusion (see Definition~\ref{rem:hilbert metric}), we also get $[b_m,y_m,y_m',b_m']\le[a_m,y_m,y_m',a_m']$. This concludes the monotonicity part of the statement.
 
 The upper bound follows from the semicontinuity of the extension of the Hilbert distance to the boundary (see Definition~\ref{rem:hilbert metric}).
 \end{proof}

 Let $\lambda>1$ be the constant from Lemma~\ref{lem:contraction}. In order to get a contradiction, it is enough to prove the following:

 \begin{claim}{2}
 For any $n$ there exists $m\geq n$ such that $d_\Omega(y_m,y'_m)\geq \lambda d_\Omega(y_n,y'_n)$.
 \end{claim}
 
 \begin{proof}[Proof of the claim]
 
 Pick $n\in\N$.
 Since $D(\xi)$ does not intersect $X$, there exists $k\geq 2$ such that $F_{n+1}^k$ is empty. Take it minimal, meaning that $F_{n+1}^{k-1}\neq\emptyset$.
 
 If $k=2$, then $F_n^3=\emptyset$. Set $m=n$.
 
 If $k\geq 3$, then $F_{n+1}^{k-1}=F_{n+2}^{k-2}=\dots=F_{n+k-2}^2$ (since all strata have codimension $2$), and $F_{n+k-2}^3=F_{n+1}^{k}=\emptyset$. Set $m=n+k-2$.

 Let $\cal V'$ be the component of $\cal V\smallsetminus\{D_{m},D_{m+3}\}$ containing $D_{m+1}$.
 By assumption, there exists  a properly convex compact set $K\subset \bigcup\cal V'$ containing $D_{m+1}\cup D_{m+2}$ such that
 $$ (D_{m}\cap D_{m+1},K,D_{m+2}\cap D_{m+3})\in\GL_{\dimd+1}(\R)\cdot\cal K.$$
 By Lemma~\ref{lem:contraction}, we get
 \[
  d_\Omega(y_{m+2},y'_{m+2})\geq \lambda d_\Omega(y_{m},y'_{m})\geq \lambda d_\Omega(y_n,y'_n).
 \]
 \end{proof} 
 
\begin{claim}{3}
The point $D(\xi)$ is extremal.
\end{claim}

\begin{proof}[Proof of the claim]
This is an immediate consequence of the fact that $\bar X\smallsetminus D(\xi)$ is convex.
\end{proof}
 
 Saying that $p$ is $\cal C^1$ is equivalent (see \eg \cite[Prop.\,5.4.8]{ThesePL}) to saying that 
 for any two rays $r,r'$ of $\Omega$ ending at $p$, the distance $\inf\{d_\Omega(x,x'):(x,x')\in r\times r'\}$ is null,
 which can be proved using almost the same proof as above. 
\end{proof}

\section{Convex-cocompactness and relative hyperbolicity}
\label{sec:sec4}

The goal of this section 3-fold: We are going to establish three geometric properties of the convex projective manifolds (potentially with totally geodesic boundary) that we construct in this paper, namely:
\begin{enumerate}[(a)] 
\item{Convex-cocompactness.}
\item{Relative hyperbolicity.}
\item{Extended geometric finiteness.}
\end{enumerate}

We now state the results. Later we will give precise definitions of their terms. 

The first property we establish is convex-cocompactness.
The following result is probably well-known to experts, as the proof follows a classical strategy.
Recall that for us the terminology b-manifolds denotes projective manifolds with totally geodesic boundary.

\begin{prop}\label{prop:cvxcocpct}
 Let $M$ be a compact properly convex b-manifold with universal cover $\tilde M\subset\sph^\dimd$.
 Suppose that every proper face of $\tilde M$ touching the closure of a wall belongs to it.
 Then $\pi_1(M)$ acts \emph{convex-cocompactly} on any invariant properly convex open set containing $\tilde M$.
\end{prop}

\begin{rk}\label{rem:not cvxcocpct}
 In the setting of Proposition~\ref{prop:cvxcocpct},
 there does not always exists an invariant properly convex open set containing $\tilde M$.
 For instance,  if $M$ is the manifold $\Omega_2/\Gamma_2$ obtained in Theorem~\ref{thm:main2}, then $\pi_1(M)=\Gamma_2$ does not act naively convex-cocompactly on any properly convex open set $\Omega$ (see the definition of naive convex-cocompactness in Section~\ref{sec:cvxcocpct}).
\end{rk}

The second property we consider is relative hyperbolicity.
There are two cases, depending on whether we consider a manifold with or without boundary.

\begin{thm}\label{thm:relhypb}
 Let $M$ be a compact properly convex b-manifold with universal cover $\tilde M\subset\sph^\dimd$.
 Suppose that:
 \begin{itemize}
  \item All walls of $\tilde M$ have pairwise disjoint closures.
  \item Any point of $\partial \tilde M$ outside closed walls is extremal.
 \end{itemize}
 Then $\pi_1(M)$ is \emph{hyperbolic relatively to} the fundamental groups of its walls,
 and the \emph{Bowditch boundary} is equivariantly homeomorphic to the quotient $Y:=\partial \tilde M/_\sim$ obtained by contracting closed walls into points.
\end{thm}

Relative hyperbolicity is a group theoretic generalization of the fundamental groups of geometrically finite hyperbolic manifolds, in which case the Bowditch boundary corresponds to the usual limit set (see Section~\ref{sec:yaman} for more details).

The above theorem can be seen as a consequence of very general work of Weisman \cite[Th.\,1.16]{We21} or work of Islam and Zimmer \cite[Th.\,1.3-6]{IZ22relhypb}, {in the special case where $\pi_1(M)$ is convex-cocompact}. 
However, notice that we will apply this theorem in a case where convex-cocompactness does not hold (see the above Remark~\ref{rem:not cvxcocpct}). In fact we will exploit a different property: Namely that our subgroups with respect to which we want to prove relative hyperbolicity have codimension 1.
The strategy of proof of Theorem~\ref{thm:relhypb} follows the same line as in Weisman and Islam--Zimmer's works, which themself are inspired by older works.

In the case where $M$ is closed, relative hyperbolicity comes from the following:

\begin{fact}[{Weisman \cite[Th.\,1.16]{We21}}]\label{fact:Teddyrelhypb}
 Let $M=\Omega/\Gamma$ be a closed properly convex projective manifold.
 Consider closed, totally geodesic, embedded, pairwise disjoint hypersurfaces $W_j=C_j/\Gamma_j\subset M$ where $C_j=\Omega\cap H_j$ with $H_j$ projective hyperplanes, $\Gamma_j=\Stab_\Gamma(H_j)$, and $j=1\dots n$.
 Suppose that:
 \begin{itemize}
  \item The $\Gamma$-translates of the $C_j$'s in $\Omega$ have pairwise disjoint closures.
  \item Any point of $\partial \Omega$ outside the closures of these translates is extremal.
 \end{itemize}
 Then $\Gamma$ is hyperbolic relatively to $\Gamma_1,\dots,\Gamma_n$,
 and the Bowditch boundary is equivariantly homeomorphic to the quotient $Y:=\partial \Omega/_\sim$ obtained by contracting $\gamma\partial C_j$ for every $\gamma\in\Gamma$ and $j=1\dots n$.
\end{fact}

The last property we consider is extended geometrical finiteness, a recent notion due to Weisman \cite[Def.\,1.3]{We22}.
Extended geometrically finite discrete subgroups of semisimple Lie groups generalize the class of geometrically finite discrete subgroups of rank-one semisimple Lie groups such as $\SO(n,1)$.
In particular, all extended geometrically finite groups are by definition relatively hyperbolic with respect to some proper subgroups.

Let $\mathcal F$ be the partial flag variety of pairs of points and hyperplanes $(x,H)$ of $\mb{RP}^d$ such that $x\in H$.

\begin{prop}
 In the settings of Theorem~\ref{thm:relhypb} and Fact~\ref{fact:Teddyrelhypb}, the action of $\Gamma$ on $\cal F$ is \emph{extended geometrically finite} in the sense of \cite[Def.\,1.3]{We22}.
\end{prop}

Extended geometrical finiteness in the setting of Fact~\ref{fact:Teddyrelhypb} and in the setting of Theorem~\ref{thm:relhypb} \emph{with convex-cocompactness} was already known in \cite[Th.\,1.12]{We22}.

\subsection{Convex-cocompactness}\label{sec:cvxcocpct}

Consider a properly convex open set $\Omega\subset\sph^\dimd$ and a discrete group $\Gamma\subset\SL^\pm_{\dimd+1}(\R)$ preserving it.

The action of $\Gamma$ on $\Omega$ is said to be \emph{naively convex-cocompact} if there exists a $\Gamma$-invariant convex subset of $\Omega$ on which $\Gamma$ acts cocompactly.
In such a case, and if $\Omega$ is not strictly convex, there is not always a smallest such convex set, contrarily to convex-cocompactness in hyperbolic geometry.
One reason is that the set of accumulation of points of an orbit $\Gamma\cdot x$ may depend on the choice of $x\in\Omega$ (\eg if $\Omega$ is a triangle in $\sph^2$, and $\Gamma$ is generated by an infinite-order non-proximal element).
For this reason, and others, explained in \cite{DGK17}, it appears that this notion of convex-cocompactness is not the best one; this is why another one has been introduced, as follows.

The \emph{full orbital limit set} of $\Gamma$ is the union over all $x\in\Omega$ of the set of accumulation points of the orbit $\Gamma\cdot x$, \ie $\cup_{x\in\Omega}\overline{\Gamma x}\cap\partial\Omega$.

\begin{defi} [{Danciger--Guéritaud--Kassel \cite[Def.\,1.11]{DGK17}}]\label{def:cvxcocpct} 
The action of $\Gamma$ on $\Omega$ is said to be \emph{convex-cocompact} if the convex hull in $\Omega$ of the full orbital limit set is non-empty and has compact quotient by $\Gamma$.
\end{defi}

We refer to \cite[\S1.4--1.6--1.7--4.1--10.7]{DGK17} and \cite{DGKLM} for more details and examples on convex-cocompactness.

\begin{proof}[Proof of Proposition~\ref{prop:cvxcocpct}]
 Let $\Omega$ be an invariant  properly convex open set containing $\tilde M$ and $\Gamma=\pi_1(M)$.
 Let $\Lambda$ be the full orbital limit set.
 
 The convex set $\tilde M$ is the union of its interior with the relative interior of countably many of its codimension $1$ faces (its totally geodesic boundary).
 Hence $\partial_i\tilde M=(\partial\tilde M)\smallsetminus\tilde M$ is compact, and the closure of $\tilde M$ is the convex hull of $\partial_i\tilde M$.
 Moreover, one easily checks by compactness of $M$ that $\partial_i\tilde M\subset\Lambda\subset\partial\Omega$.

 To prove the proposition it is now enough to prove that any $\xi\in\Lambda$ lies in $\partial_i\tilde M$.
 Assume by contradiction that $\xi\not\in\partial_i\tilde M$.
 There exist $x\in\Omega$ and a sequence $\gamma_n\in\subset\Gamma$ such that $\gamma_nx$ converges to $\xi$.

 Pick $y$ in the interior of $\tilde M$.
 Consider $a,b\in\partial\Omega$ and $z$ in a wall of $\tilde M$ such that 
 $a,y,z,x,b$
 are aligned in this order.
 Up to extracting a subsequence, let us assume that 
 $$\gamma_n(a,y,z,x,b)\underset{n\to\infty}{\longrightarrow}(\alpha,\eta,\zeta,\xi,\beta)\in \partial\Omega\times\partial_i\tilde M\times\partial_i\tilde M\times\partial\Omega\times\partial\Omega.$$
 
 Since $\xi\not\in\partial_i\tilde M$, one then checks that $\alpha,\eta,\zeta,\xi,\beta$ are pairwise distinct.
 
 Let $K$ be the closure of $\tilde M$ in $\sph^\dimd$ and $f:K\rightarrow\partial\tilde M$ map each $p$ to the entry point of $[\xi,p]$ in $K$.
 By convexity, this map is continuous.
 Moreover, $f(p)$ belongs to a wall of $\tilde M$ for every $p$ in the interior of $\tilde M$;
 by continuity and connectivity they all belong to the same wall $W$.
 By continuity again $f(\eta)\in\partial W\cap [\zeta,\xi]$.
 
 The segment $[\eta,f(\eta)]$ is non-trivial, contained in $\partial\tilde M$ and touches $\bar W$, hence it must be contained in $\bar W$. 
 Thus $\xi\in\Span W$.
 
 This is absurd: $\Span(\xi,y)\cap\Span W=\{\xi\}$ whence $f(y)\not\in W$.
\end{proof}

\subsection{Yaman's characterization of relative hyperbolicity}\label{sec:yaman}

We now give our working definition of relative hyperbolicity, following Yaman.
This is not the original definition, which is due to Bowditch \cite{Bow12}.
There are many definitions in the literature; the following one is equivalent to Bowditch's definition.

\begin{defi}[Relative Hyperbolicity {\cite[Th.\,0.1]{Yam04}}]\label{def:yaman}
Let $Y$ be a compact metrisable space and $\Gamma$ a discrete group acting by homeomorphisms on $Y$.
\begin{itemize}
 \item The action is called a \emph{convergence action} if $\Gamma$ acts properly on the set of distinct triples of $Y$.
 \item A point $p\in Y$ is called \emph{conical} if there exists a sequence $\gamma_n\in\Gamma$ and $a\neq b\in Y$ such that $\gamma_nx\to a$ and $\gamma_ny\to b$ for any $y\neq x$.
 \item A point $p\in Y$ is called \emph{parabolic} if $\Stab_\Gamma(p)$ is infinite and  acts properly on $Y\smallsetminus\{p\}$.
 \item A parabolic point $p\in Y$ is called \emph{bounded} if $\Stab_\Gamma(p)$ acts cocompactly on $Y\smallsetminus\{p\}$.
 \item The action is \emph{geometrically finite} if it is a convergence action and all points are conical or bounded parabolic.
\end{itemize}
Suppose the action is geometrically finite, and all stabilizers of parabolic points are finitely generated, then
 \begin{enumerate}
  \item $\Gamma$ is \emph{relatively hyperbolic} with respect to the stabilizers of parabolic points.
  \item $Y$ is the \emph{Bowditch boundary}.
 \end{enumerate}
\end{defi}

\begin{rk}\label{rk:yaman}
 In the setting of Definition~\ref{def:yaman}, the action is \emph{uniform} if all points of $Y$ are conical.
 Then $\Gamma$ is hyperbolic and $Y$ is its Gromov boundary.
\end{rk}

Let us now consider the convex projective setting.

Fix a properly convex open set $\Omega\subset\sph^\dimd$, a closed subset $F\subset\Omega$, and a discrete subgroup $\Gamma\subset\SL^\pm_{\dimd+1}(\R)$ preserving $\Omega$ and $F$.
(In our case $F$ will always be $\partial\Omega$, except in Proposition~\ref{prop:conicality} and Section~\ref{sec:EGF}.)
Fix a $\Gamma$-invariant equivalence relation $\sim$ on $F$ which is closed (in the sense that $\{(x,y)\in F^2 : x\sim y\}\subset F^2$ is closed) and such that $Y:=F/_\sim$ contains at least $3$ points.

We say that
\begin{itemize}
 \item $F$ satisfies the \emph{convexity condition} if it contains all segment $[x,y]\subset\partial\Omega$ with $x,y\in F$.
 \item $F$ satisfies the \emph{visibility condition} it contains all segments of $\partial\Omega$ touching it.
 \item $\sim$ satisfies the \emph{visibility condition} if $[x,y]\subset\partial\Omega$ implies $x\sim y$ for all $x,y\in F^2$.
\end{itemize}

These assumptions imply that the convex hull of $F$ intersects $\Omega$.

\begin{fact}[{Islam--Zimmer \cite[Prop.\,8.1-2]{IZ22relhypb}}]\label{fact:IZ}
 If the convexity condition on $F$ and the visibility condition on $\sim$ are satisfied, then
 $Y$ is a compact metrisable space on which the action of $\Gamma$ is a convergence action.
\end{fact}

By the fact, in order to get relative hyperbolicity, it remains to understand conical and bounded parabolic points.

\subsection{Conicality}

Consider $\Omega$, $F$, $\Gamma$, $\sim$ and $Y:=F/_\sim$ as in the previous section.

We first give a criterion for conicality.
It is similar to a classical criterion in real hyperbolic geometry (see for instance the discussion \cite[p.\,75]{Tu98}), and it also has a similar flavor to a result of Islam and Zimmer \cite[Prop.\,8.4]{IZ22relhypb} and the argument behind a proof of Weisman \cite[Prop.\,8.17]{We22}.

For any $p\in Y$, denote by $\mathrm{Sing}_p\subset T^1\Omega$ the set of unit tangent vectors $v$ such that $v^+,v^-\in p$.
Moreover, set $\mathrm{Sing}:=\bigcup_{p\in Y}\mathrm{Sing}_p$, which is a closed subset of $T^1\Omega$.

\begin{prop}\label{prop:conicality}
 Suppose the convexity condition on $F$ and the visibility condition on $\sim$.
 Consider $p\in Y$, $x\in p\subset{F}$, and $v\in T^1\Omega$ pointing forward at $x$ and backward at $y\in{F}$ with $y\not\sim x$.
 Then $p$ is conical if and only if the projection in $T^1\Omega/\Gamma$ of the forward geodesic orbit of $v$ intersects infinitely often a compact subset of $(T^1\Omega\smallsetminus \mathrm{Sing})/\Gamma$.
\end{prop}
\begin{proof}
 Suppose that the projection of the  forward geodesic orbit of $v$ intersects infinitely often a compact subset of $(T^1\Omega\smallsetminus \mathrm{Sing})/\Gamma$, \ie there exists $\gamma_n\in\Gamma$ and $t_n$ going to infinity such that $\gamma_n\phi_{t_n}v$ converges to $w\not\in \mathrm{Sing}$.
 This implies that $\gamma_nx$ and $\gamma_ny$ converge respectively to $x':=w^+$ and $y':=w^-\in{F}$ with $x'\not\sim y'$.
 It suffices to prove that $\gamma_nz$ accumulates on $[y']$ for any $z\in{F}\smallsetminus p$.
 Up to extracting a subsequence we may assume that $\gamma_nz$ converges to $z'$.
 It is enough to check that $[y',z']\subset\partial \Omega$.
 We can assume that $[y,z]$ intersects $\Omega$.
 It is then enough to check that the Hilbert distance from the base point of $\gamma_n\phi_{t_n}v$ to $\gamma_n[y,z]$ tends to infinity,
 in other words that the distance from $\phi_{t_n}v$ to $[y,z]$ tends to infinity.
 This is a consequence of the fact that $z\not\sim x$ (which implies that $z$ and $x$ do not have the same face).
 
 Suppose that $p$ is conical: there exists $\gamma_n\in\Gamma$ diverging and $a\neq b\in Y$ such that $\gamma_np$ tends to $a$ while $\gamma_nq$ tends to $b$ for any $q\in Y\smallsetminus\{p\}$.
 Up to extracting a subsequence we may assume that $\gamma_nx$ and $\gamma_ny$ converge to respectively $x'\in a$ and $y'\in b$.
 The segment $[x',y']$ intersects $\Omega$ since $x'\not\sim y'$.
 Therefore there exists $t_n\in\R$ such that $\gamma_n\phi_{t_n}v$ converges to some unit tangent vector $w$.
 Since $\gamma_n$ diverges, the sequence $t_n$ accumulates on $\{\pm\infty\}$.
 Up to extracting again a subsequence we may assume that $t_n$ converges.
 If by contradiction the limit was $-\infty$, then by the first step of the present proof, for any $c\in Y\smallsetminus\{y\}$ the sequence $\gamma_nc$ would converge to $a$, which is absurd.
\end{proof}

\subsection{Technical lemmas for parabolicity}

\begin{fact}\label{fact:div seq aut}
 Consider a properly convex open set $\Omega$ and a sequence of automorphisms $g_n\in\GL_{\dimd+1}(\R)$ that converges to a non-invertible non-zero matrix $g$.
 Then the kernel and image of $g$ intersects $\overline\Omega$ but not $\Omega$.
\end{fact}

\begin{lemma}\label{lem:parabolic cvx}
 Consider a properly convex open set $\Omega\subset\sph^\dimd$, a closed subset $F\subset\partial\Omega$ with the visibility condition (it contains all segments of $\partial\Omega$ touching it), and a closed subgroup $\Gamma\subset\Aut(\Omega)$ preserving $F$.
 
 Let $X\subset\sph^\dimd$ be the union of supporting hyperplanes of $\Omega$ at points of $F$ (it is closed).
 Let $\cal O$ be the connected component of $\sph^\dimd\smallsetminus X$  that contains $\Omega$.
 It is open, convex, $\Gamma$-invariant, and contains $\overline\Omega\smallsetminus F$.
 Then:
 \begin{itemize}
  \item Either $\Gamma$ acts properly on $\cal O$, and every $\Gamma$-orbit of a compact subset of $\cal O$ accumulates on $F$,
  \item or $F$ is reduced to a point and $\Gamma$ is virtually generated by a \emph{rank-one} automorphism in the sense of Islam \cite[Def.\,6.2]{I19}.
 \end{itemize}
\end{lemma}
\begin{proof}
 Consider a diverging sequence $g_n\in\Gamma$ such that $\Vert g_n^{\pm1}\Vert^{-1} g_n^{\pm1}$ converges to a non-zero non-invertible matrix $g_\pm$.
 By Fact~\ref{fact:div seq aut}, the following four compact convex sets are non-empty and contained in $\partial\Omega$.
 $$\emptyset\neq I_-:=\Im(g_-)\cap\overline\Omega\ \subset\ K_+:=\Ker(g_+)\cap\overline\Omega\ \subset\partial\Omega$$
 $$\emptyset\neq I_+:=\Im(g_+)\cap\overline\Omega \ \subset \ K_-:=\Ker(g_-)\cap\overline\Omega\ \subset\partial\Omega$$
 
 Let us prove that if $\Gamma$ is not virtually generated by a rank-one element, then $K_-$ and $K_+$ are contained in $F$.
 Suppose for instance that $K_+$ is not contained in $F$.
 Then those two compact convex subsets of $\partial\Omega$ are disjoint by the visibility condition on $F$.
 By definition of $K_+$ this means that $g_nF$ tends to $g_+(F)\subset I_+$ as $t\to\infty$.
 But $g_nF=F$ for all $n$ so $F\subset I_+$, and hence $I_+=K_-=F$ by the visibility condition on $F$.
 This implies that $I_-$ and $K_-$ are disjoint, and hence $I_-$ admits a compact convex neighborhood $U$ disjoint from $K_-=F$ such that $g_n^{-1}U\subset U$ for $n$ large.
 By the Brouwer fixed point theorem, $g_n$ has a fixed point $x\in U$, and also one $y\in F$.
 By \cite[Cor.\,6.9]{mesureBM}, $g_n$ is a rank-one automorphism and $F=\{y\}$, and hence $\Gamma$ is a rank-one group that fixes $y$.
 This implies that $\Gamma$ is virtually generated by $g_n$ by \cite[Prop.\,2.4]{EeERfH+}.
 
 Suppose now that $K_-$ and $K_+$ are contained in $F$.
 Let us consider a converging sequence $x_n\in\cal O$ with limit $x\in \cal O$ and prove that $g_nx_n$ converge to a point of $F$.
 Consider $y\in K_+$.
 The kernel $\Ker(g_+)$ is contained in a supporting hyperplane of $\Omega$ at $K_+$ so it does not intersects $\cal O$ (by definition), hence $g_nx_n$ tends to $g_+x$ as $n\to\infty$.
 By definition of $\cal O$, the ray $[x,y)\subset \cal O$ enters $\overline\Omega\smallsetminus F$ at some point $x'$.
 Since $y\in K_+$, we have $g_+x=g_+x'=\lim_ng_nx'\in \overline\Omega\cap \Im(g_+)=K_+$. 
\end{proof}

Lemma~\ref{lem:parabolic cvx} has the following corollary, where is exploited the assumption in Theorem~\ref{thm:relhypb} that the subgroups ``with respect to which we want to prove relative hyperbolicity'' have codimension $1$; assumption that, we recall, compensate for the absence of convex-cocompactness.

\begin{cor}\label{cor:parab cvx}
Consider a properly convex open set $\Omega$, a closed subset $F\subset\partial\Omega$ with the visibility condition, a discrete group $\Gamma\subset \Aut(\Omega)$ preserving $F$, a $\Gamma$-invariant closed equivalence relation $\sim$ on $F$ with at least three equivalence classes and the visibility condition (see Section~\ref{sec:yaman}), and $Y:=\partial\Omega/_\sim$.

Pick a point $y\in Y$, seen as a closed subset of $\partial\Omega$, which is not fixed by a rank-one element of $\Gamma$ (for instance if $y$ contains at least three points).
Then $y$ is a parabolic point.

Suppose moreover that:
\begin{itemize}
 \item $F=\partial\Omega$.
 \item $\partial\Omega\smallsetminus y$ is homeomorphic to a finite disjoint union of copies of $\R^{\dimd-1}$ (\eg if $y$ is a face of $\partial\Omega$).
 \item The virtual cohomological dimension of $\Stab_\Gamma(y)$ is equal to $\dimd-1$ (\eg if $y$ is a face on which $\Stab_\Gamma(y)$ acts properly cocompactly).
\end{itemize}
Then $y$ is bounded parabolic.
\end{cor}
\begin{proof}
Lemma~\ref{lem:parabolic cvx} tells us $\Stab_\Gamma(y)$ acts properly discontinuously on $F\smallsetminus y$.

Suppose by contradiction that two points $x,x'\in Y\smallsetminus \{y\}$ are dynamically related by $x_n\in Y\smallsetminus \{y\}$ converging to $x$ (for $Y$'s topology), $\gamma_n\in\Stab_\Gamma(y)$ diverging and $\gamma_nx_n$ converging to $x'$.

Pick any sequence $(p_n)_n\in\prod_nx_n$.
By definition of the quotient topology on $Y$, up to extracting we can assume that it converges to $p\in x\subset\partial\Omega\smallsetminus y$,
and that $\gamma_np_n$ converges to $p'\in x'\subset\partial\Omega\smallsetminus y$.
This means $\Stab_\Gamma(y)$ does not act properly on $\partial\Omega\smallsetminus y$, which is a contradiction.
 
The last conclusion of the corollary follows immediatly from the well-known fact that a discrete group of cohomological dimension $\dimd-1$ acting properly on $\R^{\dimd-1}$ must act cocompactly. See for instance \cite{cohomgpdiscret}.
\end{proof}

\subsection{Proof of Theorem~\ref{thm:relhypb} and Fact~\ref{fact:Teddyrelhypb}}

We are going to prove both results at the same time.
Let us call ``walls'' the hypersurfaces $W_1,\dots,W_n\subset M$ in the setting of Fact~\ref{fact:Teddyrelhypb}, as well as their lifts in $\tilde M$.

A point of $Y$ corresponds either to an extremal point of $\partial\tilde M$ or to a closed subset of $\partial\tilde M$ of the form $\bar W\cap \partial\tilde M$ where $W$ is a wall of $\tilde M$.

By Fact~\ref{fact:IZ}, $Y$ is compact metrisable and the action of $\pi_1(M)$ is a convergence action.

By Corollary~\ref{cor:parab cvx}, every point of $Y$ corresponding to a wall of $\tilde M$ is bounded parabolic.

Thus, by the working definition of relative hyperbolicity, it is enough to check that  every point $x\in\partial\tilde M$ outside of closed walls projects to a conical point of $Y$, using Proposition~\ref{prop:conicality}.
 
Fix compact pairwise disjoint neighborhoods for the walls of $M$.
Pick also a point $p$ in the interior of $\tilde M$.
 
By Proposition~\ref{prop:conicality} it is enough to check that the projection in $M$ of the ray $[p,x)$ passes infinitely often in the complementary of our fixed set of neighborhoods.
Let us assume the contrary: that there exists $q\in[p,x)$ such that $[q,x)$ is contained in a lift $U$ of a neighborhood of a wall; denote by $W\subset U$ the lift of the wall.

Let $q_n\in[q,x)$ converge to $x$.
Since the stabilizer $\Stab_\Gamma(W)$ of $W$ acts cocompactly on it and on $U$, there exists a diverging sequence $g_n\in\Stab_\Gamma(W)$ such that $g_nq_n$ stays in a compact set $K$ of $U$.

Up to extracting a subsequence we may assume that $\Vert g_n\Vert^{-1}g_n$ converges to a matrix $g$.
By Lemma~\ref{lem:parabolic cvx}, if $\Omega$ is the interior of $\tilde M$ then $\Ker(g)\cap \overline\Omega$ and $\Im(g)\cap \overline\Omega$ are contained in $\overline W$, which does not contain $x$.
Thus $g_nq_n$ tends to $gx$, which is also the limit of $g_nx$ and hence belongs to the relative boundary of $W$, which does not meet $K$: contradiction.

\subsection{Extended geometrical finiteness}\label{sec:EGF}

Consider a properly convex open set $\Omega\subset\sph^\dimd$, a closed subset $F\subset\Omega$ 
with the visibility condition (it contains all segments of $\partial\Omega$ touching it), and a discrete subgroup $\Gamma\subset\Aut(\Omega)$ preserving $F$.
Fix a $\Gamma$-invariant closed equivalence relation $\sim$ on $F$ with at least three equivalence classes and the visibility condition (see Section~\ref{sec:yaman}), and $Y:=F/_\sim$.

Denote by $\R\mb P^\dimd{}^*$ the set of hyperplanes of $\R\mb P^\dimd$.
Let $F^*\subset\partial\Omega^*$ be the set of supporting hyperplanes intersecting $F$.
Notice that $F^*$ is a closed subset with the visibility condition, and that it carries a natural $\Gamma$-invariant closed equivalence relation with the visibility condition ($H\sim H'$ if there are $x\in H\cap F$ and $x'\in H'\cap F$ with $x\sim x'$), such that the map $F^*/_\sim \rightarrow Y$, mapping $[H]$ to $[x]$ when $x\in H$, is well-defined and is a homeomorphism.

Let $\cal F\subset\R\mb P^\dimd\times \R\mb P^\dimd{}^*$ be the partial flag variety of pairs $(x,H)$ such that $x\in H$.
Set $\Lambda:=(F\times F^*)\cap \cal F$.
We have a natural $\Gamma$-equivariant surjective continuous map $\phi:\Lambda\rightarrow Y$, sending $(x,H)$ to the equivalence class of $x$.
This map extends to a map $\phi_1\times\phi_2:F\times F^*\rightarrow Y$.

Our assumptions imply that $\phi$ is \emph{antipodal}: any two points $(x,H),(x',H')\in \Lambda$ with different images by $\phi$ are \emph{antipodal}, in the sense that $x\not\in H'$ and $x'\not\in H$.

\begin{defi}[{Weisman \cite[Def.\,1.2]{We22}}]\label{def:EGF}
 A map $\psi:\Lambda\rightarrow Y$ \emph{extends the convergence action of $\Gamma$ on $Y$} if for every $z\in \Lambda$ there exists an open subset $C_z\subset \mathcal F$ containing $\Lambda\smallsetminus\psi^{-1}(z)$ such that every diverging sequence $\gamma_n\in\Gamma$ admits a subsequence $\gamma_{n_k}$ and $z_\pm\in Y$ such that $\gamma_{n_k}K$ accumulates on $\psi^{-1}(z_+)$ for every compact subset $K\subset C_{z_-}$.
\end{defi}

\begin{lemma}
 The map $\phi$ \emph{extends the convergence action of $\Gamma$ on $Y$}.
\end{lemma}
\begin{proof} 
 For any $z\in Y$, let $A_z\subset \R\mb P^\dimd$ (\resp $A_z^*\subset \R\mb P^\dimd{}^*$) be the connected component containing $\Omega$ (\resp $\Omega^*$) of the (open and convex) set of $x$ (\resp $H$) such that $x\not\in H'$ (\resp $x'\not\in H$) for any $(x,H)\in\phi^{-1}(z)$.
 
 Consider a diverging sequence $g_n\in\Gamma$.
 
 Up to extracting a subsequence, we may assume that $\Vert g_n^{\pm1}\Vert^{-1}g_n^{\pm1}$ converges to a matrix $g_\pm$.
 Set $K_\pm=\Ker(g_\pm)\cap \partial \Omega$ and $I_\pm={\rm Im}(g_\pm)\cap \partial\Omega$ (note that $I_\pm\subset K_\mp$).
 
 By Fact~\ref{fact:div seq aut}, for every $x\in \Omega$ the sequence $g_n^{\pm1}x$ tends to $g_\pm x\in I_\pm$.
 This applies in particular to any $x$ in the convex hull of $F$, and the $\Gamma$-invariance of $F$ implies that it intersects $K_\pm$, and hence contains it by the visibility condition on $F$.
 
 The visibility condition on $\sim$ implies that $\phi$ sends all pairs $(x,H)\in\Lambda$ with  $x\in K_\pm$ on the same point of $z_\mp\in Y$.
 
 Let us check that $g_+A_{z_-}\subset I_+$, with the consequence that $g_nK\to \phi_1^{-1}(z_+)$ for any compact subset $K\subset A_{z_-}$.
 Pick $x\in A_{z_-}$ and $y\in K_+$.
 By definition of $A_{z_-}$, the segment $[x,y]$ intersects $\Omega$ at some point $p$.
 Now any lifts of $x$ and $p$ in $\R^{\dimd+1}$ differ by a vector of $\Ker g_+$, so $gx=xp\in I_+$.
 
 A duality argument (involving $g_n^{-1}$) ensures that for any compact subset $K\subset A_{z_-}^*$, the orbit $g_n K$ accumulates on $\phi_2^{-1}(z_+)$, which is the set of supporting hyperplanes of $\Omega$ at $\phi_1^{-1}(z_+)$.
 
 In conclusion, for any compact subset $K\subset (A_{z_-}\times A_{z_-}^*)\cap \cal F$, the orbit $g_nK$ accumulates on $\phi^{-1}(z_+)$, \ie $\phi$ extends the convergence dynamics of $\Gamma$ on $Y$.
\end{proof}

\begin{cor}
 If the action of $\Gamma$ on $Y$ is geometrically finite, as in the settings of Theorem~\ref{thm:relhypb} and Fact~\ref{fact:Teddyrelhypb}, then the action of $\Gamma$ on $\cal F$ is (by definition) \emph{extended geometrically finite} in the sense of \cite[Def.\,1.3]{We22}.
\end{cor}

\section{Projective gluings}
\label{sec:sec5}

In this section we describe a gluing construction whose input consists of projective manifolds with totally geodesic boundary and corners and whose output is a convex projective cone-manifold. By translating the conditions of the previous sections in the language of gluings:
\begin{itemize}
\item{We establish when the singularities of the glued projective cone-manifold can be blown up to totally geodesic boundary components.}
\item{We describe the shape of the resulting convex domain, in particular, we characterize the flat regions of the boundary and the extremal and $\mc{C}^1$ points.}
\end{itemize}

\subsection{Gluings}\label{sec:combidescription}

We start with some topological preliminaries about gluings.

\begin{defi}[Gluing]
\label{def:topgluing}
 Given a collection of b-manifolds $\cal M$, a {\em gluing} is a smooth involution $f:\partial\cal M\to\partial\mc{M}$ that does not preserve any boundary component. We denote by $\cal M_f$ the quotient of $\cal M$ under the equivalence relation $x\sim f(x)$ for $x\in \partial\cal M$.
 
Every component $M\subset\mc{M}$ can be seen as an \emph{immersed} codimension 0 submanifold of $\cal M_f$ called a {\em cell}. 
Similarly, every boundary component $W\subset\partial\mc{M}$ projects to an embedded codimension 1 submanifold of $\cal M_f$ called a {\em wall}. 
The {\em graph of cells} is the graph whose vertices are cells and whose edges are walls (a wall links the two adjacent cells).
\end{defi}

\subsubsection{Universal cover}\label{sec:univgluing}

Let $\cal M$ be a collection of b-manifolds and $f$ a gluing of the boundary. Assume that the gluing $\mc{M}_f$ is connected. 
We now describe its universal cover 
\[
\pi:\tilde{\mc{M}}_f\to\mc{M}_f.
\]

The preimages of the interior of the cells and the preimages of the walls give a natural decomposition of $\tilde{\mc{M}}_f$
\[
\tilde{\mc{M}}_f=\bigsqcup_{M\subset\mc{M}}{\pi^{-1}({\rm int} (M))}\sqcup \bigsqcup_{M\subset\mc{M}}{\pi^{-1}(\partial M)}.
\]
into cells which are manifolds with boundary that cover the cells of $\cal M_f$ (the components of $\cal M$).

It is convenient to realize this decomposition as a gluing of a covering $\cal N$ of $\cal M$.
We now explain more formally how to do that.

The covering $\cal N$ is a pullback
\[
 \xymatrix{\cal N \ar[r] \ar[d] & {\tilde{\cal{M}}}_f \ar[d]^{\pi} \\
            \cal M \ar[r]_{\pi_f} & {\cal M}_f}
\]
More precisely, $\cal N\subset \tilde {\cal M}_f\times\cal M$ is the set of pairs $(x,y)\in \tilde {\cal M}_f\times \cal M$ such that $\pi(x)=\pi_f(y)$.
It is elementary to check that the second coordinate projection $\cal N\to\cal M$ is a covering.

Then $\tilde{\cal M}_f$ is naturally identified with $\cal N_g$, where $g$ is the involution of $\partial \cal N=\cal N\cap (\tilde{\cal M}_f\times\partial\cal M)$ such that $g(x,y)= (x,f(y))$ for any $(x,y)$.

In general, the cells and walls of $\cal N_g=\tilde{\cal M}_f$ (\ie the components of $\cal N$ and $\partial\cal N$) are not copies of the universal covers of the corresponding cells and walls of $\cal M_f$.
We now show that this is the case when the gluing satisfies a $\pi_1$-injectivity property. 

Conveniently, this can be explained using the theory of graph of spaces of Scott and Wall \cite{SW79}: Let $\mc{G}$ be the graph of cells of $\cal N_g=\tilde{\cal M}_f$ (vertices are components of $\cal N$ and edges are images in $\cal N_g$ of components of $\partial\cal N$). 

%
%


\begin{fact}\label{fact:seifertvankampen}
 If $\cal M$ has $\pi_1$-injective boundary, then $\mc{G}$ is a tree,
 and each cell and wall of $\tilde{\cal M}_f$  (\ie each component of $\cal N$ and $\partial \cal N$) is simply connected.
 \end{fact}

\begin{proof}
 Each wall of $\cal M$ admits a tubular neighborhood, and hence $\cal M_f$ is a realization of the \emph{graph of space} associated to $\cal M$ and $f$ and their graph of cells, in the sense of Scott and Wall \cite[p.\,155]{SW79}.
 According to Scott and Wall, this means that the fundamental group of $\cal M_f$ is isomorphic to the \emph{fundamental group of the associated graph of groups}, in which Scott and Wall \cite[Prop.\,3.6]{SW79} prove that the fundamental groups of each cell and wall injects (see also \cite[\S5 Cor.\,1]{arbres}).
 This proves that each component of $\cal N$ and $\partial \cal N$ is simply connected.
 
 The fact that $\mc{G}$ is a tree comes from that $\tilde{\mc{M}}_f$ is also the realization of a graph of spaces with underlying graph $\mc{G}$.
 Indeed since each cell of $\tilde{\mc{M}}_f$ is simply connected, the fundamental group of $\tilde{\mc{M}}_f$ (which is trivial) coincides with the fundamental group of $\mc{G}$ by \cite[p.\,61]{arbres}.
\end{proof}

The gluings that we will consider from now on will always satisfy the $\pi_1$-injectivity condition of Fact~\ref{fact:seifertvankampen} thanks to the following classical fact.

\begin{fact}
\label{pi1 injective}
 Any wall of a properly convex b-manifold is $\pi_1$-injective.
\end{fact}
\begin{proof}
 Let $M$ be a properly convex b-manifold with universal cover $\tilde M$.
 Each lift in $\tilde M$ of a wall of $M$ is mapped homeomorphically by the developing map to a convex subset of a hyperplane.
 In particular, it is simply connected.
\end{proof}

\subsubsection{Projective gluing}\label{sec:proj gluing}
Next, we introduce geometric structures on gluings:

\begin{defi}[Projective Gluing]
Let $\cal M$ be a collection of b-manifolds.
A \emph{projective gluing} is the data of a gluing $f$ together with a projective structure on $\cal M_f$  that restricts to the projective structure of each component of $\cal M$.
\end{defi}

Not all gluings admit a compatible projective structure and, if it exists, it is not necessarily unique. 
Note for instance that if there is  compatible projective structure, then $f:\partial\cal M \to \cal M$ is an isomorphism of  projective $(\dimd-1)$-manifolds (but this is not a sufficient condition in general).

\subsubsection{Admissible gluing}
We now come to gluings of projective manifolds with totally geodesic boundary and corners. They are defined as follows:

\begin{defi}[Admissible Gluing]
Let $\mc{MC}=\{\mb{M}_j\}_{j\in J}$ be a collection of bc-manifolds with smooth locus  $\cal M=\{M_j\subset\mb{M}_j\}_{j\in J}$.
A projective gluing $f$ for $\mc{M}$ is said to be \emph{admissible} if
for each corner $\corner$ of $\mc{MC}$,
we can find cells $\mb{M}_1,\cdots,\mb{M}_m\subset\mc{MC}$ such that 
each $\mb{M}_i$ has a corner $\corner_i$ (with $\corner_0=\corner$) and two adjacent walls $W_{i}^{-}$ and $W_{i}^{+}$ such that
\begin{itemize}
 \item $fW_i^{+}=W_{i+1}^-$ for each $i$.
 \item $f_{|W_i^{+}}$ extends to a homeomorphism $f_{|\corner_i}:\corner_i\rightarrow\corner_{i+1}$.
 \item $f_{|\corner_{m-1}}\circ\cdots\circ f_{|\corner_1}\circ f_{|\corner_0}$ is the identity map of $\corner_0$. 
\end{itemize}
(All the indices are thought modulo $m$). We denote by $\mc{MC}_f$ the quotient of $\mc{MC}$ by the closure of the relation $x\sim f(x)$ for $x\in \partial{\cal M}$.

A \emph{corner} or \emph{singularity} of $\mc{MC}_f$ is the projection of a corner of $\mc{MC}$. 

A choice of $\{(\mb{M}_i,W_i^\pm,\corner_i)\}_{i\le m}$ as in the previous paragraph is called a \emph{cell ordering} of the corner in $\mc{MC}_f$.
\end{defi}

Admissible gluings provide natural examples of projective cone-manifolds:

\begin{lemma}\label{lem:cone-manifold}
Let $\mc{MC}$ be a collection of bc-manifolds with smooth locus $\cal M$.
Let $f$ be an admissible gluing. Then $\mc{MC}_f$ is a cone-manifold (Definition~\ref{def:cone-manifold}).
\end{lemma}

\begin{proof}
Fix a corner $\corner$. Let $\{(\mb{M}_i,{\rm Wall}_i^\pm,\corner_i)\}_{i\le n}$ be a cell ordering for $\corner$.

One can find an open neighborhood $V$ of $\corner$ that satisfies the following:
\begin{itemize}
\item{$V$ does not intersect any other singularity.}
\item{$V={V}_1\cup\cdots\cup{V}_n$ where each ${V}_j$ is a tubular neighborhood of the lift $\corner_j\subset\mb{M}_j$ of $\corner$.}
\item{$W_j^\pm:=V_j\cap {\rm Wall}_j^\pm$ is a tubular neighborhood of $\corner_j$ in ${\rm Wall}_j^\pm$.}
\item{The universal cover $\tilde{V}$ of $V$ can decomposed as a gluing of universal covers $\tilde V_1,\cdots,\tilde V_n$
via $f_i:\tilde W_i^+\to \tilde W_{i+1}^-$ such that $f_n\circ\dots\circ f_1(x)=x$ for any $x\in\tilde\corner_1$.}
\item{$\pi_1(\corner)$ naturally acts on $\tilde \corner$ and $\tilde V$ and all the $\tilde\corner_i$'s and $\tilde V_i$'s and $\tilde W_i^\pm$'s, so that $\alpha(f_i(x))=f_i(\alpha(x))$ for any $i$ and $x\in \tilde W_i^+$.}
\end{itemize}

Our goal is to find a tube $T$ and a local homeomorphism $\tilde{V}\to T$ which is projective on the smooth locus and equivariant with respect to a morphism $\pi_1(\corner)\to\Aut(T)$. 

Fix a developing map $\dev^1:\tilde V_1\to\sph^\dimd$ that sends $\tilde\corner_1$ to $\sph^{\dimd-2}$. By induction, the projective structure on $\cal M_f$ determines for each $i\le n$ a developing map $\dev^i:\tilde V_i\to\sph^\dimd$ such that
$\dev^{i-1}=\dev^i\circ f_{i-1}$ on $\tilde W^+_{i-1}$. It also determines a new developing map of $\tilde V_1$ of the form $\dev^{n+1}=g\circ\dev^1$ where $g\in \SL_{\dimd+1}(\R)$, such that $\dev^{n}=\dev^{n+1}\circ f_n$ on $\tilde W^+_{n}$.

\begin{claim}{1}
$g$ fixes every point of $\sph^{\dimd-2}$.
\end{claim}

\begin{proof}[Proof of the claim]
For every $x\in\corner_1$ we have
\begin{align*}
 g\dev^1(x) & = \dev^{n+1}(f_{n}\circ\cdots\circ f_{2}\circ f_{1}(x)) \\
 & = \dev^{n}(f_{n-1}\circ\cdots\circ f_{2}\circ f_{1}(x))\\
 & = \dev^1(x).\qedhere
\end{align*}
\end{proof}

Up to making $V$ smaller, we may assume that for each $1\leq i\leq n+1$ the image of $\dev^i$ is contained in the sector $S_i$ of $\sph^\dimd$ between the half-hyperplanes spanned by $\dev^i(\tilde W_i^-)$ and $\dev^i(\tilde W_i^+)$.
We choose a lift $\tilde S_1\subset\blow\sqcup\sph^{\dimd-2}$ of the first sector, which determines by induction lifts of $S_2,\dots,S_{n+1}$ to sectors $\tilde S_2,\dots,\tilde S_{n+1}\subset\blow\sqcup\sph^{\dimd-2}$. Then $\dev^i$ lifts in a unique way to ${\rm d\tilde ev}^i:\tilde V_i\to\tilde S_i$. Similarly $g$ lifts to a unique $\tilde g\in \Aut(\blow)$ that maps $\tilde S_1$ to $\tilde S_{n+1}$ (and  fixes $\sph^{\dimd-2}$).

Let $T$ be the tube 
\[
T:=\tilde S_1\cup\dots\cup \tilde S_{n}/_\sim
\]
where $\sim$ identifies  via $\tilde g$ the wall  of $\tilde S_1$ spanned by ${\rm d\tilde ev}^1(\tilde W^-_1)$ with the wall of $\tilde S_{n}$ spanned by ${\rm d\tilde ev}^n(\tilde W^+_n)$.
Denote by $\pi_T:\tilde S_1\cup\dots\cup \tilde S_n\to T$ the quotient projection.
We have the following:

\begin{claim}{2}
The map $F:\tilde V\to T$ such that $F(x)=\pi_T({\rm d\tilde ev}^i(\alpha x))$ for all $i$ and $x\in \tilde V_i$ is well-defined and is a local embedding that respects the projective structures.
\end{claim}

Let $\hol:\pi_1(\corner)\to \SL_{\dimd+1}\R$ be the holonomy of $\dev^1$, which lifts to a holonomy ${\rm h\tilde ol}:\pi_1(\corner)\to\Aut(\blow)$ of ${\rm d\tilde ev}^1$. For any $\alpha\in\pi_1(\corner)$, the new developing map ${\rm d\tilde ev}^1\circ\alpha$ determines by induction as above new developing maps of $\tilde V_2,\dots,\tilde V_n,\tilde V_1$, which can be written ${\rm d\tilde ev}^i\circ\alpha$. On the other hand, the developing map ${\rm h\tilde ol}(\alpha)\circ {\rm d\tilde ev}^1$ clearly determines the developing maps ${\rm h\tilde ol}(\alpha)\circ {\rm d\tilde ev}^i$. We check that:

\begin{claim}{3}
${\rm h\tilde ol}(\alpha)$ is an automorphism of $T$.
\end{claim}

\begin{proof}[Proof of the claim]
Since ${\rm d\tilde ev}^1\circ\alpha={\rm h\tilde ol}(\alpha)\circ {\rm d\tilde ev}^1$, we have ${\rm d\tilde ev}^i\circ\alpha={\rm h\tilde ol}(\alpha)\circ {\rm d\tilde ev}^i$ for all $2\leq i\leq n+1$.
As a consequence, ${\rm h\tilde ol}(\alpha)$ preserves each $\tilde S_i$.
Moreover, it commutes with $\tilde g$ since
\begin{align*}
{\rm h\tilde ol}(\alpha)\circ \tilde g\circ {\rm d\tilde ev}^{1}
&={\rm h\tilde ol}(\alpha)\circ {\rm d\tilde ev}^{n+1}
={\rm d\tilde ev}^{n+1}\circ \alpha\\
&=\tilde g\circ {\rm d\tilde ev}^{1}\circ \alpha
=\tilde g\circ {\rm h\tilde ol}(\alpha) \circ {\rm d\tilde ev}^{1}.
\end{align*}
Thus ${\rm h\tilde ol}(\alpha)$ yields an automorphism $\phi(\alpha)$ of $T$.
\end{proof}

It is clear that $\phi$ is a morphism. To conclude, let us check that: 

\begin{claim}{4}
$F\circ\alpha=\phi(\alpha)\circ F$.
\end{claim}

\begin{proof}[Proof of the claim]
Let $x\in \tilde V$.
We have $x\in \tilde V_i$ for some $i$.
Then $\alpha(x)$ is also in $\tilde V_i$, and by definition
\[
F(\alpha(x))=\pi_T({\rm d\tilde ev}^i(\alpha x))=\pi_T({\rm h\tilde ol}(\alpha){\rm d\tilde ev}^i(x))=\phi(\alpha)\pi_T({\rm d\tilde ev}^i(x)).\qedhere
\]
\end{proof}

This concludes the proof of the lemma.
\end{proof}

\subsection{Polars and bulging}
We describe the notion of polars and gluings adapted to polars.
They come in one-parameter families obtained via a bulging procedure.
The goal of this section is to explain this picture.

Let $\mc{MC}$ be a collection of bc-manifolds. Let $f$ be a gluing for $\mc{M}$.

\begin{defi}[Polars]
A \emph{polar point} of a wall $W\subset\cal M$ is the data, for each projective chart of $W$ in $\sph^{\dimd}$, of a point of $\mb{RP}^{\dimd}$ transverse to the image of the chart, which behaves well with respect to transition maps.
 
Each developing map $\tilde W \rightarrow \sph^{\dimd}$ gives rise to a holonomy invariant polar.
 
A projective gluing $\cal M_f$ is said to be \emph{adapted to a given choice of polars} if for any developing map $\tilde{\cal M}_f\rightarrow \sph^{\dimd}$ and any wall $W\subset\partial\cal M$, the associated polars of $W$ and $fW$ coincide.
\end{defi}

\begin{rk}
 If all walls of $\cal M$ admit a polar, then there exists a projective gluing adapted to the polars if and only if $f$ is a projective isomorphism of $(\dimd-1)$-dimensional projective manifolds and for any wall $W\subset\partial\cal M$, denoting by $\rho$ (\resp $\rho'$) the {\em dilation character} of $W$ (\resp $fW$), we have $\rho=\rho'\circ f_*$.
 
The dilation character is defined as follows: For each wall $W\subset\partial\cal M$ and developing map $\dev : \tilde W \rightarrow \sph^{\dimd}$, whose image spans $V\subset \R^{\dimd+1}$, the \emph{dilation character} of $W$ is the determinant of the representation of $\pi_1(W)$ to $\GL(\R^{\dimd+1}/V)$ induced by the holonomy representation. It is a real character that does not depend on the choice of $\dev$.

If $\cal M$ is hyperbolic, then the dilation characters are trivial, the polars are the orthogonal subspaces for the Lorentzian bilinear form, and the unique hyperbolic gluing is adapted to this choice of polars.
 
 If $\cal M_f$ is the double of a manifold $M$, and all boundary components of $M$ admit polars, then there is a natural projective gluing adapted to the polars.
\end{rk}

Let us finally recall the definition of bulging.

\begin{defi}[Bulging]
 Assume that each wall of $\cal M$ admits a polar, and that $\cal M_f$ admits a projective gluing adapted to the polars with atlas $\cal A$.
 Let $(\mu_W)_{W\subset\partial\cal M}$ be a family of positive real parameters.
 The \emph{bulging of $\cal M_f$ along the polars with parameters $(\mu_W)_{W\subset\partial\cal M}$} is the projective gluing which includes the following charts.
 
 For any point $p\in\cal M_f$ which is the image of $\bar p\in W\subset M\subset \cal M$ and $f(\bar p)\in fW\subset M'\subset\cal M$, for any chart $\phi$ of $\cal A$ defined on a small enough neighborhood $U\sqcup U'\subset M\sqcup M'$,
 the map $B\circ \phi_{|U} \cup B'\circ\phi_{|U'}$ is a chart of the bulging, where $B$ (\resp $B'$) is the linear transformation that fixes $\phi(U\cap W)$ (\resp $\phi(U'\cap fW)$) with eigenvalue $\mu_{W}^{-1}$ (\resp $\mu_{fW}^{-1}$) and also fixes the polar, with eigenvalue $\mu_{W}^{\dimd}$ (\resp $\mu_{fW}^{\dimd}$).  
\end{defi}

\subsection{Convexity of gluings}
Having described how an admissible gluing of bc-manifold gives rise to a projective cone-manifold, we now turn to the problem of determining when (the smooth locus of) such manifold is convex and when its decomposition into gluing pieces corresponds to a periodic tessellation of a convex domain in $\sph^\dimd$.
 
In order to do so, we revisit the results of the previous sections using the terminology of admissible projective gluings and Section~\ref{sec:combidescription}.

First of all, we consider a single tile $\mb{M}\in\mc{MC}$, a properly convex bc-manifold with universal cover $\tilde{\mb{M}}\subset\sph^\dimd$ and study its geometry. We describe under which conditions it satisfies the hypothesis of Proposition \ref{prop:injetcvx}.  

\subsubsection{Codimensions 1 and 2 strata}
\label{sec:wallcplt}

Let $\mb{M}$ be a properly convex bc-manifold, with universal cover $\tilde {\mb{M}}\subset\sph^\dimd$. First, we have to make sure that the walls and corners of $\mb{M}$ correspond to faces of $\tilde{\mb{M}}$.

\begin{defi}[Complete Walls and Corners]
\label{def:Cplteness}
A wall (\resp corner) of $\mb{M}$ is said to be \emph{complete} if any lift in $\tilde {\mb{M}}$ consists of a full codimension 1 (\resp codimension $2$) face.
\end{defi}

Completeness of walls and corners is not difficult to check:

\begin{fact}
We have the following:
\begin{itemize}
\item{Any convex hyperbolic bc-manifold which is complete for the hyperbolic metric also has complete walls and corners, in the sense of the previous definition.}
\item{Any \emph{compact} properly convex bc-manifold has complete walls and corners.}
\end{itemize}
\end{fact}

\begin{proof}
The proof is subdivided into claims about $(\dimd-2)$ and $(\dimd-1)$-manifolds.
The first one is classical and not proved here. 
It implies that complete convex hyperbolic bc-manifolds have complete corners.
\begin{claim}{1}
 Any hyperbolic $(\dimd-2)$-manifold which is complete in the metric sense is also complete in the sense of $(G,X)$-structures: The developing map is a bijection onto $\H^{\dimd-2}$.
\end{claim}

Next we have another classical result about hyperbolic manifold.
It implies that complete convex hyperbolic bc-manifolds have complete walls.
\begin{claim}{2}
 Let $N$ be a complete hyperbolic $(\dimd-1)$-manifold with totally geodesic boundary, with universal cover $\tilde N\subset\H^{\dimd-1}$.
 Then $\tilde N$ is the biggest convex subset of $\H^{\dimd-1}$ delimited by the walls of $\tilde N$.
\end{claim}

We now turn to the second point of the fact.
The following classical fact implies that compact properly convex bc-manifolds have complete corners.
We recall the proof for the reader's convenience.
\begin{claim}{3}
 Let $\Omega/\Gamma$ be a closed convex projective $(\dimd-2)$-manifold.
 Then $\Omega$ is maximal for inclusion among $\Gamma$-invariant properly convex open subsets of $\sph^{\dimd-2}$.
\end{claim}
\begin{proof}
 Let $\Omega'$ be $\Gamma$-invariant and contain $\Omega$.
 It suffices to show $\Omega$ is closed in $\Omega'$.
 Let $x\in\Omega'$ be the limit of $x_n\in\Omega$.
 By compactness there is $\gamma_n\in\Gamma$ such that $\gamma_nx_n$ converges to $y\in\Omega$.
 Since $\Gamma$ acts properly on $\Omega'$, the sequence $\gamma_n$ converges in $\Gamma$ and $x\in\Omega$.
\end{proof}

To conclude, the last claim implies that compact properly convex bc-manifolds have complete walls.
\begin{claim}{4}
 Let $N$ be a compact properly convex projective $(\dimd-1)$-manifold with totally geodesic boundary, with universal cover $\tilde N\subset\sph^{\dimd-1}$ and holonomy $\Gamma<\Aut(\mathrm{int}(\tilde N))$.
 Then $\mathrm{int}(\tilde N)$ is maximal among $\Gamma$-invariant  properly convex open subsets of $\sph^{\dimd-1}$ delimited by the span of the walls of $\tilde N$.
\end{claim}
\begin{proof}
 Let $\Omega={\rm int}(\tilde N)$ and let $\Omega'\subset\sph^{\dimd-1}$ be $\Gamma$-invariant, open, properly convex, and delimited by the span of the walls of $\tilde N$.
 It suffices to show $\Omega$ is closed in $\Omega'$.
 Let $x\in\Omega'$ be the limit of $x_n\in\Omega$.
 By compactness there is $\gamma_n\in\Gamma$ such that $\gamma_nx_n$ converges to some $y\in \tilde N$.
 Since $\Gamma$ acts properly on $\Omega'$, if $y\in\Omega$ then $\gamma_n$ converges in $\Gamma$ and $x\in\Omega$.
 Assume that $y$ lies in a wall $W$ of $\tilde N$ (\ie a component of $\tilde N\cap \partial\tilde N$).
 
 Consider a small closed half-ball $B$ centered at $y$ and contained in $\tilde N$, whose boundary is decomposed into $\partial B=(B\cap W)\sqcup \Sigma$ where $\Sigma\subset\Omega$.
 One may check that $\gamma_nx_n\in {\rm int}(B)$ for $n$ large and that the Hilbert distance $d_{\Omega'}(\gamma_nx_n,\Sigma)$ is bounded below by some constant $\epsilon$ independent of $n$.
 Note that $d_{\Omega'}(\gamma_nx_n,\gamma_nx)=d_{\Omega'}(x_n,x)$ tends to zero, and thus is less than $\epsilon$ for $n$ large enough.
 Then $\gamma_nx\in {\rm int}(B)$, otherwise $[\gamma_nx_n,\gamma_nx]$ would cross $\Sigma$ (it cannot cross $B\cap W$) and have length at least $\epsilon$.
 Therefore $x\in\Omega$.
\end{proof}

This concludes the proof of the fact.
\end{proof}

\subsubsection{Ghost strata}\label{sec:ghost}
Let $\mb{M}$ be a properly convex bc-manifold which we now assume to have complete walls and corners which will lift to codimension 1 and 2 faces of the universal cover $\tilde{\mb{M}}\subset\sph^\dimd$.

The next issue that we want to address is the potential presence of codimension 2 faces that do not come from corners. 

\begin{defi}[Ghost Corner]
\label{def:ghost}
 A \emph{ghost stratum} is a face of $\tilde{\mb{M}}$ which is the intersection of at least two closed walls but which is \emph{not} a closed corner.
 It is called a \emph{ghost corner} if it has codimension 2.
\end{defi}

\begin{rk}
\label{ex:ghost}
Here is an example of ghost strata: Let $C\subset\R^3$ be the positive cone consisting of vectors $v=(v_1,v_2,v_3)\in\R^3$ such that $v_1,v_2,v_3>0$ or $v_1,v_2>0=v_3$ or $v_1,v_3>0=v_2$, and let $T:=\sph(C)$. Then $\{[e_1]\}$ is a ghost corner of $T$ and its compact quotient $M$ under the action of any diagonal matrix with positive decreasing diagonal entries. Moreover, $T\cup\{[e_1]\}$ is a properly convex bc-manifold with complete walls and corners and without ghost corner, but, since it is compact, $M$ cannot be described as the complement of corners in a properly convex bc-manifold with complete walls and corners and without ghost corner.
\end{rk}

Again, there is a simple criterion to check that there are no ghost strata in the hyperbolic setting:

\begin{fact}\label{obs:hyperbolic ghost stratum}
 Let $\mb{M}$ be a complete convex hyperbolic bc-manifold. Then:
 \begin{itemize}
 \item{Any ghost stratum has dimension 0.}
 \item{If $\mb{M}$ is compact then there is no ghost stratum.}
 \end{itemize}
\end{fact}

\begin{proof}
 The universal cover $\tilde {\mb M}$ is a closed convex subset of the hyperbolic space $\H^\dimd$ with non-empty interior.
 Let $S$ be a ghost stratum.

 It is a proper face of $\tilde {\mb M}$ and lies in $\partial \tilde{\mb M}\smallsetminus \tilde{\mb M}\subset\partial\H^\dimd$, and hence is a point, since $\H^\dimd$ is strictly convex.
 
 Suppose by contradiction that ${\mb M}$ is compact.
 Then there is $r>0$ such that for any ball $B$ of $\H^\dimd$ of radius $r$, if $B$ intersects two walls of $\tilde {\mb M}$ then these walls share a corner and $B$ intersects no other wall.
 Fix $p\in \tilde {\mb M}$.
 
 If three walls were adjacent to $S$, then any ball of radius $r$ with center in $[p,S)$ close enough to $S$ would intersect all three walls, contradicting the definition of $r$.
 
 Thus $S$ is the intersection $\overline W_1\cap \overline W_2$ where $W_1,W_2$ are two walls of $\tilde{\mb M}$.
 As before any ball of radius $r$ with center in $[p,S)$ close enough to $S$ intersects both $W_1$ and $W_2$, making $S$ an actual corner of $\tilde {\mb M}$, which is absurd.
\end{proof}

\subsubsection{Local convexity}\label{sec:containment}

We are now ready to translate the local convexity conditions of Proposition \ref{prop:injetcvx}.

\begin{defi}[Containment Condition]\label{def:ContainmentCond} 
Let $M$ be a properly convex  b-manifold. We say that the polar point associated to a wall $W$ of $M$ satisfies the \emph{Containment Condition} if for any lift $\tilde W$ of $W$ in $\tilde M$, and any developing map $\dev : \tilde M \rightarrow \sph^{\dimd}$, the image of $\dev$ is contained in the convex hull of $\dev(\tilde W)$ and the polar point.
\end{defi}

\begin{fact}\label{obs:containment implies adj cells cvx}
 Let $\cal M$ be a collection of properly convex b-manifolds with polar points satisfying the Containment Condition, and consider a compatible projective gluing.
 Then the union of any two adjacent cells of the universal cover is convex.
\end{fact}

\begin{proof}
 This is a consequence of Remark~\ref{rk:union of two convex sets}.
\end{proof}

In the case of convex hyperbolic manifolds, the Containment Condition follows from angle assumptions on the corners.

\begin{fact}\label{obs:hyperbolic containment condition}
 Let ${\mb M}$ be a complete hyperbolic bc-manifold and $W$ a wall such that all adjacent corners have angle at most $\pi/2$.
 Then the hyperbolic polar point of $W$ satisfies the Containment Condition.
\end{fact}

\begin{proof}
 Suppose $\tilde{\mb M}\subset \H^\dimd$ and $\tilde W\subset\H^{\dimd-1}$.
 Let $p\in\sph^{\dimd}$ be one of the two hyperbolic polars, such that $\tilde {\mb M}$ intersects the cone $C=\cal{CH}(p,\H^{\dimd-1})$.
 Let $X\subset\H^\dimd$ be the half-space with boundary $\H^{\dimd-1}$ that contains $\tilde {\mb M}$.
 
 Fact~\ref{obs:hyperbolic containment condition} is a consequence of the classical facts that:
 \begin{itemize}
  \item $X$ satisfies the Containment Condition with respect to $\H^{\dimd-1}$ and $p$, \ie $X\subset C=\cal{CH}(\H^{\dimd-1},p)$.
  \item For any $x\in \H^{\dimd-1}$, the projective line through $x$ and $p$ intersects $\H^\dimd$ in exactly the geodesic perpendicular to $\H^{\dimd-1}$ at $x$.
 \end{itemize}
 
 Indeed, let $\{\tilde W_i\}_{i}$ be the collection of walls of $\tilde{\mb M}$ adjacent to $\tilde W$.
 For each $i$ let $H_i=\Span(\tilde W_i)\cap \H^\dimd$, let $H_i'\subset \H^\dimd$ be the hyperplane orthogonal to $\H^{\dimd-1}$ at $\tilde W\cap\tilde W_i$, and let $X_i$ and $X_i'$  be the half-spaces with boundary $H_i$ and $H_i'$ containing $\tilde{\mb M}$.
 
 Then $Y':=X\cap\bigcap_iX_i'$ is the preimage of $\tilde W$ by the orthogonal projection on $\H^{\dimd-1}$,
 and it satisfies the Containment Condition with respect to $Y'\cap\H^{\dimd-1}=\tilde W$, \ie $X'\subset \cal{CH}(\tilde W,p)$.
 
 Moreover, $Y'$ contains $Y:=X\cap\bigcap_iX_i$ which contains $\tilde {\mb M}$ since all corners adjacent to $\tilde W$ have angle at most $\pi/2$.
 Thus $\tilde {\mb M}$ must also satisfy the Containment Condition.
\end{proof}

\subsubsection{Global convexity}

We now rephrase Proposition~\ref{prop:injetcvx} using the terminology that was introduced in the previous sections.

\begin{thm}\label{thm:cvx}
 Let $\mc{MC}$ be a collection of bc-manifolds. Let $f$ be a projective gluing on $\cal M_f$.
 Assume that:
 \begin{enumerate}
  \item \label{item:thmcvx1} $\cal M_f$ is connected.
  \item \label{item:thmcvx2} All cells are properly convex, with complete walls and corners (Definition~\ref{def:Cplteness}), and without ghost corners (Definition~\ref{def:ghost}).
  \item \label{item:thmcvx3} The union of any two adjacent cells in the universal cover is convex (satisfied for instance if there are polar points satisfying the Containment Condition).
  \item \label{item:thmcvx4} Each singularity is uniformisable in $\sph^\dimd$ (see Proposition~\ref{prop:uniformization}).
 \end{enumerate}
 Then $\mc{M}_f$ is convex.
\end{thm}

\begin{proof}
 Let $\cal G=(\cal V,\cal E)$ be the graph of cells associated to the gluing of $\tilde{\cal M}_f$: each $v\in\cal V$ corresponds to a cell of $\tilde{\cal M}_f$ and each $e\in\cal E$ corresponds to a wall of $\cal M_f$ (see Definition~\ref{def:topgluing} and Section~\ref{sec:univgluing}).
 Fix a developing map $\tilde{\cal M}_f\rightarrow\sph^\dimd$, and for each $v\in\cal V$ let $D_v$ be the closure in $\sph^\dimd$ of the image of $v$ by the developing map.
 
 It is enough to check that $(\cal G,\{D_v\}_{v\in\cal V})$ satisfies the assumptions of Proposition~\ref{prop:injetcvx}, \ie that it is a gluing kit (Definition~\ref{def:gluing kit}). Note that connectedness of $\cal G$ is an immediate consequence of Property~\eqref{item:thmcvx1}.
 
{\bf Assumption~\eqref{item:hypi1} of Proposition~\ref{prop:injetcvx}.} This follows from the fact that walls of $\cal M$ are complete (Property~\eqref{item:thmcvx2}) and Property~\eqref{item:thmcvx3}.
 
 {\bf Assumption~\eqref{item:hypi2} of Proposition~\ref{prop:injetcvx}.}
 Let $p=(v_1,v_2,\dots,v_n)\subset \cal G$ be a path such that $F_p:=D_{v_1}\cap\dots\cap D_{v_n}$ is a codimension $2$ face of $D_{v_i}$ for every $i$.
 Since $\cal M$ has no ghost corner (Property~\eqref{item:thmcvx2}), ${\rm int}(F_p)$ is a lift of a corner of $\mc{MC}_f$.
 As a consequence, $p$ may be extended to a bi-infinite path $q=(v_k)_{k\in\Z}$ such that $\bigcap_{k\in\Z}D_{v_k}=F_p$.
 Then Property~\eqref{item:thmcvx4} implies by Proposition~\ref{prop:uniformization} that the sequence $(D_{v_k})_{k\in\Z}$ satisfies the assumptions of Proposition~\ref{prop:hypi2 revisited}, whose conclusion is that $F_p$ is contained in the boundary of a half-space of $\sph^{\dimd}$ containing $\bigcup_{n\in\Z}D_{v_k}$, as required by Proposition~\ref{prop:injetcvx}.
\end{proof}

\subsection{Totally geodesic blowups of gluings}
As we have seen in Lemma~\ref{lem:cone-manifold}, admissible projective gluings give rise to projective cone-manifolds. Hence, as proved in Section \ref{sec:sec2}, under suitable conditions on the holonomy of the meridians of the corners, we can blow up each singularity to a totally geodesic boundary component and we can also give an explicit description of such components in terms of the geometry of the corners. This is the content of the next theorem.

\begin{thm}\label{thm:addendum1} 
 In the setting of Theorem~\ref{thm:cvx}, 
 assume further that
 \begin{itemize}
  \item $\cal M_f$ is properly convex (for instance if one of the cells has an irreducible holonomy).
  \item Each singularity of $\mc{MC}_f$ is special (Definition~\ref{def:special tubes}).
 \end{itemize}
 Then the following holds:
 \begin{enumerate}
 \item{The totally geodesic blowup $N$ of the singularities of $\mc{MC}_f$ (Definition~\ref{def:totgeod blowup}) has complete walls and no ghost corners.
 If each cell of $\cal M$ has no ghost stratum then so does $N$.}
 \item{For each wall $W$ of $N$ blowing up a corner $\corner$ there is a developing map of $N$ sending $\tilde W$ onto the interior of $\conv(\tilde\corner,[e_{\dimd}])$
such that the holonomy of $\pi_1(W)$ is generated by that of the corner
\[
 \gamma\in\pi_1(\corner)\longmapsto \left(
\begin{array}{c c}
\phi(\gamma)^{-1}\rho'(\gamma) & \\
 &\phi(\gamma)^{\frac{\dimd-1}{2}}\mb{I}_2\\
\end{array}
\right)\in{\rm SL}_{d+1}(\mb{R})
\]
and of the meridian 
\[
 \left(
\begin{array}{ccc}
 \mu^{-1}\mb{I}_{d-1}&& \\
 &\mu^{\frac{\dimd-1}2}\lambda &\ep\\
 &&\mu^{\frac{\dimd-1}2}\lambda^{-1}\\
\end{array}
\right)\in{\rm SL}_{d+1}(\mb{R}),
\]
where $\rho'$ is a morphism, $\phi$ a positive character, $\lambda\geq1$, $\ep=1$ if $\lambda=1$ and $0$ otherwise, and $\mu>\lambda^{\frac{\dimd+1}2}$.}
\item{If $\lambda >1$ then $[e_{\dimd+1}]$ yields a polar of $W$ which satisfies the {\rm Containment Condition}.}
\end{enumerate}
\end{thm}

\begin{proof}
 We start with the description of the developing map of $N$ sending $\tilde W$ onto $\conv(\tilde\corner,[e_{\dimd}])$.

 By Lemma~\ref{lem:cone-manifold}, $\mc{MC}_f$ is a cone-manifold: Every point of $\corner$ admits a neighborhood that embeds into a tube $\mb T=T\sqcup\sph^{\dimd-2}$.
 Thus we can find a tubular neighborhood $\mb{U}\subset \mc{MC}_f$ of $\corner$ and a developing map $\dev_1:\tilde{\mb{U}}\rightarrow \mb T$ with holonomy in the automorphism group of $\mb T$.
 Since each cell of $\cal M$ is properly convex, $\corner$ is uniformisable in $\sph^{\dimd-2}$ so $\dev_1$ is injective on $\tilde \corner$.
 Up to making $\mb{U}$ smaller we can assume that $\dev_1$ is injective, which identifies $\tilde{\mb U}$ with a subset of $\dev_1(\tilde{\mb U}):=\mb V\subset\mb T$ (note that $\mb V$ is the universal cover of $\mb U$ but the smooth locus $V$ is \emph{not} the universal cover of the smooth locus $U$).
 
 Let $\pi: T^b\rightarrow \mb T=T\cup\sph^{\dimd-2}$ be the totally geodesic blowup of $\mb T$.
 By definition, the totally geodesic blowup $V^b$ of $\mb{V}$ is $\pi^{-1}(\mb{V})$.
 , and there is a developing map of the universal cover of the totally geodesic blowup $U^b$ of $\mb U$ denoted by $\dev_2:\tilde{U}^b\rightarrow V^b$ whose holonomy consists of automorphisms of $\mb T$.

 There is a developing map $\dev_3:\tilde{T}^b\rightarrow\sph^\dimd$ that sends the boundary onto $\conv(\sph^{\dimd-2},[e_{\dimd}])$ such that the holonomy of the meridian has the form
 \[
 g=\left(
\begin{array}{ccc}
 \mu^{-1}\mb{I}_{d-1}&& \\
 &\mu^{\frac{\dimd-1}2}\lambda&\ep\\
 &&\mu^{\frac{\dimd-1}2}\lambda^{-1}\\
\end{array}
\right)\in{\rm SL}_{d+1}(\mb{R}),
\]
with $\lambda\geq1$, $\ep=1$ if $\lambda=1$ and $0$ otherwise, and $\mu>\lambda^{\frac{\dimd+1}2}$, since by assumption each singularity of $\mc{MC}_f$ is special.

The map $\dev_2$ lifts to a developing map $\dev_4:\tilde{U}^b\rightarrow \sph^\dimd$.

When restricted to a cell $X$ of $\tilde{\cal M}_f$ (smooth locus of $\mb{X}$, which is a universal cover of a cell of $\mc{MC}$), $\dev_4$ extends to a developing map $\dev_5$ of $\mb{X}$ that sends $\tilde\corner$ onto $\tilde C\subset\sph^{\dimd-2}$, which in fine does not depend on $X$.

The map $\dev_4$ sends the boundary of $\tilde{U}^b$ onto $\conv(\tilde C,[e_{\dimd}])$, and the holonomy of the meridian is $g$, and the holonomy of $\corner$ is in $\Aut(\blow)$ and commutes with $g$, so it preserves $\sph^{\dimd-2}$ and $[e_\dimd]$ and $\Span([e_\dimd],[e_{\dimd+1}])$.
Since this holonomy must also preserve the image by $\dev_5$ of the walls of $\mb{X}$ adjacent to $\tilde \corner$, it has to fix every point of $\Span([e_\dimd],[e_{\dimd+1}])$ and hence have the form stated in Theorem~\ref{thm:addendum1}.

By Property \ref{item:aroundcodim2 cvxhull} of Proposition~\ref{prop:aroundcodim2}, $\conv(\tilde C,[e_{\dimd}])$ corresponds to a cell at infinity of the associated gluing kit $\cal G=(\cal V,\cal E)$.
Property \ref{item:aroundcodim2 face} of Proposition~\ref{prop:aroundcodim2} says moreover that such a cell at infinity is a full closed face of $\partial\tilde N$, which implies that $N$ has complete walls.
 
By Property \ref{item:cells infinity 1} of Proposition~\ref{prop:cells infinity} and the above, a ghost stratum of $\tilde N$ (the intersection of at least two closed walls of $\tilde N$) is in fact the intersection of at least two corners of ${\tilde{\mc{MC}}}_f$.
 This implies that $N$ has no ghost corners, and moreover no ghost stratum at all if $\cal M$ has no ghost stratum.
 
 That each wall of $N$ above a singularity with hyperbolic $\SL_2$-angle admits a polar point satisfying the Containment Condition is an immediate consequence of Property \ref{item:aroundcodim2 containment} of Proposition~\ref{prop:aroundcodim2}.
\end{proof}

\begin{rk}
 In the setting of Theorem~\ref{thm:addendum1}, if every singularity has hyperbolic $\SL_2$-angle then the double $P$ of $N$ admits a convex projective structure.
\end{rk}

\subsection{Stratification of the boundary for a gluing}
Lastly, after blowing up the singularity of the convex projective cone-manifold given by a suitable admissible gluing, we want to describe the geometry of its universal cover. In order to do so, we have to understand the stratification of the boundary. We do this exploiting the work we developed in Section \ref{sec:sec3} as stated in the next result:

\begin{thm}\label{thm:addendum2}
 In the setting of Theorem~\ref{thm:addendum1},
 suppose that: 
 \begin{itemize}
  \item For each cell $M\subset\cal M$ with universal cover $\tilde M$, each non-trivial face of $\tilde M$ touching a wall is contained in a wall.
  \item For each cell $M\subset\cal M$ with universal cover $\tilde M$, each non-trivial face of $\tilde M$ touching a  corner is a wall adjacent to the corner or is contained in the corner.
  \item There is no ghost stratum.
  \item $\mc{MC}_f$ is compact.
 \end{itemize}
 Then the associated gluing kit satisfies the assumptions of Proposition~\ref{prop:contracting cells at infinity}, 
 so any cell at infinity which is not a wall of the totally geodesic blowup $\tilde N$ consists of an extremal $\cal C^1$ point.
\end{thm}

\begin{proof}
 Let $\cal G=(\cal V,\cal E)$ be the gluing kit associated to $(\cal M,f)$, so that $\tilde M_f=\interior X$ where $X:=\bigcup_{D\in\cal V}D$.
 Let us see how to derive the above theorem from Lemmas~\ref{lem:finding visible triples} and \ref{lem:walls far away} and from our compactness assumption.

 Fix a smooth structure on $\mc{MC}_f$ that extends that of $\cal M_f$ by using a nice parametrization. Fix also a Riemannian metric.
 
 The group $\Gamma:=\pi_1(\cal M_f)$ acts on the set of pairs $(D,W)$ where $D$ is a cell of $\cal G$ and $W$ is an adjacent wall, and this action admits finitely many representatives $(D_1,W_1),\dots,(D_n,W_n)$ since $\mc{MC}_f$ is compact.
 
 Fix $i$. 
 Choose for each wall $W$ of $D_i$ not adjacent to $W_i$, in a $\Stab_\Gamma(W_i)$-equivariant way, a segment from $W_i$ to $W$ with smallest length for the Riemannian metric; denote by $v_W$ the starting vector (transverse to $W_i$) and by  $y_W\in W$ the ending point of the segment.
 By compactness of $\mc{MC}_f$, we can find a compact set $K$ of transverse vectors to $W_i$ and a set $\cal W_i$ of $\Stab_\Gamma(W_i)$-representatives of walls $W$ of $D_i$ non-adjacent to $W_i$ such that $v_W\in K$ for any $W\in\cal W_i$.
 Then there exists an open neighborhood $U_i$ of $W_i$ such that $y_W\not\in U_i$ for any $W\in\cal W_i$.
 By Lemma~\ref{lem:walls far away}, we can find an open neighborhood $U_i'$ of $W_i$ such that $D_i\smallsetminus U_i'$ is convex and contains walls in $\cal W_i$.
 By Lemma~\ref{lem:finding visible triples}, $(W_i,\bar X,D_i\smallsetminus U_i')$ is contracting.
 
 By compactness, there is a finite number of $\Stab_{\Gamma}(W_i)$-representatives $\{W_{i,j}\}_j$ of walls of $D_i$ adjacent to $W_i$.
 Fix $j$.
 Let $D_{i,j}$ be the cell other than $D_i$ which is adjacent to $W_{i,j}$ and $\corner_{i,j}:=W_i\cap W_{i,j}$.
 Choose for each corner $\corner$ of $W_{i,j}$ other than $\corner_{i,j}$, in a $\Stab_\Gamma(W_i,\corner_{i,j})$-equivariant way, a segment from $\corner_{i,j}$ to $\corner$ with smallest length for the Riemannian metric; denote by $v_\corner$ the starting vector (transverse to $\corner_{i,j}$) and by  $y_\corner\in \corner$ the ending point of the segment.
 
 By compactness, we can find a compact set $K$ of vectors tangent to $W_{i,j}$ and transverse to $\corner_{i,j}$ and a set $\cal{C}_{i,j}$ of $\Stab_\Gamma(W_i,\corner_{i,j})$-representatives of corners $\corner$ of $W_{i,j}$ other than $\corner_{i,j}$ such that $v_\corner\in K$ for any $\corner\in\cal{C}_{i,j}$.
 Then there exists an open neighborhood $U_{i,j}$ of $W_i$ such that $y_\corner\not\in U_i$ for any $\corner\in\cal C_{i,j}$.
 By Lemma~\ref{lem:walls far away}, we can find an open neighborhood $U_{i,j}'$ of $W_i$ such that $D_i\cup D_{i,j}\smallsetminus U_{i,j}'$ is convex and contains walls adjacent to a corner of  $\cal C_{i,j}$, other than $W_{i,j}$.
 By Lemma~\ref{lem:finding visible triples}, $(W_i,\bar X,(D_i\cup D_{i,j})\smallsetminus U_{i,j}')$ is a visible triple.
 
 Let us now stop fixing $i,j$, and let $\cal K$ be the set of visible triples of the form $(W_i,\bar X,M)$ where $M\subset D_i\smallsetminus U_i'$ or $(D_i\cup D_{i,j})\smallsetminus U_{i,j}'$ for some $i,j$.
 
 Let $A,B,C,D$ be a path of cells of $\cal G$ such that $A\cap B\cap C\cap D$ is empty.
 \begin{itemize}
  \item if $A\cap B\cap C$ is empty, then do the following.
  Find $i$ and $g\in\Gamma$ such that $(B,A\cap B)=g(D_i,W_i)$.
  Find $W\in\cal W_i$ and $h\in\Stab_\Gamma(W_i)$  such that $B\cap C=ghW$.
  Then $(A\cap B,\bar X,C\cap D)\in gh\cal K$;
  \item if $B\cap C\cap D$ is empty, then apply the previous point to the path $D,C,B,A$, so that $(C\cap D,\bar X,A\cap B)\in\GL_{\dimd+1}(\R)\cdot \cal K$;
  \item if $A\cap B\cap C$ and $B\cap C\cap D$ are non-empty, then do the following.
  Find $i$ and $g\in\Gamma$ such that $(B,A\cap B)=g(D_i,W_i)$.
  Find $j$ and $h\in\Stab_\Gamma(W_i)$ such that $B\cap C=hgW_{i,j}$ and $A\cap B\cap C=gh\corner_{i,j}$.
  Find $\corner\in \cal C_{i,j}$ and $k\in\Stab_\Gamma(W_i,\corner_{i,j})$ such that $B\cap C\cap D=ghk\corner$.
  Then $(A\cap B,\bar X,C\cap D)\in ghk\cal K$.\qedhere
 \end{itemize}
\end{proof}


\section{Hyperbolic building blocks}
\label{sec:sec6}

In this section we construct the hyperbolic building blocks that serve as input for the construction in the proof of Theorems \ref{thm:mainA} and \ref{thm:mainB}. In particular we describe the results of Bonahon and Otal \cite{BO04} and we prove Theorem \ref{thm:main7}.

\subsection{Dimension $d=3$}\label{sec:dim3}

In dimension $d=3$ there is a lot of flexibility in the choice of building blocks. 

\subsubsection{3-manifold topology}
The existence of a convex hyperbolic metric with totally geodesic boundary bent along the some corners on a compact 3-manifold $(M,\partial M)$ such that the pleating locus coincides with a multicurve $\alpha=\alpha_1\sqcup\cdots\sqcup\alpha_n\subset\partial M$ can be reduced to an explicit purely topological problem by work of Bonahon and Otal \cite{BO04}.

Before stating their results, we briefly review the necessary terminology and notion from 3-dimensional topology. For a more in-depth discussion, we refer to \cite[Ch\;9]{Mar16}.

\begin{defi}[Irreducible]
\label{def:irreducible}
An orientable 3-manifold $M$ is {\em irreducible} if every embedded 2-sphere $\mb{S}^2\subset M$ bounds a 3-ball $\mb{B}^3\subset M$.
\end{defi}

\begin{defi}[Atoroidal]
\label{def:atoroidal}
An orientable irreducible 3-manifold $M$ is {\em atoroidal} if every embedded $\pi_1$-injective 2-torus $\mb{T}^2\subset M$ is isotopic to a component of $\partial M$.
\end{defi}

\begin{defi}[Doubly Incompressible]
\label{def:doubly incompressible}
Let $M$ be a compact 3-manifold with boundary $\partial M\neq\emptyset$. An essential multicurve $\alpha=\alpha_1\sqcup\cdots\sqcup\alpha_n\subset\partial M$ is {\em doubly incompressible} if:
\begin{enumerate}
\item{We have 
\[
\sum_{j\le n}{i(\alpha_j,\partial A)}>0
\]
for every properly embedded essential annulus or Möbius band $(A,\partial A)\subset(M,\partial M)$.}
\item{We have 
\[
\sum_{j\le n}{i(\alpha_j,\partial D)}>2
\]
for every properly embedded essential disk $(D,\partial D)\subset(M,\partial M)$.}
\end{enumerate}

Here $i(\bullet,\bullet)$ denotes the {\em geometric intersection number} between closed curves on $\partial M$ (see \cite[Ch.\,8]{Mar16}).
\end{defi}

The notion of doubly incompressible multicurve was introduced by Thurston in \cite{Th3}. It is well-known that doubly incompressible multicurves are abundant, in fact, there is an open set $\mc{O}\subset\mc{PML}$ of the space of projective measured laminations on $\partial M$ (see Lecuire \cite{Le05}), such that every multicurve $\alpha\in\mc{O}$ satisfies such property. We refer to \cite[Ch.\,8]{Mar16} for an introduction to (projective) measured laminations.

\subsubsection{Convex hyperbolic 3-manifolds with prescribed corners}
The existence of hyperbolic structures on $(M,\partial M)$ is a consequence of Thurston's Hyperbolization Theorem:

\begin{fact}[{Thurston, see \cite[Th.\,1.43]{K01}}]
\label{thm:hyperbolization}
Let $M$ be a compact orientable irreducible atoroidal 3-manifold with non-empty boundary $\partial M\neq\emptyset$. Let $T$ be the union of all toroidal components of $\partial M$. Then $M-T$ admits a complete convex hyperbolic structure of finite volume, that is $M-T$ is homeomorphic to a quotient $\mc{C}/\Gamma$ where $\mc{C}\subset\mb{H}^3$ is a $\Gamma$-invariant convex set and $\Gamma<{\rm Isom}^+(\mb{H}^3)$ is a discrete and torsion free subgroup. 
\end{fact}

If $M$ has a convex hyperbolic metric of finite volume as in Fact \ref{thm:hyperbolization}, then it has a large deformation space of such metrics. 

In particular, Bonahon and Otal proved:

\begin{fact}[{Bonahon--Otal \cite[Th.\,2]{BO04}}]
\label{thm:bonahon otal}
Let $M$ be a compact orientable 3-manifold with non-empty boundary $\partial M\neq\emptyset$ such that $M$ is irreducible and atoroidal. Let $\alpha_1\cup\cdots\cup\alpha_n\subset\partial M$ be a multicurve and let $\theta_1,\cdots,\theta_n\in[0,\pi]$ be a set of angles. Suppose that the following properties hold: 
\begin{enumerate}
\item{We have 
\[
\sum_{j\le n}{\theta_j i(\alpha_j,\partial A)}>0
\]
for every properly embedded essential annulus or Möbius band $(A,\partial A)\subset(M,\partial M)$.}
\item{We have 
\[
\sum_{j\le n}{\theta_j i(\alpha_j,\partial D)}>2\pi
\]
for every properly embedded essential disk $(D,\partial D)\subset(M,\partial M)$.}
\end{enumerate} 

Let $\mu\subset\alpha$ be the subset of curves with angle $\pi$. Then, there exists a convex hyperbolic metric on $M-\mu$, unique up to isotopy, such that 
\begin{itemize}
\item{Each curve $\alpha_j\subset\alpha-\mu$ is a geodesic.}
\item{The closure of each component of $\partial M-\alpha$ is totally geodesic.}
\item{The pleating angle between the components adjacent to $\alpha_j$ of $\alpha-\mu$ is $\pi-\theta_j$.}
\item{Every curve $\alpha_j\subset\mu$ is a rank one cusp.}
\end{itemize}
\end{fact}

The statement of Fact \ref{thm:bonahon otal} is not phrased in the same language used in Theorem 2 of \cite{BO04}, but is equivalent to it. We formulated it in a language adapted to this paper.

Endowed with such hyperbolic metric $M$ is a hyperbolic 3-manifold with totally geodesic boundary and prescribed corners $\alpha_j\subset\partial M$ and corner angles $\theta_j$.

It is immediate to check that if a multicurve $\alpha=\alpha_1\cup\ldots\cup\alpha_n$ is doubly incompressible, then it is enough to choose angles $\theta_j$ very close to $\pi$ to be sure that the requirements of Fact \ref{thm:bonahon otal} are also fullfilled.

Note that Fact \ref{thm:bonahon otal} guarantees some freedom in the choice of the angles at the corners. This is an important feature and we will use it to obtain a certain degree of flexibility in the choice of the projective structure and holonomy of the totally geodesic boundary tori $T_j$ in Theorem \ref{thm:mainB}.

\subsubsection{Deformation space}

The space of isotopy classes of hyperbolic metrics on $M$ with totally geodesic boundary and corners or rank one cusps at the multicurve $\alpha=\alpha_1\cup\cdots\cup\alpha_n$ has a natural topology induced by the topology of algebraic convergence of the holonomy representations $\rho:\pi_1(M,\star)\to{\rm Isom}^+(\mb{H}^3)$ where $\star\in{\rm int}(M)$ is a fixed basepoint. The next result of Bonahon and Otal gives a simple description of convergence of metrics in terms of convergence of angles:

\begin{fact}[{Bonahon--Otal \cite[Th.\,24]{BO04}}]
\label{thm:bending continuous}
Let $\alpha=\alpha_1\cup\cdots\cup\alpha_n\subset\partial M$ be a doubly incompressible multicurve. The map that associates to each hyperbolic metric on $M$ with totally geodesic boundary and corners or rank one cusps at the multicurve $\alpha$ the pleating angles $(\theta_1,\cdots,\theta_n)\in(0,\pi]^n$ is a homeomorphism onto an open polytope $P_\alpha\subset(0,\pi]^n$. The image of those metrics that have rank one cusps at the curves $\alpha_{j_1},\cdots,\alpha_{j_k}$ is contained in the subset with $\theta_{j_1},\cdots,\theta_{j_k}=\pi$.
\end{fact}

Again, the statement of Fact \ref{thm:bending continuous} is not literally the one of \cite[Th.\,24]{BO04}, but is equivalent to it. 

Using a different set of parameters, Choi and Series \cite{CS06} showed that the lengths of the pleating locus of the boundary give a local parametrization of the same space of metrics considered by Bonahon and Otal \cite{BO04}. 

\begin{fact}[{Choi--Series \cite[Th.\,A]{CS06}}]
\label{thm:length continuous}
Let $\alpha=\alpha_1\cup\cdots\cup\alpha_n\subset\partial M$ be a doubly incompressible multicurve. The map that associates to each hyperbolic metric on $M$ with totally geodesic boundary and corners or rank one cusps at the multicurve $\alpha$ the pleating lengths $(\ell_1,\cdots,\ell_n)\in[0,\infty)^n$ is an injective local homeomorphism. The image of those metrics that have rank one cusps at the curves $\alpha_{j_1},\cdots,\alpha_{j_k}$ is contained in the subset with $\ell_{j_1},\cdots,\ell_{j_k}=0$.
\end{fact}

Fact \ref{fact:ChoiSeries} from the introduction is a combination of Facts \ref{thm:bending continuous} and \ref{thm:length continuous} (note that the pleating angle is $\pi$ minus the corner angle).

\subsection{Dimension $d\ge 3$}

In this section we prove Theorem \ref{thm:main7}.

For convenience of the reader, we briefly recall our goal: We want to find a closed hyperbolic $d$-manifold $M$ which contains $k$ totally geodesic connected embedded hypersurfaces $N_1,\cdots,N_k$ whose pairwise intersection $C=N_i\cap N_j$ is a (fixed) codimension 2 connected totally geodesic submanifold and such that the angles of intersection are $\angle N_jN_{j+1}=\pi/k$ for every $j\le k$ (indices modulo $k$). Furthermore, we want that the hypersurfaces $N_j$ are pairwise isometric and related by a cyclic isometry $\rho$ of $M$ that fixes $C$ pointwise and maps $\rho(N_j)=N_{j+1}$.

When the dimension is $d\ge 4$, hyperbolic $d$-manifolds with patterns of totally geodesic hypersurfaces can be obtained by using arithmetic techniques and separability properties of arithmetic lattices as in Gromov-Thurston \cite{GT87} and Kapovich \cite{K07} and, in fact, our construction is inspired by those. Let us point out that, while we have little control on the topology and geometry of these examples, we can, crucially, make sure that two properties are satisfied: The pattern and the angles at the intersections are controlled. 

Our strategy is as follows: For every $k$ we consider an explicit cocompact arithmetic lattice $G<{\rm SO}_0(d,1)$ defined in terms of a suitable quadratic form $q$. We first construct explicitly a convex-cocompact subgroup $Q<G$ such that $\mb{H}^d/Q$ contains, by construction, the desired configuration of $k$ totally geodesic embedded hypersurfaces $N_1,\cdots, N_k$ intersecting exactly along a codimension 2 totally geodesic corner $C=N_1\cap\cdots\cap N_k$ with angles of $\angle N_jN_{j+1}=\pi/k$. Then we use strong separability properties of $G$ due to Bergeron, Haglund, and Wise \cite{BHW11} to embed $N_1\cup\cdots\cup N_k$ in a finite cover of $\mb{H}^d/G$.

We begin with the description of some standard arithmetic manifolds which is the starting point of the construction.

\subsubsection{Arithmetic manifolds from quadratic forms}
From now on let $k\ge 2$ be a fixed integer.

Consider the number field $F:=\mb{Q}(\cos(\pi/k),\sin(\pi/k))$. It is a {\em totally real} extension of $\mb{Q}$, that is, every field embedding $F\to\mb{C}$ has image contained in $\mb{R}$. We denote by $\sigma_1,\cdots,\sigma_r$ the $r={\rm deg}(F/\mb{Q})$ real embeddings $\sigma_j:F\to\mb{R}$ where $\sigma_1$ is the trivial one $\sigma_1(x)=x$ for every $x\in F$. 

Let $\mc{O}\subset F$ be the ring of integers of the field. It is a standard fact that every $x\in F$ can be written as $x=\frac{a}{b}$ with $a\in\mc{O}$ and $b\in\mb{N}$.

\begin{lemma}
\label{lem:number field}
There exists $\tau\in\mc{O}$ such that 
\begin{enumerate}
\item{$F=\mb{Q}(\tau)$.}
\item{$\tau=\sigma_1(\tau)>0$.}
\item{$\sigma_j(\tau)<0$ for every $j\ge 2$.}
\end{enumerate}
\end{lemma}

\begin{proof}
Choose $\tau_1,\cdots,\tau_r\in F$ with $\sigma_j(\tau_j)$ pairwise distinct and with $\sigma_1(\tau_1)>0$ and $\sigma_j(\tau_j)<0$ for every $j>1$. Let $\ep>0$ be smaller than $\min_{j\le r}\{|\sigma_j(\tau_j)|\}$ and $\min_{i<j}\{|\sigma_i(\tau_i)-\sigma_j(\tau_j)|\}$. By the Weak Approximation Theorem (see \cite[Th.\,7.20]{Milne}) applied to the $r$ absolute values $|\bullet|_j:=|\sigma_j(\bullet)|$, there exists $\tau\in F$ such that $|\sigma_j(\tau)-\sigma_j(\tau_j)|<\ep$. By our choice of $\ep$, we have that $\sigma_1(\tau),\cdots,\sigma_r(\tau)$ are pairwise distinct, in particular, $\tau$ has degree $r={\rm deg}(F/\mb{Q})$ over $\mb{Q}$ which implies $F=\mb{Q}(\tau)$. Furthermore, by our choice of $\ep$, we also have $\sigma_1(\tau)>0$ and $\sigma_j(\tau)<0$ for every $j>1$.
\end{proof}

Consider the quadratic form of signature $(n,1)$
\[
q(x):=x_1^2+\ldots+x_d^2-\tau x_{d+1}^2.
\]
Let ${\rm SO}(q)_{\mc{O}}$ be the group of orientation preserving isometries of $q$ with coefficients in $\mc{O}$ that preserve the hyperboloid
\[
\mb{H}^d:=\{x\in\mb{R}^{d,1}\left|q(x)=-1,x_{d+1}>0\right.\}.
\]

\begin{lemma}
\label{lem:cocompact}
${\rm SO}(q)_{\mc{O}}<{\rm SO}(q)$ is discrete and acts cocompactly on $\mb{H}^d$.
\end{lemma}

\begin{proof}
This comes from the so-called restriction of scalars construction (see \cite[Ex.\,5\,\S 2.1]{Be5lectures}) which requires that the quadratic forms 
\[
q^{\sigma_j}:=x_1^2+\ldots+x_d^2-\sigma_j(\tau)x_{d+1}^2.
\]
are all positive definite for $j>1$ while $q^{\sigma_1}$ has signature $(d,1)$ (as ensured by Lemma \ref{lem:number field}).
\end{proof}

We will work with closed oriented hyperbolic $d$-manifolds that are quotients of $\mb{H}^d$ by torsion free finite index subgroups of ${\rm SO}(q)_{\mc{O}}$. Such groups are abundant by Selberg's Lemma and Malcev's Theorem (see \cite{N13}). The group ${\rm SO}(q)_{\mc{O}}$ has also much stronger subgroup separability properties and we will need them later on: 

\begin{fact}[{Bergeron--Haglund--Wise \cite[Th.\,1.4]{BHW11}}]
\label{bhw}
There exists a torsion free finite index subgroup $G<{\rm SO}(q)_{\mc{O}}$ such that every {\rm convex-cocompact} subgroup of $G$ is {\rm separable}. 
\end{fact}

Let us fix such a $G<{\rm SO}(q)_{\mc{O}}$.

We briefly recall the definition of separability: 

\begin{defi}[Separable]
A subset $S\subset A$ of a group $A$ is {\em separable} if for every $g\in A-S$ there exists a finite index subgroup $A'<A$ such that $S\subset A'$ and $g\not\in A'$. 
\end{defi}

Consider the hyperplane
\[
H:=\{x\in\mb{R}^{d,1}\left|x_1=0\right.\}
\]
and the codimension 2 plane 
\[
V:=\{x\in\mb{R}^{d,1}\left|x_1=x_2=0\right.\}.
\]

Let $\rho$ be the $\pi/k$-rotation around $V$
\[
\rho:=\left(
\begin{array}{c c c}
\cos(\pi/k) &-\sin(\pi/k) &\\
\sin(\pi/k) &\cos(\pi/k) &\\
& &\mb{I}_{d-1}\\
\end{array}
\right).
\]

Note that $\rho\in{\rm SO}(q)_F$. This implies that: 

\begin{lemma}
\label{lem:congruence}
The subgroup
\[
G\cap\rho G\rho^{-1}\cap\cdots\cap\rho^{2k-1} G\rho^{-(2k-1)}
\]
has finite index in $G$ and is normalized by $\rho$.
\end{lemma}

\begin{proof}
Recall that every $x\in F$ can be written as $x=\frac{a}{b}$ where $a\in\mc{O}$, the ring of integers of $F$, and $b\in\mb{N}$. For every $1\le j\le 2k$, write every entry of $\rho^j$ as a fraction of an element in $\mc{O}$ and an element in $\mb{N}$ and let $D\subset\mb{N}$ be the finite set of the denominators (note that $\rho^{-j}=\rho^{2k-j}$ for every $0\le j\le 2k$). Let $\delta\in\mb{N}$ be the product of all the elements in $D$.

Let $I<\mc{O}$ be the ideal generated by $\delta^2$. 

In particular, as $\mc{O}$ is finitely generated as a $\mb{Z}$-module (see \cite[Cor.\,2.30]{Milne}), $\mc{O}/I$ is a finite extension of the finite ring $\mb{Z}/\delta^2\mb{Z}$ and, hence, it is itself a finite ring. 

Consider the congruence subgroup of $G$
\[
G_\delta:={\rm ker}(G<{\rm SO}(q)_{\mc{O}}\to{\rm SL}_{n+1}(\mc{O}/I)).
\]
$G_\delta$ has finite index in $G$. Every matrix $A\in G_\delta$ can be written as $A=I+B$ with every entry of $B$ of the form $\delta^2u$ for some $u\in\mc{O}$. Therefore, the matrix $\rho^jA\rho^{-j}=I+\rho^jB\rho^{-j}$ has coefficients in $\mc{O}$ and preserves $q$. In other words, the conjugation by $\rho^j$ induces an injective homomorphism $G_\delta\to{\rm SO}(q)_{\mc{O}}$. 

Note that $\rho^jG_\delta\rho^{-j}<{\rm SO}(q)_{\mc{O}}$, being cocompact in ${\rm SO}(q)$, has finite index in ${\rm SO}(q)_{\mc{O}}$. 

As the intersection $G\cap\rho G\rho^{-1}\cap\cdots\cap\rho^{2k-1} G\rho^{-(2k-1)}$ contains the intersection of finite index subgroups $G_\delta\cap\rho G_\delta\rho^{-1}\cap\cdots\cap\rho^{2k-1} G_\delta\rho^{-(2k-1)}$, it has finite index in $G$. 
\end{proof}

Recall that separability passes to subgroups, meaning that, if $S$ is separable in $A$ and $B<A$, then $S\cap B$ is separable in $B$. Hence, up to replacing $G$ with the intersection given by Lemma \ref{lem:congruence}, we can assume that $G$ is normalized by $\rho$. From now on we assume this is the case.

\subsubsection{A ping-pong argument}
We now construct a convex-cocompact subgroup $Q<G$ such that $\mb{H}^d/Q$ contains the desired configuration of hypersurfaces. 

Recall that $H\subset\mb{H}^d$ is the hyperplane $\{x\in\mb{R}^{d,1}\left|x_1=0\right.\}$.

Let $K$ denote the isometry group $K:=G\cap {\rm SO}(q|_H)_{\mc{O}}$ where we identify ${\rm SO}(q|_H)_{\mc{O}}$ with its natural inclusion in ${\rm SO}(q)_{\mc{O}}$. Observe that $K$ acts cocompactly on $H$ as ${\rm SO}(q|_H)_{\mc{O}}$ act cocompactly on $H$ (by Lemma \ref{lem:cocompact}) and $K$ is a finite index subgroup of it.

We have:

\begin{prop}
\label{pro:cover}
There exists a finite set $A\subset K$ such that for every finite index subgroup of $T<K$ with $T\cap A=\emptyset$ we have the following: The group 
\[
Q=\langle T_1,\cdots,T_k\rangle<G,
\]
where $T_j:=\rho^j T\rho^{-j}$ leaves invariant the hyperplane $H_j:=\rho^j H$, has the following properties:
\begin{enumerate}
\item{It is convex-cocompact.}
\item{The projection $\pi:H_j\to\mb{H}^d/Q$ induces an embedding $H_j/T_j\to\mb{H}^d/Q$.}
\item{For every $1\le i<j\le k$ we have $H_i/T_i\cap H_j/T_j=V/P$ where $P={\rm Stab}_T(V)$.}
\end{enumerate} 
\end{prop}

\begin{proof}
The idea of the proof is to construct a suitable ping-pong system adapted to the groups $T_1,\cdots,T_k$ and hyperplanes $H_1,\cdots,H_k$.

The group $Q$ is a subgroup of $G$ because $G$ is normalized by $\rho$ and $T<G$. 

Let us consider a very large $R>0$ which we will adjust for our needs later on.

First, we select a finite index subgroup $T<K$ so that 
\[
d_{\mb{H}^d}(V,\alpha V)\ge R.
\]
for every $\alpha\in T-P$ where $P={\rm Stab}_T(V)$. 

We can find such subgroup by using hyperplane separability properties: Let $W\subset V$ be a compact fundamental domain for the action of $P$ on $V$. Consider the finite set
\[
A_R=\{\gamma\in K-P\left|\;d_{\mb{H}^d}(W,\gamma W)<R\right.\}.
\]
By \cite[Lem.\,1.7]{BHW11} the subgroup $P<K$ is separable, thus there exists a finite index subgroup $T<K$ containing $P$ and such that $T\cap A_R=\emptyset$. Such subgroup satisfies the assumption: If we had $d_{\mb{H}^d}(V,\gamma V)\le R$ for some $\gamma\in T-P$, then we would have $d_{\mb{H}^d}(\alpha W,\gamma\beta W)\le R$ for some $\alpha,\beta\in P$ and, hence, $d_{\mb{H}^d}(W,\alpha^{-1}\gamma\beta W)\le R$. As a consequence, by our choice of $A_R$, we would get $\alpha^{-1}\gamma\beta\in P$ and, hence, $\gamma\in P$ contradicting the initial assumption.

Consider the two hyperplanes $H_j',H_j''$ containing $V$ that form with $H_j$ angles of $\pi/2k$. The hyperplanes $H_j',H_j''$ divide $\mb{H}^d$ into four quadrants. Let us consider the union $U_j$ of two opposite closed quadrants containing $H_j$. Notice that, by the angle assumption, we have $U_i\cap U_j=V$ for every $i<j$. In particular, the interiors of the sectors $U_1,\cdots, U_k$ are pairwise disjoint.

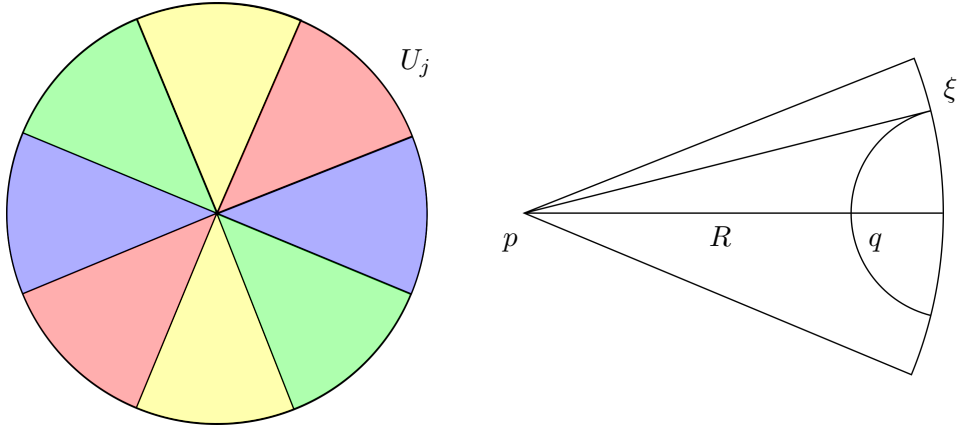
\begin{figure}[h]
\label{fig:ping-pong}
\centering 
\begin{overpic}[scale=0.7]{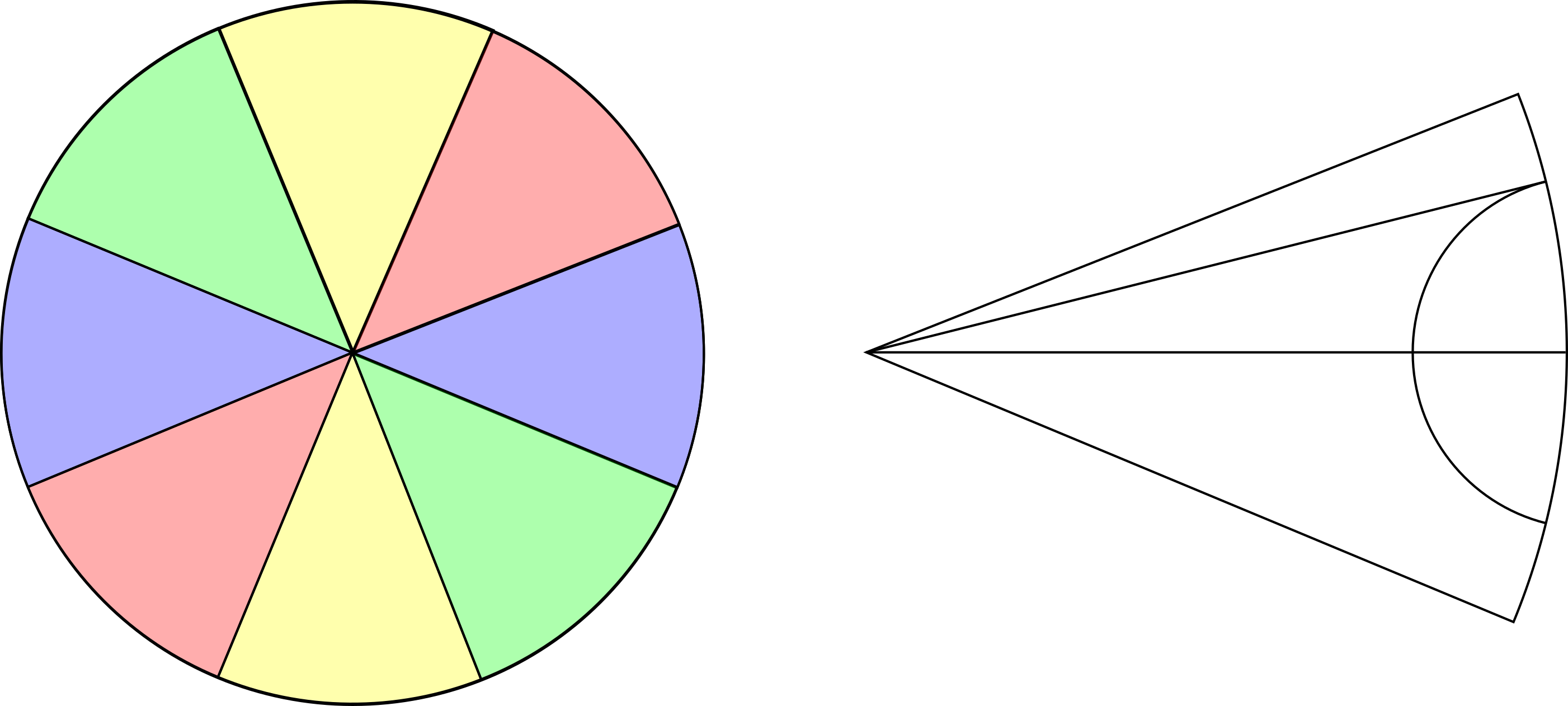}
\put (42,38) {$U_j$}
\put (75,19) {$R$}
\put (92,19) {$q$}
\put (53,19) {$p$}
\put (100,35) {$\xi$}
\end{overpic}
\caption{Ping-pong system.}
\end{figure}

We use these sectors as a ping-pong system.

\begin{claim}{1}
We have:
\begin{enumerate}[(i)]
\item{If $\beta\in P$, then $\beta(U_j)=U_j$.}
\item{Let $R$ be big enough. If $\alpha_i\in T_i-P$ and $i\neq j$, then $\alpha_i(U_j)\subset{\rm int}(U_i)$.}
\end{enumerate} 
\end{claim}

\begin{proof}[Proof of the claim] 
The first part is clear.

For the second part, consider $H_i,H_j,\alpha_i H_j$. We need the following:

\begin{lemma} 
\label{lem:orthogonal plane}
There is a 2-plane $Z\subset\mb{H}^d$ that intersects orthogonally all of them.
\end{lemma}

\begin{proof}
We work in the projective model $\mb{H}^d\subset\mb{RP}^d$. 

With a slight abuse of notation, unless it is not necessary, we will not make a distinction between a totally geodesic subspace of $\mb{H}^d$ and the corresponding linear subspace of $\mb{RP}^d$. Also recall that every linear subspace $L\subset\mb{RP}^d$ has a dual $L^*\subset\mb{RP}^d$ which corresponds to the orthogonal subspace $L^\perp$ with respect to the quadratic form $q$.

A 2-plane $Z\subset\mb{H}^d$ is orthogonal to a hyperplane $H\subset\mb{H}^d$ if and only if $Z$ passes through the dual point $H^*$.

Let $x_i,x_j,\alpha_jx_i$ be the dual points of $H_i,H_j,\alpha_jH_i$. Let 
\[
Z:={\rm Span}\{x_i,x_j,\alpha_jx_i\}
\]
be the linear subspace of $\mb{RP}^d$ they generate. Notice that $Z$ is dual to the intersection (as linear subspaces) $H_i\cap H_j\cap\alpha_j H_i$. The triple intersection of the hyperplanes (as linear subspaces) avoids the closure of $\mb{H}^d$ as 
\[
(H_i\cap H_j\cap\alpha_j H_i)\cap\bar{\mb{H}}^d=(V\cap\alpha_jV)\cap\bar{\mb{H}}^d=\emptyset.
\]
In particular, the restriction of the quadratic form $q$ to the triple intersection is positive definite. As a consequence, the restriction of the quadratic form $q$ on the dual space $Z$ has signature $(2,1)$ which implies that $Z$ intersects $\mb{H}^d$ in a totally geodesic $\mb{H}^2$. Furthermore, by construction, $Z$ passes through the dual points $x_i,x_j,\alpha_jx_i$, hence, it is orthogonal to the subspaces $H_i,H_j,\alpha_jH_i$.
\end{proof}

Let $Z$ be the 2-plane provided by Lemma \ref{lem:orthogonal plane}.

Consider the intersection points $p=Z\cap V,q:=Z\cap\alpha_i V$. 

By the assumption on $T$, we have that $d_{\mb{H}^d}(p,q)\ge d_{\mb{H}^d}(V,\alpha_iV)\ge R$.

Consider also the intersections at infinity $\{\xi,\zeta\}:=\partial Z\cap\partial\alpha_iH_j$. The angles $\angle pq\xi,\angle pq\zeta$ are respectively $|j-i|\pi/k$ and $\pi-|j-i|\pi/k$. Consider the triangles $\Delta(p,q,\xi)$ and $\Delta(p,q,\zeta)$. By hyperbolic trigonometry (see \cite{}), the angles at $p$ are very small: By the hyperbolic law of cosines and a small manipulation we have 
\begin{align*}
\sin\left(\angle qp\xi\right),\sin\left(\angle qp\zeta\right) &\le\frac{2}{\cosh\left(d_{\mb{H}^n}(p,q)\right)\sin\left(|j-i|\frac{\pi}{k}\right)}\\
 &\le\frac{2}{\cosh(R)\sin\left(|j-i|\frac{\pi}{k}\right)}.
\end{align*}

If $R$ is chosen so that the angles $\angle qp\xi,\angle qp\zeta$ are much smaller that $\pi/2k$, then $\alpha_i U_j$ is contained in $U_i$.
\end{proof}

Using the ping-pong system, we deduce the properties that we need.

{\bf Properties (2) and (3)}. Let us pick an element $\gamma\in Q$. We can always write it as $\gamma=\alpha_{j_1}\cdots\alpha_{j_n}$ where $\alpha_{j_k}\in T_{j_k}-P$ and $\alpha_{j_n}\in T_{j_n}$ and such that $j_k\neq j_{k+1}$ for every $k$. 

Let us observe that, by the first claim, we have $\gamma H_i\subset U_{j_1}$ for every $H_i$ and, furthermore, $\gamma H_i\subset{\rm int}(U_{j_1})$ if $n>1$: In fact, consider first $\alpha_{j_n}H_i$ there are two cases: If $j_n=i$, then $\alpha_{j_n}H_i=H_i$ and, hence, it is contained in $U_i=U_{j_n}$. Otherwise, by the Claim, $\alpha_{j_n}H_i\subset \alpha_{j_n}U_i\subset U_{j_n}$. The conclusion follows by inductively applying the Claim using the fact that $j_k\neq j_{k+1}$:
\begin{align*}
\gamma H_i &=\alpha_{j_1}\cdots\alpha_{j_n}H_i\\
 &\subset\alpha_{j_1}\cdots\alpha_{j_{n-1}}U_{j_n}\\
 &\subset\alpha_{j_1}\cdots\alpha_{j_{n-2}}{\rm int}(U_{j_{n-1}})\\
 &\subset\cdots\\
 &\subset\alpha_{j_1}U_{j_2}\subset{\rm int}(U_{j_1}).
\end{align*}

We can now prove the following

\begin{claim}{2}
If $\gamma H_i\cap H_j\neq\emptyset$, then $\gamma=\alpha_j\alpha_i$ with $\alpha_i\in T_i$ and $\alpha_j\in T_j$.
\end{claim}

\begin{proof}[Proof of the claim] 
We proceed by induction on the length of the shortest representative of $\gamma$ of the above form $\gamma=\alpha_{j_1}\cdots\alpha_{j_n}$. The base case $n=1$ is clear. If $n>1$ then, by the above computation, $\gamma H_i\subset{\rm int}(U_{j_1})$. As $\gamma H_i\cap H_j\neq\emptyset$, we must have $j_1=j$. Consider $\gamma'=\alpha_{j_1}^{-1}\gamma$. The element $\gamma'$ has by construction a shorter representative than $\gamma$ and satisfies $\gamma' H_i\cap H_j=\alpha_{j_1}^{-1}(\gamma H_i\cap H_j)\neq\emptyset$. Therefore, by the inductive hypothesis, we have $\gamma'=\alpha_j'\alpha_i'$ with $\alpha_j'\in T_j$ and $\alpha_i'\in T_i$. Finally, we conclude $\gamma=(\alpha_{j_1}\alpha_j')\alpha_i'$ which concludes the inductive step as $\alpha_{j_1}\alpha_j'\in T_j$.
\end{proof}
 
\begin{claim}{3}
For every $j\le k$, the natural map $H_j/T_j\to\mb{H}^d/Q$ is an embedding. 
\end{claim}

\begin{proof}[Proof of the claim] 
Suppose that $\pi(x)=\pi(y)$ for $x,y\in H_j$. Then, there exists $\gamma\in Q$ such that $y=\gamma x$ and, in particular, $\gamma H_j\cap H_j\neq\emptyset$. By the previous claim, we deduce that $\gamma\in T_j$ and, hence $[x]=[y]$ in $H_j/T_j$.
\end{proof}

\begin{claim}{4} 
For every $i<j$ we have $H_i/T_i\cap H_j/T_j=V/P$.  
\end{claim}

\begin{proof}[Proof of the claim] 
Suppose that $\pi(x)=\pi(y)$ for $x\in H_i,y\in H_j$. Then, there exists $\gamma\in Q$ such that $y=\gamma x$ and, in particular, $\gamma H_i\cap H_j\neq\emptyset$. By the previous claim, we deduce that $\gamma=\alpha_j\alpha_i$ with $\alpha_j\in T_j$ and $\alpha_i\in T_i$. In order to coclude we just observe that 
\[
\alpha_j^{-1}(\gamma H_i\cap H_j)=\alpha_j^{-1}(\alpha_j\alpha_i H_i\cap H_j)=H_i\cap H_j=V
\]
so that $y\in\gamma H_i\cap H_j=\alpha_jV$ and $x=\gamma^{-1}y\in\alpha_iV$. Therefore $\pi(x)=\pi(y)\in V/P$.
\end{proof}

This concludes the proof of the properties (2) and (3).

{\bf Property (1)}. By Fact \ref{fact:quasi-convex}, it is enough to prove that the orbit of a point $p\in V$ is quasi-convex provided that $R$ is large enough. As the orbit is coarsely dense in $H_j$ for each $j\le k$, it is enough to show that 
\[
C:=\bigcup_{\gamma\in Q,\;j\le k}{\gamma H_j}
\]
is quasi-convex. 

Notice that $C$ has an intrinsic path metric. Two distinct points of $C$ can be connected by either a geodesic segment of $\mb{H}^d$, if they lie on the same flat piece $\gamma H_i$, or by a concatenation $\kappa=\kappa_1\star\ldots\star\kappa_m\subset C$ of geodesic segments of $\mb{H}^d$ which is geodesic for $C$ and where each $\kappa_j$ lies on some translate $\gamma_j H_{i_j}$ and has endpoints on two different translates of $V$ in $\gamma_j H_{i_j}$. We show that such intrinsic $C$-geodesics $\kappa$ are a uniform quasi-geodesics and, therefore, by stability of quasi-geodesics, they lie at a uniform Hausdorff distance from the geodesic in $\mb{H}^d$ joining the same endpoints. This suffices to prove that $C$ is quasi-convex.

In order to prove that $\kappa$ is a uniform quasi-geodesic we need a couple of observations: First, the length of each $\kappa_j$ is at least $R$ as the geodesic segment joins distinct translates of $V$. Second, as the concatenation is a geodesic for the intrinsic path metric of $C$, the angle between two consecutive segments $\kappa_j,\kappa_{j+1}$ is at least the angle between $\gamma_jH_{i_j},\gamma_{j+1}H_{i_{j+1}}$ which is at least $\pi/k$. The conclusion follows from the following elementary fact

\begin{fact}
For every $\theta>0$ there exist $R>0$ and $c>0$ such that each concatenation $\kappa\subset\mb{H}^d$ of geodesic segments of length at least $R$ forming angles of at least $\theta$ is a $c$-quasi-geodesic.
\end{fact}

This concludes the proof.
\end{proof}

We are now ready to prove Theorem \ref{thm:main7}.

\subsubsection{The proof of Theorem \ref{thm:main7}}
We prove part (a).

Let $Q<G$ be the subgroup provided by Proposition \ref{pro:cover}. By Fact \ref{bhw}, $Q$ is separable in $G$. A standard argument shows that:  

\begin{claim*}
There is a finite index subgroup $G'<G$ containing $Q<G'$ and such that the covering projection $\mb{H}^d/Q\to\mb{H}^d/G'$ restricts to an embedding on $N_1\cup\cdots\cup N_k\subset\mb{H}^d/Q$ where  $N_j:=H_j/T_j$. 
\end{claim*}

\begin{proof}[Proof of the claim] 
By Proposition \ref{pro:cover}, $Q$ acts cocompactly on a convex set $\mc{CH}\subset\mb{H}^d$. The quotient $\mc{CH}/Q$ embeds in $\mb{H}^d/Q$ and contains $H_j/T_j$. Consider the covering projections $\pi:\mb{H}^d/Q\to\mb{H}^d/G$ and the restriction of $\pi$ to $\mc{CH}/Q$. Observe that $\pi(p)=\pi(q)$ for some $p,q\in\mc{C}/Q$ if and only if there exists lifts $x,y\in F$ of $p,q$ to some compact fundamental domain $F\subset\mc{CH}$ and an element $\gamma\in G$ such that $\gamma x=y$. 

We use separability of $Q$ to eliminate this configuration in an intermediate finite covering $\mb{H}^d/Q\to\mb{H}^d/G'\to\mb{H}^d/G$: Consider the finite set
\[
A:=\{\gamma\in G-Q\left|\gamma F\cap F\neq\emptyset\right.\}. 
\]
By separability of $Q$ in $G$, there exists a finite index subgroup $G'<G$ containing $Q$ and such that $G'\cap A=\emptyset$. As $G'$ contains $Q$, we have a covering projection $\pi':\mb{H}^d/Q\to\mb{H}^d/G'$. By our choice of $A$ and the fact that $G'\cap A=\emptyset$, the restriction of $\pi'$ to $\mc{CH}/Q$ is now injective.
\end{proof}

By Proposition \ref{pro:cover} and the fact that the projection $\mb{H}^d/Q\to\mb{H}^d/G'$ restricts to an embedding on $N_1\cup\cdots\cup N_k\subset\mc{C}/Q$, the abstract completion $M$ of any of the complementary components of $\mb{H}^d/G'-N_1\cup\cdots\cup N_k$ is a compact hyperbolic $d$-manifold with totally geodesic boundary and corners, all the corner angles are equal to $\pi/k$, and, if $k$ is even, then the graph of the boundary is bipartite.

This finishes the proof.\qed

\section{Hyperbolic doubles}
\label{sec:sec7}

In this section we prove Theorems \ref{thm:mainA}, \ref{thm:mainB}, and \ref{thm:main2}.

We will deduce them from the following special case of the glue-and-bend construction, namely, the double of a compact convex hyperbolic manifold with totally geodesic boundary bent along some corners.

\begin{thm}
\label{thm:hyperbolic doubles}
Let $\mb{M}$ be a compact convex hyperbolic bc-manifold with non-empty boundary of dimension $\dimd\ge 3$. Fix a parameter for each wall. Suppose that
\begin{enumerate}[(i)]
\item{All the corner angles are smaller than $\pi/2$.}
\item{For every corner $\corner$ with angle $\theta$ and adjacent walls $W$ and $W'$ with parameters $\mu$ and $\mu'$, denoting $\nu=\mu^{\frac{\dimd+1}{2}}$ and $\nu'=\mu'{}^{\frac{\dimd+1}{2}}$,
\[
t:=\cos(\theta)^2\left(\frac{\nu}{\nu'}+\frac{\nu'}{\nu}\right)-\sin(\theta)^2\left(\nu'\nu+\frac1{\nu'\nu}\right)\geq 2.
\]
Set 
$
 \mu_\corner := \max\left\{\frac{\mu}{\mu'},\frac{\mu'}{\mu}\right\}
$
and $\lambda_\corner:=\frac {t}2+\sqrt{\frac{t^2}4-1}$.
}
\end{enumerate}

Then the following holds:
\begin{enumerate}
\item{$N:=D\mb{M}-U$ obtained from the double $D\mb{M}$ of $\mb{M}$ by removing a tubular neighborhood $U$ of the corners admits a convex projective structure with totally geodesic boundary.}

\item{For each wall $W$ of $N$ there is a developing map of $N$ sending a copy of $\tilde W$ onto
$
\conv(\mb H^{\dimd-2},[e_{\dimd}])
$
such that the holonomy of $\pi_1(W)$ is the group generated by
\[
\left(
\begin{array}{c c}
\Gamma_\corner & \\
 &\mb{I}_2\\
\end{array}
\right), 
\left(
\begin{array}{ccc}
 \mu_\corner^{-1}\mb{I}_{d-1}&& \\
 &\mu_\corner^{\frac{\dimd-1}2}\lambda_\corner&\ep\\
 &&\mu_\corner^{\frac{\dimd-1}2}\lambda_\corner^{-1}\\
\end{array}
\right)\in{\rm SL}_{d+1}(\mb{R}),
\]  
with $\ep=1$ if $\lambda=1$ and $0$ otherwise and where $\corner$ is the corner $W$ blows up, with holonomy $\Gamma_{\corner}<{\rm SO}(d-2,1)$.}
\item{If $\lambda_\corner>1$ then $N$ satisfies the Containment Condition with respect to $W$ and the polar $[e_{\dimd+1}]$.}
\item{Every pair of distinct walls of $\tilde N\subset\sph^\dimd$ have disjoint closures (\ie $N$ has no ghost stratum) and every point of $\partial\tilde N$ outside closures of walls is extremal and $\cal C^1$.}
\end{enumerate}
\end{thm}

Before proving the main result, let us fix some notations.

\subsection{Some notations}
For every $\theta\in[0,2\pi)$, we will denote by $r_\theta,R_\theta$ the rotations
\begin{equation}
r_\theta:= 
\left(
\begin{array}{c c}
\cos(\theta) & -\sin(\theta)\\
\sin(\theta) & \cos(\theta)\\
\end{array}
\right)
{\scriptstyle\in{\rm SL}_2(\mb{R})}
,\quad 
R_\theta:=
\left(
\begin{array}{c c}
\mb{I}_{d-1} & \\
  &r_\theta\\
\end{array}
\right)
{\scriptstyle\in{\rm SL}_{d+1}(\mb{R}).}
\end{equation}

For every $\mu>0$, we will also denote by $b_\mu,b'_\mu,B_\mu$ the bulging matrices
\begin{equation}
b_\mu:=
\left(
\begin{array}{c c}
\frac{1}{\mu} &0\\
0 &\mu\\
\end{array}
\right){\scriptstyle \in{\rm SL}_2(\mb{R})},\quad
b'_\mu:=
\left(
\begin{array}{c c}
\frac{1}{\mu} &0\\
0 &\mu^d\\
\end{array}
\right)
{\scriptstyle\in{\rm GL}^+_2(\mb{R})},
\end{equation}
and
\begin{equation}\label{eq:Bmu}
B_\mu:=
\left(
\begin{array}{c c}
\frac{1}{\mu}\mb{I}_d & \\
 &\mu^d\\
\end{array}
\right)
=
\left(
\begin{array}{c c}
\frac{1}{\mu}\mb{I}_{d-1} & \\
 &\mu^{\frac{d-1}{2}}b_{\mu^{(d+1)/2}}\\
\end{array}
\right)
{\scriptstyle\in{\rm SL}_{d+1}(\mb{R}).}
\end{equation}

\subsection{Projective structures on hyperbolic doubles}

The basic computation behind the proof of Theorem~\ref{thm:hyperbolic doubles} is that of the holonomy of a meridian around each corner, and the tube-type of the corner.
Let us use the notation from Theorem~\ref{thm:hyperbolic doubles}, and further denote by $\mb{M}'$ the copy of $\mb{M}$ that we glue to $\mb{M}$, and by $M$ and $M'$ the smooth loci of these manifolds.

We bulge the gluing $DM$ of $M$ with its double $M'$ (and hence obtain a projective structure on $DM$) using the parameters of Theorem~\ref{thm:hyperbolic doubles} for the walls of $M$, and using trivial parameters equal to $1$ for the walls of $M'$.

\subsubsection{Holonomy of meridians}\label{sec:holonomy computation}

Consider a corner $\corner$ of $\mb{M}$  with angle $\theta$ and adjacent walls $W$ and $W'$ with bulging parameters $\mu$ and $\mu'$.

Set $\nu=\mu^{(\dimd+1)/2}$ and $\nu'=\mu'{}^{(\dimd+1)/2}$.

Let $\mb{U}\subset D\mb{M}$ be an open neighborhood of a point of $\corner$  such that $\mb{U}\cap \mb{M}$ and $\mb{U}\cap \mb{M}'$ are both small hyperbolic balls (hence convex).

Let $\tilde U\subset \tilde{T}$ be the lift of $U$, which identifies with any lift of $U$ in $\tilde{DM}$.
It is moreover the gluing of sequences of lifts $(V_n)_{n\in\Z}$ and $(V'_n)_{n\in\Z}$ of respectively $U\cap M$ and $U\cap M'$, along lifts $(W_n)_{n\in\Z}$ and $(W_n')_{n\in\Z}$ of $U\cap W$ and $U\cap W'$, such that $V_n\cap V'_n=W_n$ and $V_n'\cap V_{n+1}=W_{n+1}'$.

Consider a hyperbolic developing chart $\phi$ of $\mb{U}\cap \mb{M}$ that sends $\mb{M}\cap\corner$ in $\H^{\dimd-2}$, sends $\mb{U} \cap W'$ in $\conv(\mb H^{\dimd-2},[e_{\dimd}])$, and sends $\mb{U}\cap \mb{M}$ in $\conv(\mb H^{\dimd-1},[e_{\dimd+1}])$.

This determines a developing map $\dev$ for the (bulged) projective structure on $\tilde{U}$. An elementary computation shows the following:

\begin{lemma}
\label{lem:holonomy computation}
Let $\gamma$ be a meridian around $\corner$ that starts on $U\cap W'$, goes through $U\cap M$, intersects $U\cap W$, then goes through $U\cap M'$ back to its starting point.
\begin{itemize}
\item{The $\SL_{\dimd+1}(\R)$-holonomy of $\gamma$ is
$
\rho(\gamma)=R_{\theta} B_{\mu} R_{\theta} B_{\mu'}^{-1}
$.
}
\item{The $\SL^\pm_2$-angle of $\rho(\gamma)$ (Definition~\ref{def:tubes}) is
\[
g(\gamma)=r_{\theta} b_{\nu} r_{\theta} b_{\nu'}^{-1}.
\]
}
\item{The projection of $\dev(V_0\cup V'_0)$ in $\sph^1=\sph(\R^{d+1}/\R^{d-1})$ is the properly convex segment between $[e_{\dimd}]$ and $r_{\theta} b_{\nu} r_{\theta}b_{\nu'}^{-1}[e_{\dimd}]$.
}
\end{itemize}
\end{lemma}

\subsubsection{The proof of Theorem~\ref{thm:hyperbolic doubles}}\label{sec:pfTh8.1}

The plan is the following.
\begin{itemize}
 \item We check that the conditions of Theorem~\ref{thm:cvx} are satisfied.
 \item We apply Theorem~\ref{thm:cvx} to obtain that $DM$ is convex.
 \item We check that the conditions of Theorem~\ref{thm:addendum1} are satisfied, in particular that $DM$ is properly convex.
 \item We apply Theorem~\ref{thm:addendum1} to obtain the totally geodesic blowup $N$ of $D\mb{M}$, which has no ghost stratum and satisfy the Containment Condition at walls blowing up singularities with hyperbolic $\SL_2$-angle.
 \item We prove that every point of $\partial\tilde N$ outside closures of walls is extremal and $\cal C^1$ using Theorem~\ref{thm:addendum2}, Corollary~\ref{cor:extremal points} and Observation~\ref{obs:smooth}.
\end{itemize}

\begin{proof}[Proof of Theorem~\ref{thm:hyperbolic doubles}]
Let us check that the assumptions of Theorem~\ref{thm:cvx} are verified (this is where we use Lemma~\ref{lem:holonomy computation}):

\begin{claim}{1}
$DM$ is connected.
\end{claim}

\begin{proof}[Proof of the claim]
This is clear.
\end{proof}

\begin{claim}{2}
$\mb{M}$ is properly convex, with complete walls and corners, and without ghost strata.
\end{claim}

\begin{proof}[Proof of the claim]
This follows from our assumption and Fact~\ref{obs:hyperbolic ghost stratum}.
\end{proof}

\begin{claim}{3}
Every wall of $\mb{M}$ satisfies the Containment Condition relatively to its hyperbolic polar.
\end{claim}

\begin{proof}[Proof of the claim]
This is given by Fact~\ref{obs:hyperbolic containment condition}.
\end{proof}

\begin{claim}{4}
\label{item:holonomy computation}
Each singularity of $D\hat{M}$ is uniformisable in $\sph^\dimd$ (Proposition~\ref{prop:uniformization}).
\end{claim}

\begin{proof}[Proof of the claim]
Fix a corner $\corner$ of $D\mb{M}$ and take notation from Section~\ref{sec:holonomy computation} and Lemma~\ref{lem:holonomy computation}.

By Lemma~\ref{lem:uniformization} and the last point of Lemma~\ref{lem:holonomy computation}, we only need to check that $g(\gamma)=r_\theta b_\nu r_\theta b_{\nu'}^{-1}$ is non-trivial (this is an elementary computation) and has trace at least $2$, which is exactly the assumption of  Theorem~\ref{thm:hyperbolic doubles} since this trace is
\[
 t=\cos(\theta)^2\left(\frac{\nu}{\nu'}+\frac{\nu'}{\nu}\right)-\sin(\theta)^2\left(\nu'\nu+\frac1{\nu'\nu}\right).\qedhere
\]
\end{proof}

At this point we can use Theorem~\ref{thm:cvx} to say that $DM$ is convex.
To be able to use Theorem~\ref{thm:addendum1}, we need to check that:

\begin{claim}{5}
\label{item:holonomy computation2}
 Each singularity of $D\mb{M}$ is special (Definition~\ref{def:special tubes}).
 \end{claim}
 
 and 
 
 \begin{claim}{6}
\label{item:proper convexity} 
$DM$ is \emph{properly} convex.
\end{claim}

\begin{proof}[Proof of Claim \ref{item:holonomy computation2}]
Take notation from the above proof of \eqref{item:holonomy computation}.
We have
\[
 \rho(\gamma)=
 \begin{pmatrix}
  \left(\frac{\nu}{\nu'}\right)^{-\frac{2}{\dimd+1}}& \\
  & \left(\frac{\nu}{\nu'}\right)^{\frac{\dimd-1}{\dimd+1}}g(\gamma) 
 \end{pmatrix}.
\]
According to Definition~\ref{def:special tubes} (more precisely \eqref{def:special tubes}), we must show that  $\left(\frac{\nu}{\nu'}\right)^{2}-\left(\frac{\nu}{\nu'}\right)t+1$ is positive.
We compute:
\[
 \left(\frac{\nu}{\nu'}\right)^{2}-\left(\frac{\nu}{\nu'}\right)t+1=\sin(\theta)^2\left( \left(\frac{\nu}{\nu'}\right)^{2}+1+ \nu'{}^2 + \nu^{-2} \right)>0.\qedhere
\]
\end{proof}

To prove Claim \ref{item:proper convexity} we will need the following.

\begin{fact}\label{fact:irred}
 Any compact convex hyperbolic bc-manifold $\mb{X}=(\mc{C}\subset\mb{H}^\dimd)/\Gamma$ of dimension at least $3$ has irreducible holonomy (no invariant proper subspace).
\end{fact}

\begin{proof}[Proof of the fact]
 It is classical that $\Gamma$ is irreducible if and only if the convex hull in $\H^\dimd$ of its limit set has non-empty interior (see for instance \cite[Prop.\,2.5]{EeERfH+}).
 
As $\Gamma$ acts cocompactly on the closed convex subset $\mc{C}\subset\H^\dimd$, its limit set is $\partial\mc{C}\cap\partial\H^\dimd$.
Moreover, the closure of $\mc{C}$ (in $\sph^\dimd$) is the convex hull of its extremal points.
None of these extremal points is in $\partial\mc{C}\cap \H^\dimd$ since this set is a union of walls and corner (the corners have positive dimension since $\dimd>2$), so they are all in $\partial\mc{C}\cap\partial\H^\dimd$, the limit set of $\Gamma$, whose convex hull is therefore $\mc{C}$, which has non-empty interior.
\end{proof}

\begin{proof}[Proof of Claim \ref{item:proper convexity}]
Suppose $\dimd>2$, then by Fact~\ref{fact:irred} $\pi_1(M)$ acts irreducibly on $\sph^\dimd$, which immediately implies, by standard arguments, that $DM$, which is $\pi_1(M)$-invariant, is properly convex (see Definition~\ref{def:prop convex}).
\end{proof}

%

At this point we can apply Theorem~\ref{thm:addendum1}.

Finally, Theorem~\ref{thm:addendum2}, Corollary~\ref{cor:extremal points} and Fact~\ref{obs:smooth} ensure that every point of $\partial\tilde N$ outside closures of walls is extremal and $\cal C^1$, provided that:

\begin{claim}{7}
\label{item:hypb double reg1} 
Every point of $\partial\tilde M\cap\partial\H^\dimd$ outside the closures of the walls is an extremal $\cal C^1$ point of $\partial \tilde M$.
\end{claim}

\begin{claim}{8}
\label{item:hypb double reg2} 
For every wall $\tilde W\subset\tilde M$, every point of $\partial\tilde W\cap \partial\H^\dimd$ outside closures of a corners is an extremal $\cal C^1$ point of the boundary of the union of $\tilde M$ with any bulged reflection of $\tilde M$ across $\tilde W$.
\end{claim}

The proofs of these two points are essentially the same.

\begin{proof}[Proof of Claim \ref{item:hypb double reg1}]
 Let $x\in\partial\tilde M\cap\partial\H^\dimd$ be outside the closures of the walls.
 This point is extremal for $\tilde M$ because it is extremal for $\H^\dimd$, which contains $\tilde M$.
 
 Pick $p\in \tilde M$.
 Since $\mb{M}$ is compact, the fact that $x$ does not lie in the closure of a wall says that the projection of the ray $[p,x)$ comes back infinitely often in a compact subset of the interior of $M$.
 Therefore there exists $r>0$ and $p_n\in[p,x)$ tending to $x$ such that the ball of $\H^\dimd$ of radius $r$ around all $p_n$ is contained in $\tilde M$.
 This classically implies that $\partial\tilde M$ is $\cal C^1$ (otherwise there would be a ray $[q,x)\subset\H^\dimd$ outside of $\tilde M$, but then by basic hyperbolic geometry  this ray would eventually get arbitrarily close to $[p,x)$).
\end{proof}

\begin{proof}[Proof of Claim \ref{item:hypb double reg2}]
 Consider a bulged reflection $\tilde M'$ of $\tilde M$ across $\tilde W$.
 One can stretch $\H^\dimd$ to obtain a $\pi_1(W)$-invariant ellipsoid $\cal E$ that contains $\tilde M\cup \tilde M'$, and such that $\cal E\cap\sph^{\dimd-1}=\H^{\dimd-1}$.
 
 Let $x\in\partial\tilde W\cap\partial\H^{\dimd-1}$ be outside the closures of corners.
 This point is extremal for $\tilde M$ because it is extremal for $\cal E$, which contains $\tilde M\cup \tilde M'$.
 
 Pick $p\in \tilde W$.
 Since $W$ is compact, the fact that $x$ does not lie in the closure of a corner says that the projection of the ray $[p,x)$ comes back infinitely often in a compact subset of the interior of $W$.
 Therefore there exists $r>0$ and $p_n\in[p,x)$ tending to $x$ such that the ball of $\cal E$ of radius $r$ around all $p_n$ is contained in $\tilde M\cup \tilde M'$.
 This classically implies that $\partial\tilde M\cup \tilde M'$ is $\cal C^1$.
\end{proof}

This concludes the proof of Theorem \ref{thm:hyperbolic doubles}.
\end{proof}

\begin{rk}[The case $d=2$]
Note that the proof given above holds also in dimension $d=2$ except in Claim \ref{item:proper convexity}. However, the conclusion of the claim still holds (by an elementary argument). 
\end{rk}

\subsection{The proofs of the main theorems}

\begin{proof}[Proof of Theorems~\ref{thm:main1} and \ref{thm:main2}]
Fix $k=4$. Let $M$ be the convex compact hyperbolic $d$-manifold with totally geodesic boundary and corners with angles $\pi/4$ provided by Theorem \ref{thm:main7}, whose graph of the boundary is bipartite.
Let us consider a coloring of the walls of $M$ in yellow and purple such that any two adjacent walls have different colors.

For $j=1,2$, consider $\nu_j$ such that 
\[
 \frac1{\sqrt2}\left( \nu_j^2 +\nu_j^{-2} \right) - \frac1{\sqrt2}\left( 1+1 \right) = 2 + (2-j).
\]

By Theorem~\ref{thm:hyperbolic doubles}, by giving parameters $\nu_j^{(\dimd+1)/2}$ and $\nu_j^{-(\dimd+1)/2}$ to yellow and purple walls, we get a compact convex projective manifold with totally geodesic boundary $N_j$ without corners, with complete walls (Definition~\ref{def:Cplteness}), and without ghost strata (Definition~\ref{def:ghost}).
(Note that $N_1$ and $N_2$ are homeomorphic.)

Let $\tilde N_j\subset\sph^\dimd$ be the universal cover of $N_j$. Then:
\begin{itemize}
\item{The description stated in Theorem~\ref{thm:main2} of the walls of $\tilde N_j$ (the lifts of the walls of $N_j$) as cones with stabilizer of  the form $\pi_1(\corner)\times \Z$ with $\corner\subset M$ a corner, is given in Theorem~\ref{thm:hyperbolic doubles}.}
\item{Using again Theorem~\ref{thm:hyperbolic doubles}, every point of $\partial N_j$ which is not in the closure of a wall is extremal and $\cal C^1$.}
\end{itemize}

Let $\Gamma_j$ be the fundamental group of $N_j$.

Let $\Lambda_j$ be the full orbital limit set of $\Gamma_j$ in $\partial\Omega_j$.

\begin{itemize}
\item{By Theorem~\ref{thm:relhypb}, $\Gamma_j$ is relatively hyperbolic with respect to the fundamental groups of the walls of $N_j$, \ie the stabilizers of the walls of $\tilde N_j$.}
\end{itemize}

{\bf Case $j=1$}. If $j=1$, then:

\begin{itemize}
\item{Each wall has a polar that satisfies the Containment Condition by Theorem~\ref{thm:hyperbolic doubles}.}
\item{By Theorem~\ref{thm:cvx}, the projective structure on $N_1$ extends to a convex projective structure on the double $N=DN_1$, which is actually \emph{properly convex} as $\dimd\geq3$ by Fact~\ref{fact:irred}.}
\item{Let $\Omega_1\subset\sph^\dimd$ be the universal cover of $N$, which is divided by $\Gamma:=\pi_1(N)$ since $N$ is a closed manifold. By Proposition~\ref{prop:cvxcocpct}, $\Gamma_1$ acts convex-cocompactly on $\Omega_1$.}
\item{Let us call ``wall'' the image in $N$ of the walls of $N_1$, and the lifts of these in $\Omega_1$.
By the description of the walls of $\tilde N_1$ and their stabilizers, the walls of $\Omega_1$ form a $\Gamma$-invariant collection of properly embedded cones with pairwise disjoint closures, whose stabilizers have the form $\pi_1(\corner)\times \Z$ where $\corner\subset M$ is a corner.}
\item{By Theorem~\ref{thm:addendum2}, Corollary~\ref{cor:extremal points} and Fact~\ref{obs:smooth}, every point of $\partial \Omega_1$ which is not in the closure of a wall is extremal and $\cal C^1$.}
\item{By Fact~\ref{fact:Teddyrelhypb}, $\Gamma$ is relatively hyperbolic with respect to the stabilizers of the walls, \ie the properly embedded cones.
This concludes the proof of Theorem~\ref{thm:main1}.}
\item{By convex-cocompactness and \cite[Lem.\,4.1]{DGK17}, $\Lambda_1=\partial\tilde M_1\smallsetminus \tilde M_1$, so that $\mc{CH}(\Lambda_1)\smallsetminus\Lambda_1=\tilde M_1$.}
\end{itemize}

This finishes the proof of Theorem~\ref{thm:main2}.

{\bf Case $j=2$}. Suppose $j=2$. Let us check that $\Lambda_2=\partial\tilde M_2\smallsetminus \tilde M_2$.

The inclusion $\Lambda_2\subset\partial\tilde M_2\smallsetminus \tilde M_2$ is clear, let us prove the other one.

Fix $p\in\Omega_2$.
Let $x\in\partial\tilde M_2$ which is not in the closure of any wall.
As shown in the proof of Theorem~\ref{thm:relhypb}, the image of $[p,x)$ in $M_2$ comes back infinitely often in a compact subset of the complement of the walls, in other words there exists $x_n\in[p,x)$ going to $x$ and $g_n\in\Gamma_2$ such that $g_nx_n$ stays in a compact subset of $\Omega_2$.
Consider $R>0$ large enough so that the closed Hilbert ball $B_n:=B_R(g_nx_n)$ contains $p$ for every $n$.
Since $x$ is extremal, $g_n^{-1}B_n=B_R(x_n)$ converge to $x$ (see for instance \cite[Fact 2.1.10]{ThesePL}).
In particular, $x\in \overline{\Gamma_2\cdot p}$.

As points in the relative boundary of walls can be approached by points outside closed walls, $\partial\tilde M_2\smallsetminus \tilde M_2$ is contained in $\overline{\Gamma_2p}$, and hence in the full orbital limit set.
\end{proof}

Theorem \ref{thm:mainA} is a particular case of Theorem~\ref{thm:main1}.

We prove Theorem \ref{thm:mainB}.

\begin{proof}[Proof of Theorem \ref{thm:main3}]
Take $\ep$ from Fact~\ref{fact:ChoiSeries}.
Consider $a_i,b_i,c_i$, $i=1,\dots,n$, as in Theorem~\ref{thm:main3}, \ie such that $1<a_i<e^{\ep/2}$, $b_i>3$ and $1<c_i<-2+b_i^2/2$.

Set $\ell_i:=2\ln(a_i)<\ep$, and consider the hyperbolic metric on $M$ with totally geodesic and corners, such that the corners are the $\alpha_i$'s, with length $\ell_i$ and angle $\theta_i<\pi/4$.

The fact that $c_i<-2+b_i^2/2$ and $\theta_i<\pi/4$ guarantees that 
$$c_i+c_i^{-1}<\frac12(b_i^2+b_i^{-2})-1<\cos^2\theta_i(b_i^2+b_i^{-2})-2\sin^2\theta_i,$$
and, hence, that
$$2<\frac{\cos^2\theta_i(b_i^2+b_i^{-2})-(c_i+c_i^{-1})}{\sin^2\theta_i}.$$
Thus there exists a unique $d_i>1$ such that
$$d_i^2+d_i^{-2}<\frac{\cos^2\theta_i(b_i^2+b_i^{-2})-(c_i+c_i^{-1})}{\sin^2\theta_i},$$
and 
$$t_i:=\cos^2\theta_i(b_i^2+b_i^{-2})-\sin^2\theta_i(d_i^2+d_i^{-2})=c_i+c_i^{-1}>2.$$

Set $\mu_i:=\sqrt{b_id_i}$ and $\mu_i':=\sqrt{d_i/b_i}$.

Note that $\mu_i/\mu_i'=b_i$ and $\mu_i\mu_i'=d_i$ and $c_i=t_i/2+\sqrt{t_i^2/4 -1}$.

We get the desired projective structure on $DM\smallsetminus U$ by applying the bulging procedure from Theorem~\ref{thm:hyperbolic doubles} with parameters $\mu_i$ and $\mu_i'$ (recall that $\dimd=3$).
\end{proof}

\subsection{Commensurability}
Lastly, we show that our construction is flexible enough to produce an abundance of examples of inequivalent non-symmetric non-strictly divisible convex domains.

To begin with let us recall the definition of commensurability.

\begin{defi}[Commensurable]
Two groups $G,G'$ are {\em commensurable} if they contain isomorphic finite index subgroups $H<G,H'<G'$.
\end{defi}

We prove

\begin{thm}
\label{thm:non-commensurable}
Suppose $d\geq5$.
Let $\theta=\pi/n,\theta'=\pi/n'$ be such that the number fields $F=\mb{Q}(\cos(\theta),\sin(\theta)),F'=\mb{Q}(\cos(\theta'),\sin(\theta'))$ are different. Let $\tau\in F,\tau'\in F'$ be elements satisfying the assumptions of Lemma \ref{lem:number field}. Let $\Gamma,\Gamma'<{\rm PSL}_{d+1}(\mb{R})$ be the groups as in Theorem \ref{thm:main1} dividing the properly convex sets $\Omega,\Omega'\subset\mb{RP}^d$ which are relatively hyperbolic with respect to the collections of subgroups $\{\Theta_j\times\mb{Z}\}_{j\in J},\{\Theta_j'\times\mb{Z}\}_{j\in J'}$ with $\Theta_j,\Theta_j'$ finite index subgroups of ${\rm SO}(q)_\mc{O},{\rm SO}(q')_{\mc{O}'}$ where $q=x_1+\cdots+x_{d-2}^2-\tau y^2$ and $q'=x_1+\cdots+x_{d-2}^2-\tau' y^2$. Then $\Gamma,\Gamma'$ are not commensurable.
\end{thm}

\begin{proof}
Suppose that $\Gamma,\Gamma'$ are commensurable. This means that there exists a group $G$ that embeds in both $\Gamma,\Gamma'$ as a subgroup of finite index. 
Let us say $G<\Gamma$ and $\iota:G\hookrightarrow \Gamma'$ is an embedding.

As $\Gamma,\Gamma'$ are relatively hyperbolic with respect to the collections of subgroups $\{\Theta_j\times\mb{Z}\}_{j\in J},\{\Theta_j'\times\mb{Z}\}_{j\in J'}$ it follows that $G$ is relatively hyperbolic with respect to $\{S_{j,\alpha}:=\alpha(\Theta_j\times\mb{Z})\alpha^{-1}\cap G\}_{j\in J,\alpha\in\Gamma}$ and also with respect to $\{S_{j,\alpha}':=\iota^{-1}(\alpha(\Theta_j'\times\mb{Z})\alpha^{-1})\}_{j\in J',\alpha\in\Gamma'}$. Notice that each $S_{j,\alpha},\iota(S_{j,\alpha}')$ is a finite index subgroup of $\alpha(\Theta_j\times\mb{Z})\alpha^{-1},\alpha(\Theta_j'\times\mb{Z})\alpha^{-1}$.

By work of Islam and Zimmer \cite[Th.\,1.6]{IZ22relhypb}, we have the following: Let $L_{j,\alpha}'$ be the (full orbital) limit set of the group $S_{j,\alpha}'$ in $\Omega$. Then: 
\begin{itemize}
\item{$L_{j,\alpha}'\subset\partial\Omega$ is closed and $\mc{CH}(L_{j,\alpha}')\smallsetminus L_{j,\alpha}'$ is a convex subset of $\Omega$}
\item{The group $S_{j,\alpha}'$ acts cocompactly on $\mc{CH}(L_{j,\alpha}')\smallsetminus L_{j,\alpha}'\subset\Omega$.}
\item{There is an equivariant homeomorphism between the Bowditch boundary of $(G,\{S_{j,\alpha}'\}_{j\in J',\alpha\in\Gamma'})$ and the space obtained from $\partial\Omega$ by collapsing each $L_{j,\alpha}'$ to a point. In particular, the sets $L_{j,\alpha}'\subset\partial\Omega$ are pairwise disjoint.}
\item{If $[a,b]\subset\partial\Omega$ is a non-trivial line segment, then $[a,b]\subset L_{j,\alpha}'$ for some $j\in J'$ and $\alpha\in\Gamma'$.}
\end{itemize}

Consider one of the properly embedded cones $\mc{CH}(H,p)\cap\Omega \in\mc{C}$ on which acts cocompactly $S_{i,\beta}=\beta(\Theta_k\times\mb{Z})\beta^{-1}\cap G$ for some $k\in J$ and $\beta\in\Gamma$; let $L_{i,\beta}$ be the full orbital limit set of $S_{i,\beta}$, \ie the relative boundary of the cone.

For every $q\in\partial H$ the line segment $[q,p]\subset\partial\Omega$ is contained in some $L_{j,\alpha}'$. As all segments intersect in $p$ we conclude that they are all contained in the same $L_{j,\alpha}'$. Therefore $L_{i,\beta}\subset L'_{j,\alpha}$ and $\mc{CH}(H,p)\subset\mc{CH}(L_{j,\alpha}')$. 

Note that $S_{i,\beta}$ preserves $L'_{j,\alpha}$.
Indeed if $\gamma\in S_{i,\beta}$ then $L'_{j,\iota(\gamma)\alpha}=\gamma L'_{j,\alpha}$ contains $L_{i,\beta}$ hence intersects $L'_{j,\alpha}$ and hence is equal to it.

Since $S'_{j,\alpha}$ is a finite index subgroup of $\Stab_G(L'_{j,\alpha})$, this means that $S_{i,\beta}\cap S_{j,\alpha}'$ is a finite index subgroup of $S_{i,\alpha}$.
In particular its cohomological dimension is the same as that of $S_{i,\beta}$, which is $\dimd-1$, which is also the cohomological dimension of $S'_{j,\alpha}$.
By a standard argument (see \cite{cohomgpdiscret}), this implies that $S_{i,\beta}\cap S_{j,\alpha}'$ is also a finite index subgroup of $S'_{j,\alpha}$.

In particular, $\Theta_j'\times\mb{Z}$ and $\Theta_i\times\mb{Z}$ are commensurable as they contain the isomorphic finite index subgroups $S_{i,\beta}\cap S_{j,\alpha}'$ and $\iota(S_{i,\beta}\cap S_{j,\alpha}')$.

Let $K$ be a group such that every finite index subgroup of $K$ has trivial center. For example, all $\Theta_i,\Theta_j'$ have such property. Consider a finite index subgroup $S<K\times\mb{Z}$. Then the center of $S$ is $Z(S)=S\cap(\{1\}\times\mb{Z})$ and we have
\[
S/Z(S)=S/\left(S\cap(\{1\}\times\mb{Z})\right)\simeq \pi(S)<K
\]
where $\pi:K\times\mb{Z}\to K$ is the projection onto the first factor.

Let us prove this claim: Consider $\pi(S)<K$. It is a finite index subgroup and, hence, it has trivial center by assumption. If $(\alpha,t)\in S$ is in the center of $S$, then $\alpha$ is in the center of $\pi(S)$ which is trivial. Vice versa it is immediate to see that $\{1\}\times\mb{Z}$ is in the center of $K\times\mb{Z}$.

We apply this fact to $S:=S_{i,\beta}\cap S_{j,\alpha}'<\Theta_i\times\mb{Z}$ and $\iota(S)<\Theta_j'\times\mb{Z}$: Let $\pi_i,\pi_j$ be the projections of $\Theta_i\times\mb{Z},\Theta_j'\times\mb{Z}$ onto the first factors. We have
\[
 \pi_j(S)\simeq S/Z(S)\simeq \pi_i(\iota(S)).
\]

Thus $\Theta_j'$ is commensurable with $\Theta_i$ as they contain the isomorphic finite index subgroups $\pi_j(S)$ and $\pi_i(\iota(S))$.

In order to conclude we will use a commensurability invariant of hyperbolic manifolds, namely the invariant trace field: 

\begin{defi}[Trace Field]
Let $\Gamma<{\rm SO}_0(d-2,1)$ be a lattice. The {\em invariant trace field} of $\Gamma$ is
\[
k(\Gamma):=\mb{Q}\left(\{{\rm tr}({\rm Ad}(\gamma))\left|\;\gamma\in\Gamma\right.\}\right)
\]
where ${\rm Ad}:{\rm SO}_0(d-2,1)\to\mathfrak{so}(d-2,1)$ is the adjoint representation.

By Mostow Rigidity (see \cite[Ch.\,15]{WitteMorris}), if $d\ge 5$, the invariant trace field $k(\Gamma)$ is an invariant of the abstract group $\Gamma$ and does not depend on the particular embedding $\Gamma<{\rm SO}_0(d-2,1)$ as a lattice (which is unique up to conjugation).
\end{defi}

Vinberg \cite{Vin71b} shows that the trace field $k(\Gamma)$ is an invariant of the commensurability class of $\Gamma$:

\begin{fact}[{Vinberg \cite{Vin71b}}]
Let $\Theta<{\rm SO}_0(d-2,1)$ be a lattice. If $\Theta_0<\Theta$ is a finite index subgroup, then $k(\Theta_0)=k(\Theta)$.
\end{fact}

The computation of the trace field in our case is not difficult:

\begin{fact}[{Prasad--Rapinchuck \cite[Lem.\,2.6]{PR09}}]
Let $F$ be a number field with ring of integers $\mc{O}$. Let 
\[
q=x_1^2+\cdots+x_{d-2}^2-\tau y^2
\]
be a quadratic form where $\tau\in\mc{O}$ satisfies the assumptions of Lemma \ref{lem:number field}. Then 
\[
k({\rm SO}(q)_{\mc{O}})=F.
\]
\end{fact}

Since $\Theta_j',\Theta_i$ are commensurable we must have $k(\Theta_j')=k(\Theta_i)$. However, $\Theta_j',\Theta_i$ are finite index subgroups of ${\rm SO}(q)_\mc{O},{\rm SO}(q')_{\mc{O}'}$ respectively where $q=x_1+\dots+x_{d-2}^2-\tau y^2$ and $q'=x_1+\dots+x_{d-2}^2-\tau' y^2$ and $\mc{O},\mc{O}'$ are the ring of integers of $F,F'$. By the above results, we get $k(\Theta_j')=F'$ and $k(\Theta_i)=F$. As $F\neq F'$, this provides a contradiction and finishes the proof.
\end{proof}

\end{document}